%% file: Khovanov_homology_of_links_in_rp3.tex
\documentclass{amsart}
\usepackage{amssymb,amsmath}
\usepackage{graphicx}
\usepackage{subfigure}
\usepackage{upgreek}
\usepackage{tikz}
\usepackage{tikz-cd}
\usepackage{hyperref}
\usepackage{mathtools}
\usetikzlibrary{matrix}
\usepackage{cmap}
\usepackage[T1]{fontenc}

\usepackage[utf8]{inputenc}
\usetikzlibrary{matrix}
\usepackage[all]{xy}
\usepackage{optparams,eurosym}
\usepackage{amsthm}
\usepackage{amsmath,amssymb}
\usepackage{dbnsymb}
\newcommand{\rp}[1]{\mathbb{RP}^{#1}}
\newcommand{\brac}[1]{\langle{#1}\rangle}
\newcommand{\f}{\mathbb{F}_2}
\newcommand{\un}{v_+}
\newcommand{\coun}{v_-}

\newtheorem{theorem}{Theorem}[section]

\newtheorem{lemma}[theorem]{Lemma}
\newtheorem{proposition}[theorem]{Proposition}
\newtheorem{definition}[theorem]{Definition}
\newtheorem{non-theorem}{Non-Theorem}

\newtheorem*{remark}{Remark}

\numberwithin{equation}{subsubsection}

\begin{document}

	\title{Khovanov-type homologies of null homologous links in $\mathbb{RP}^3$}
	\author{Daren Chen}
	
		\maketitle

		\begin{abstract} Let $L$ be a null homologous link in $\mathbb{RP}^3$. We define Khovanov-type homologies of $L$ which depend on an extra input $\alpha = (V_0,V_1,f,g)$ consisting of two graded vectors spaces and two maps between them. With some specific choice of $\alpha = \alpha_{APS}$, we recover the categorification of the Kauffman bracket due to Asaeda-Przytycki-Sikora. With another choice of $\alpha = \alpha_{HF}$, we construct a spectral sequence from our theory converging to the Heegaard Floer homology of the even branched double cover of $\mathbb{RP}^3$.
	\end{abstract}
		
	\section{Introduction}

	In \cite{MR2113902}, Asaeda, Przytycki and Sikora extended the original construction of Khovanov homology in \cite{MR1740682} to links in interval-bundles over surfaces, categorifying the Kauffman bracket. In particular, their construction gives a homology theory for links in $\rp{3}$, by viewing $\rp{3} \backslash \{*\}$ as the twisted $I$-bundle over $\rp{2}$. For $I$-bundles over non-orientable surfaces like $\rp{2}$, their theory was defined with $\mathbb{F}_2$ coefficients. In \cite{MR3189291}, Gabrov{\v{s}}ek extended the definition for links in $\rp{3}$ to $\mathbb{Z}$ coefficients by fixing a sign convention. 
	
	In the first half of this paper, we are going to generalize the construction in  \cite{MR2113902}  to get a family of Khovanov-type link homologies $\widetilde{\mathit{Kh}}^{\alpha}(L)$ for null homologous links in $\rp{3}$ with  $\mathbb{F}_2$ coefficients. Our homology theory depends on an extra input $\alpha = (V_0,V_1,f,g)$ called a \textbf{dyad},  consisting of two graded vector spaces $V_0$, $V_1$ and maps $f:V_0\rightarrow V_1$, $g:V_1 \rightarrow V_0$ between them such that $f \circ g =0, \,\,g\circ f =0$. 
	
	\begin{theorem}
		For each dyad $\alpha$, the homology $\widetilde{\mathit{Kh}}^{\alpha}(L)$ is an invariant of null homologous links in $\rp{3}$.
		
	\end{theorem}
	
	With a specific choice of the dyad $\alpha_{APS} = (V,V,0,0)$ where $V=\langle\un,\coun \rangle$, we recover a reduced version of the homology defined in \cite{MR2113902} for null homologous links in $\rp{3}$. The novelty of our construction is that we associate to each smoothing $L_s$ of the link $L$ an extra parameter $e_s(P) \in \left\{0,1\right\}$, and the vector space we associate to the smoothing $L_s$ in the chain complex $\widetilde{\mathit{CKh}}^{\alpha}(L)$ will be different depending on the values of $e_s(P)$. Here $P$ is a point in the complement of the link projection in $\rp{2}$, and $e_s(P)$ counts the number of circles mod $2$ in the smoothing $L_s$ which encircles $P$. The Euler characteristic of the homology theory $\widetilde{\mathit{Kh}}^{\alpha}(L)$ is a linear combination of the even and odd Jones polynomials of $L$, with coefficients given by the graded dimension of $V_0$ and $V_1$ respectively. We will also introduce an unreduced version $\textit{Kh}^{\alpha}(L)$ of the homology, and discuss briefly what happens to other links in $\rp{3}$ which are non-trivial in $H_1(\rp{3},\mathbb{Z})$.

	In the second half of the paper, we will relate the Heegaard Floer homology of a branched double cover of $\rp{3}$ over a null homologous link $L$ to the Khovanov-type homology $\widetilde{\mathit{Kh}}^{\alpha_{\textit{HF}}} (m(L))$ for another choice of the dyad $\alpha_{HF} = (W,\bar{V},f,g)$. In this case, $W = \langle a,b,c,d \rangle$, $\bar{V} = \langle \bar{v}_+,\bar{v}_-\rangle$, and
	\begin{align*}
		f(a)=f(d)=0,\,\,\,\,\,\,\, &f(b)=f(c)=\overline{v}_- \\
		g(\overline{v}_-) = 0,\,\,\,\,\,\,&g(\overline{v}_+)=b+c.
	\end{align*}
   For a link $L$ in $S^3$, we can form the branched double cover $\Sigma(S^3,L)$ of $S^3$. In \cite{MR2141852}, Ozsv{\'a}th and Szab{\'o} defined a spectral sequence which converges to $\widehat{HF}(\Sigma(S^3,L))$ with $\mathbb{F}_2$ coefficients.  The $E^2$ term of this spectral sequence gives the reduced Khovanov homology of the mirror link $m(L)$ of $L$. We consider a extension of this construction for null homologous links in $\rp{3}$, and obtain the following result.
	
	\begin{theorem}
			Let $L$ be a null homologous link in $\rp{3}$. There is a spectral sequence whose $E^2$ term consists of the Khovanov-type homology $\widetilde{\mathit{Kh}}^{\alpha_{\textit{HF}}}(m(L))$ of the mirror of $L$ with the dyad $\alpha_{HF} = (W,\overline{V},f,g)$, which converges to the Heegaard Floer homology $\widehat{HF}(\Sigma_0(\rp{3},L))$ of the even branched double cover $\Sigma_0(\rp{3},L)$ of $\rp{3}$. 
	\end{theorem}
	
	We only consider null homologous links in $\rp{3}$ because the branched double cover $\Sigma(\rp{3},L)$ only exists for null homologous link $L$. What's more, there are two branched  double covers for each null homologous link $L$, and we will make a specific choice, called the even branched cover $\Sigma_0(\rp{3},L)$. The construction of the spectral sequence is essentially the same as the one in \cite{MR2141852}, with some particular treatment of the cobordism corresponding to $1\rightarrow1$ bifurcation, which is special for link projections to $\rp{2}$. Another difference is that when defining the spectral sequence, we need to use both of the two branched double covers. For each smoothing $L_s$ of $L$, we will use the even branched double cover $\Sigma_0(\rp{3},L_s)$ if the extra parameter $e_s(P)$ we introduced earlier equals to $0$, and the odd branched double cover $\Sigma_1(\rp{3},L_s)$ if $e_s(P)=1$. This is the reason we want to introduce the extra input $\alpha = (V_0,V_1,f,g)$ into our homology theory, where $V_0,V_1$ basically correspond to the two different branched double covers, and $f,g$ correspond to the induced maps on $\widehat{HF}$ between the two branched double covers by performing surgeries associated to $1\rightarrow1$ bifurcations.

	 Here is the organization of this paper. In Section 2, we define the Khovanov-type homology $\widetilde{Kh}^{\alpha}(L)$ combinatorially. For each null homologous link projection $L$, an arbitrary point $P$ in the complement of $L$ in $\rp{2}$ and each dyad $\alpha$, we define a chain complex $\widetilde{\mathit{CKh}}^{P,\alpha}(L)$ in Section $\ref{sec1.1}$. The differential is presented in Section $\ref{sec1.2}$. In Section $\ref{sec1.3}$, we fix a canonical choice of the point $P$, and show that the homology $\widetilde{\mathit{Kh}}^{\alpha}(L)$ is an invariant of null homologous links in $\rp{3}$. In Section $\ref{sec1.4}$, we express the Euler characteristic of $\widetilde{\mathit{Kh}}^{\alpha}(L)$ through skein relations. Section $\ref{sec1.5}$ briefly describes the unreduced versions $\mathit{CKh}^{\alpha}(L)$ and $\mathit{Kh}^{\alpha}(L)$. We give an example calculation of $\widetilde{\mathit{Kh}}^{\alpha}(L)$ for a specific link projection $L$ with some specific choices of $\alpha$ in Section $\ref{sec1.6}$. Finally in Section $\ref{section2.7}$,  we discuss the situation for other links in $\rp{3}$. In Section 3, we give the  construction of the spectral sequence converging to $\widehat{HF}(\Sigma_0(\rp{3},L))$. We discuss the branched double cover $\Sigma(\rp{3},L)$ of $\rp{3}$ in Section $\ref{sec2.1}$. Then we quickly review the construction in \cite{MR2141852} in Section $\ref{sec2.2}$. In Section $\ref{sec2.3}$, we compute the $E^2$ term of our spectral sequence, and show it equals to $\widetilde{Kh}^{\alpha_{HF}}(m(L))$. 
	 
	 \vspace*{3mm}
	 \textbf{Acknowledgments:} This work is partially supported by NSF grant number DMS-2003488. The author wishes to thank his advisor Ciprian Manolescu, who introduced the author to this topic, shared many insightful points of view and helped generously in writing up this paper.

	\section{Definition of the homology}
 An oriented link $K$ in $\rp{3}$ is \textbf{null homologous} if $[K]=0$ in $H_1(\rp{3},\mathbb{Z})$. Note that a null homologous link could have an even number of components which are non-trivial in $H_1(\rp{3},\mathbb{Z})$. Given a null homologous link $K$ in $\rp{3}$, we consider its projection $L$ to $\rp{2}$, by identifying $\rp{3}\backslash \{*\}$ with $\rp{2}\widetilde{\times} I$, the twisted $I$-bundle over $\rp{2}$.  We will associate a  Khovanov-type chain complex to the link projection $L$, and show its homology is an oriented link invariant for null homologous links in $\rp{3}$. 
 
 
  First we introduce some basic algebra notions. All vector spaces in this paper are over $\mathbb{F}_2$ unless stated otherwise. Let $V$ be the graded vector space spanned by $\un$ and $\coun$, with quantum gradings $qdeg( \un) = 1$ and $qdeg(\coun) = -1$. As in the usual definition of Khovanov homology, $V$ has the structure of a Frobenius algebra, with multiplication $m: V \otimes V \longmapsto V$ such that \[ m(\un\otimes\un)=\un, \,\, m(\un\otimes \coun)=m(\coun\otimes \un)=\coun, \,\, m(\coun\otimes \coun)=0,\] and comultiplication $\Delta: V \longmapsto V\otimes V$  such that \[\Delta(\un) = \un\otimes \coun + \coun\otimes \un,\,\,\,\Delta(\coun)=\coun\otimes \coun.\] Note that both $m$ and $\Delta$ change the quantum degree by $-1$.
 
 \begin{definition}
 	A \textbf{dyad} is a tuple $ \alpha = (V_0,V_1, f,g)$, where $V_0$ and $V_1$ are graded vector spaces, $f:V_0 \longmapsto V_1$ and $g:V_1 \longmapsto V_0$ are linear maps, both of quantum degree $-1$, and $f\circ g = 0,\,\, g\circ f=0$. The \textbf{dual dyad} $\alpha^*$ of a dyad $\alpha = (V_0,V_1, f,g)$ is $\alpha^*=(V_1,V_0,g,f)$, obtained by switching $V_1$ with $V_0$ and $g$ with $f$. 
 \end{definition}
 
  For each dyad $\alpha=(V_0,V_1, f,g)$, we give each $V_i$ a trivial right-bimodule structure over $V$, with multiplication $m: V_i \otimes V \longmapsto V_i$ such that \[m(y\otimes \un)=y, \,\,\,m(y \otimes \coun)=0,\,\,\, \forall y \in V_i, \,\,i=1,2,\]
 and comultiplication $\Delta: V_i \longmapsto V_i \otimes V$ such that \[\Delta (y) = y\otimes \coun,\,\,\, \forall y \in V_i, \,\,i=1,2.\]
 Again, both $m$ and $\Delta$ have quantum degree $-1$.
 

\subsection{The cube of resolutions}
\label{sec1.1}

Label the crossings of a null homologous link projection $L$ in $\rp{2}$ from $1$ to $n$. A \textbf{state} $s\in \{0,1\}^n$ is a choice of $0$ or $1$ for each crossing. Given a state $s$, we form a smoothing of $L$ according to the rule in Figure $\ref{fig:1}$. Then, we associate a vector space to each smoothing $L_s$ as follows.

\begin{figure}[t]
	\[
	{
		\fontsize{9pt}{11pt}\selectfont
		\def\svgscale{0.6}
		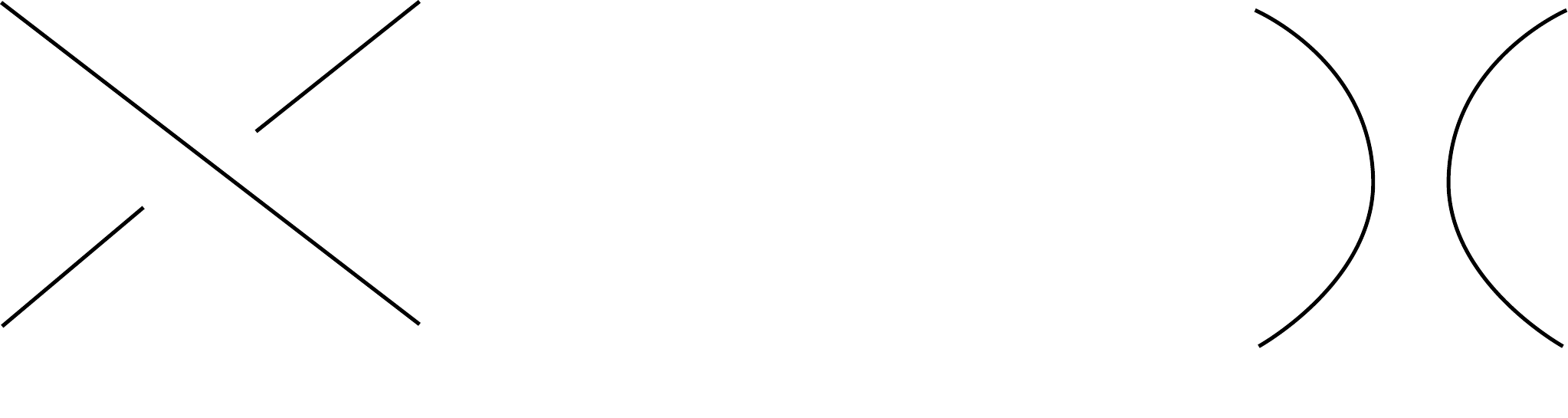
	}
	\]
	\caption{Smoothings}
	\label{fig:1}
\end{figure}


Define $L_s$ as the link diagram obtained from $L$ by smoothing according to the state $s$. For each $s$, the link $L_s$ is a disjoint union of $k_s$ embedded circles in $\rp{2}$. Now we show each circle in $L_s$ is trivial in $H_1(\rp{2},\mathbb{Z})$, hence divides $\rp{2}$ into a disk and a M{\"o}bius band.

\begin{lemma}
	\label{lem:1.1.1}
	For the link projection $L$ of a null homologous link $K$ in $\rp{3}$, every smoothing $L_s$ is null homologous as well. In particular, each circle $S^1$ in $L_s$ is trivial in $H_1(\rp{2},\mathbb{Z})$, and divides $\rp{2}$ into a disk and a M$\ddot{o}$bius band.
\end{lemma}

\begin{proof}  We draw $\rp{2}$ as a disk with half of the boundary identified with the other half in opposite direction. For the projection $L$ in $\rp{2}$, consider the number of intersection points of it with half of the boundary of the disk. $L$ represents the generator of $H_1(\rp{2},\mathbb{Z}) = \mathbb{Z}/2$ if the intersection number is odd, and is null homologous if the intersection number is even. Note that the intersection number is unchanged during the smoothing procedure, so $L_s$ has an even number of intersections for each state $s$, as we start with a null homologous link projection $L$. Hence the smoothing $L_s$ is null homologous for every $s$.

 Now we prove each circle in $L_s$ is null homologous by contradiction. Suppose there is an embedded circle in $L_s$ which represents the generator of $H_1(\rp{2},\mathbb{Z})$, then cutting $\rp{2}$ along this circle we obtain a disk $D^2$. As $L_s$ is a disjoint union of circle, there can not be another circle in $L_s$ representing the generator of $H_1(\rp{2},\mathbb{Z})$. Since $L_s$ is null homologous, we get a contradiction. Hence, each circle in $L_s$ is null homologous, and divides $\rp{2}$ into a disk and M{\"o}bius band.\end{proof}
 

Pick a point $P$ in the complement of the $L$ in $\rp{2}$, such that $P$ lies in the complement of each smoothing $L_s$ as well.
\begin{definition}
For each null homologous circle $S^1$ in $\rp{2}$, we say $P$ is \textbf{encircled} by $S^1$ if $P$ lies in the disk bound by $S^1$. Define the \textbf{encircling number} $e_{s}(P)$ as the number of circles in $L_s$ encircling $P$ mod $2$.
\end{definition}

We will associate different vector spaces $\widetilde{\mathit{CKh}}^{P,\alpha}_s(L)$ to $s$ depending on the value of $e_s(P)$. For a given dyad $\alpha=(V_0,V_1,f,g)$, define \[\widetilde{\mathit{CKh}}^{P,\alpha}_s(L) = V_{e_s(P)}\otimes V^{\otimes (k_s-1)},\]where $k_s$ is the total number of circles in the smoothing $L_s$. Now we apply the usual flattening of cube operation to define a chain complex $\widetilde{\mathit{CKh}}^{P,\alpha}_{\bullet}(L)$.
\begin{definition}
	\label{def1.1.3}
	The \textbf{Khovanov-type chain complex} $\widetilde{\mathit{CKh}}^{P,\alpha}_{\bullet}(L)$ of a null homologous link projection $L$ in $\rp{2}$, with a point $P$ in the complement of $L$ and a dyad $\alpha$, is 

\begin{equation}
\label{eq:1}
\widetilde{\mathit{CKh}}^{P,\alpha}_i(L) = \bigoplus_{ \substack{s\in\{0,1\}^n\\\#1(s)=i + n_-} } \widetilde{\mathit{CKh}}^{P,\alpha}_s(L)\{i + n_+-n_-\},
\end{equation}
where $\#1(s)$ is the number of $1$s in the state $s$, $n_+$ and $n_-$ are the numbers of positive and negative crossings in $L$ respectively, and $\widetilde{\mathit{CKh}}^{P,\alpha}_s(L)\{i+n_+-n_-\}$ is the vector space obtained from $\widetilde{\mathit{CKh}}^{P,\alpha}_s(L)$ with quantum degree shifted up by $i+n_+-n_-$. The homology of $\widetilde{\mathit{CKh}}^{P,\alpha}_{\bullet}(L)$ is denoted as $\widetilde{\mathit{Kh}}^{P,\alpha}_{\bullet}(L)$.

\end{definition}

Note that we make a shift in the homology degree by $-n_-$ as well by letting $\#1(s) = i+n_-$, as in the usual Khovanov homology. We will define the differential of this chain complex in the next subsection. 




\subsection{The differential}
\label{sec1.2}
 Choose a marked point $M$ on the link projection $L$. The differential map depends on the choice of $M$, but we will show later that the homology does not depend on it. As in usual Khovanov homology, the differential $d:\widetilde{\mathit{CKh}}^{P,\alpha}_i(L) \longmapsto \widetilde{\mathit{CKh}}^{P,\alpha}_{i+1}(L)$ is given by a summation of maps over edges in the cube of resolutions. Each edge corresponds to changing the smoothing from $0$ to $1$ at one crossing, by our definition of the homology degree $i = \#1(s)-n_-$. There are $3$ cases, as shown in Figure $\ref{fig:bifurcation}$.
\begin{enumerate}
	\item $2\rightarrow1$ bifurcation, where two circles in $L_s$ merge into a circle in $L_{s'}$;
	\item $1\rightarrow2$ bifurcation, where a circle in $L_s$ splits to two circles in $L_{s'}$;
	\item $1\rightarrow1$ bifurcation, where a circle in $L_s$ twists into a new circle in $L_{s'}$.

\end{enumerate}

\begin{figure}[t]
	\[
	{
		\fontsize{7pt}{9pt}\selectfont
		\def\svgwidth{5.5in}
		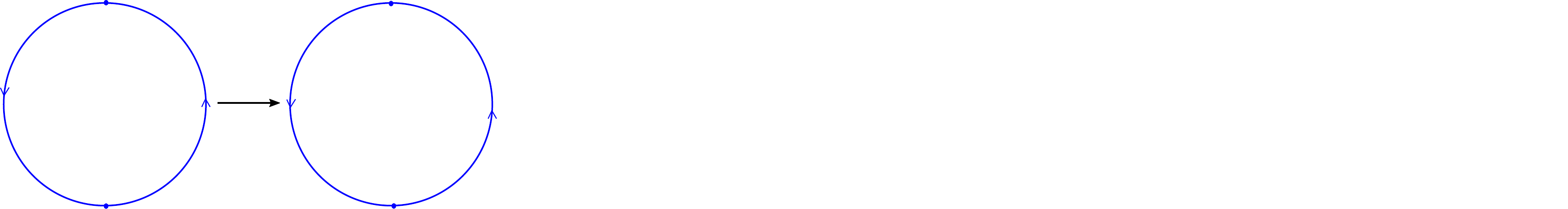
	}
	\]
	\caption{Bifurcations}
	\label{fig:bifurcation}
\end{figure}

The first two bifurcations are similar to the corresponding ones for link projections in $\mathbb{R}^2$, while the third one only appears in unoriantable surfaces, $e.g.$ $\mathbb{RP}^2$. 


First we study the change of $e_s(P)$ under different bifurcations, which will determine the domain and codomain of the chain map on each edge. Recall $e_s(P)$ is the number of circles in $L_s$ encircling $P$ mod $2$.


\begin{lemma}
	We have $e_{s'}(P) = e_{s}(P)$ for the $2\rightarrow1$ and $1\rightarrow2$ bifurcations, while $e_{s'}(P) = e_{s}(P) +1$ mod $2$ for the $1\rightarrow1$ bifurcation. 
\end{lemma}

\begin{proof} See Figure $\ref{fig:esp}$ for an illustration of the situation.

 Let us consider the $2\rightarrow1$ bifurcation first. For a circle $S^1$ not involved in the bifurcation, it is not changed in the process at all, so the encircling state of $P$ with it is not changed as well. So we only need to consider the relation of $P $ with the two merging circles. $P$ could be encircled by $0,1$ or $2$ of them. If $P$ is encircled by none of them, then $P$ is not encircled by the new circle as well, so $e_{s'}(P) = e_{s}(P)$. If $P$ is encircled by one of them, then $P$ is encircled by the new circle, so $e_{s'}(P) = e_{s}(P)$. If $P$ is encircled by two of them, then one disk lies in the other disk, and the disk bounded by the new circle is the complement of the smaller disk in the larger one. Therefore, $P$ is not encircled by the new circle, and $e_{s'}(P) = e_{s}(P)$.

The case for the $1\rightarrow2$ bifurcation is similar, except we reverse the arrow in the change. So we get the same conclusion that $e_{s'}(P) = e_{s}(P)$.

For the $1\rightarrow1$ bifurcation, the disk part and the M{\"o}bius band part of the involved circle are switched, so $e_{s'}(P) = e_{s}(P)+1$ mod $2$. 

\end{proof} 
\begin{figure}[t]
	\[
	{
		\fontsize{8pt}{10pt}\selectfont
		\def\svgwidth{5.5in}
		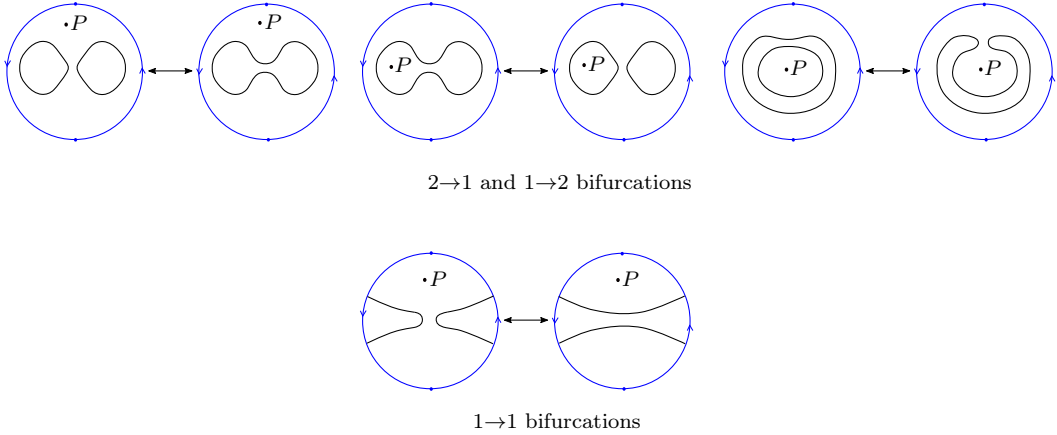
	}
	\]
	\caption{Changes of $e_s(P)$ under bifurcations}
	\label{fig:esp}
\end{figure}

Now we define the differential map   $d:\widetilde{\mathit{CKh}}^{P,\alpha}_s(L) \longmapsto \widetilde{\mathit{CKh}}^{P,\alpha}_{s'}(L)$ 
 in each case. Denote $i = {e_s(P)}$ for simplicity of notation. We identify each circle in $L_s$ with a factor in the tensor product $\widetilde{\mathit{CKh}}^{P,\alpha}_s(L)=V_i\otimes  V^{\otimes (k_{s}-1)}$, by assigning $V_i$ to the circle in $L_s$ with the marked point $M$, and $V$ to each of the other circles. For the $2\rightarrow1$ and $1\rightarrow2$ bifurcations, $d$ is identity on tensor factors corresponding to circles not involved, and we specify what $d$ does to the tensor factors corresponding to the involved circles as follows.
\begin{enumerate}
	\item $2\rightarrow1$ \textbf{bifurcation.}
			By the above lemma, $e_s(P) = e_{s'}({P})$. The number of circles is decreased by 1, so $d$ is a map from $V_{i} \otimes V^{\otimes (k_s-1)}$ to $V_{i}\otimes V^{\otimes (k_s-2)} $. 
	\begin{enumerate}

		\item Suppose the marked point $M$ lies on one of the two circles involved in the bifurcation. The differential is multiplying $V$ to the trivial $V$-module $V_i$. Specifically, \[d= m \otimes id: (V_i\otimes V) \otimes V^{\otimes (k_s-2)}\longmapsto  V_i\otimes V^{\otimes (k_s-2)},\] where $m : V_i \otimes V \longmapsto V_i $ is the multiplication defined previously, with $\un$ acting by identity, and $\coun$ acting by zero. 

	\item Suppose the marked point $M$ does not lie on either of the two circles. Then the differential works like the multiplication $m :V\otimes V \longmapsto V$ in the usual Khovanov homology.
	\end{enumerate}
\item \textbf{$1\rightarrow2$ bifurcation.}
	Again $e_s(P) = e_{s'}({P})$, and the number of circles is increased by 1, so $d$ is a map from $V_{i} \otimes V^{\otimes (k_s-1)}$ to $V_{i}\otimes V^{\otimes k_s} $. 
	\begin{enumerate}
	\item Suppose the marked point $M$ lies on the circle involved in the bifurcation. The differential is the comultiplication on the  trivial $V$-comodule $V_i$. Specifically, \[ d= \Delta \otimes id: V_i \otimes V^{\otimes (k_s-1)} \longmapsto (V_i \otimes V) \otimes V^{\otimes (k_s-1)},\] where $\Delta:V_i \longmapsto V_i\otimes V$ is the comultiplication $\Delta(y) = y\otimes \coun$ for all $y\in V_i$.
	\item Suppose the marked point $M$ does not lie on either of the two circles. Then the differential works like the comultiplication $\Delta :V \longmapsto V \otimes V$ in the usual Khovanov homology.
   \end{enumerate}

\item \textbf{$1\rightarrow1$ bifurcation.} 
For the $1\rightarrow1$ bifurcation, we have $e_{s'}(P) = e_s(P)+1$ mod $2$, and the number of circles is unchanged, so $d$ is a map from $V_i\otimes V^{\otimes (k_s-1)}$ to $V_{i+1}\otimes V^{\otimes (k_s-1)}$. This time the differential map is the same no matter whether $M$ lies on the involved circle or not. We change $V_i$ to $V_{i+1}$ by maps $f$ and $g$ in the dyad $\alpha = (V_0,V_1,f,g)$.
\begin{enumerate}
	\item If $e_s(P)=0$, then $d = f \otimes id: V_0\otimes V^{\otimes (k_s-1)}\longmapsto  V_1\otimes V^{\otimes (k_s-1)} $. 
	\item If $e_s(P)=1$, then $d = g \otimes id: V_1\otimes V^{\otimes (k_s-1)}\longmapsto  V_0\otimes V^{\otimes (k_s-1)} $. 
\end{enumerate}

\end{enumerate}
Observe that the map $d:\widetilde{\mathit{CKh}}^{P,\alpha}_s(L) \longmapsto \widetilde{\mathit{CKh}}^{P,\alpha}_{s'}(L)$ lowers the quantum degree by $1$ for all the above cases. Then, after the shift in the definition $\ref{eq:1}$, the chain map preserves the quantum degree. 

Now it is time to check this definition does give a chain complex. 
\begin{proposition}
We have	$d^2=0$.
	
\end{proposition} 

\begin{proof}
 It is enough to show that each square of the resolution cube commutes. (As we are working over $\mathbb{F}_2$, commuting is the same as anticommuting.) Hence it is enough to consider link projections with two crossings in $\rp{2}$, and ignore the rest part of the projection. A little caution need to be taken in terms of the location of the marked point $M$. It could appear on the link projection with two crossings, or it could lie on the neglected part.
 
  First we quote the following result from \cite{MR3189291}: Up to symmetries, there are exactly six singular graph in $\rp{2}$ with 2 singular points, as shown in Figure $\ref{fig:singular-curves-with-2-singularity}$.

\begin{figure}[t]
	\[
	{
		\fontsize{8pt}{10pt}\selectfont
		\def\svgwidth{5.5in}
		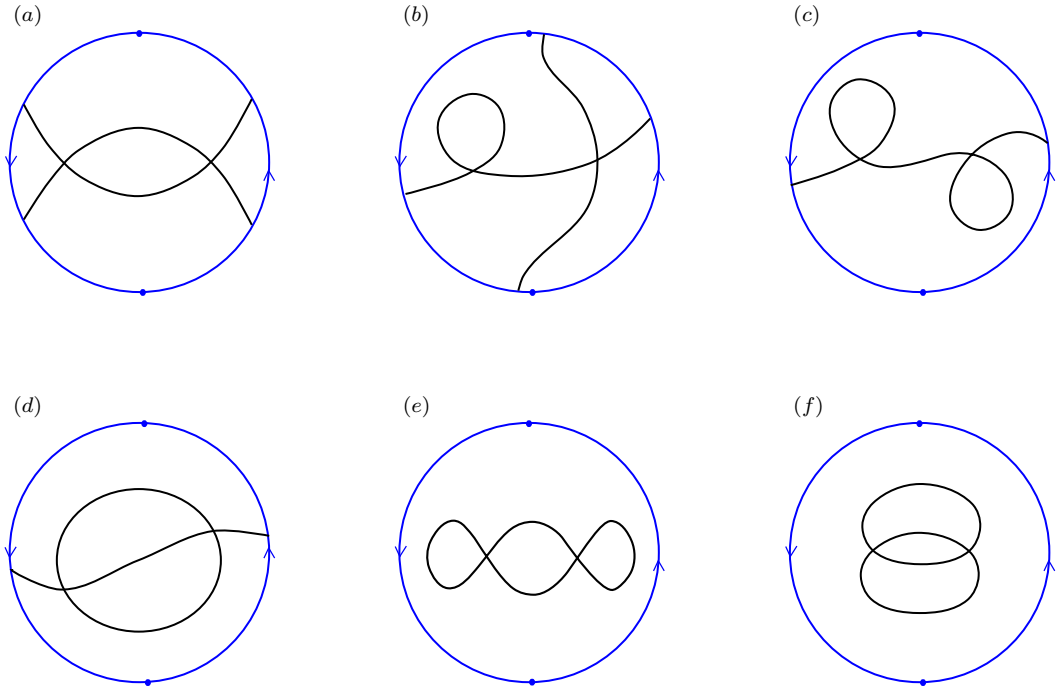
	}
	\]
	\caption{Singular graphs in $\rp{2}$ with 2 singular points}
	\label{fig:singular-curves-with-2-singularity}
\end{figure}

Now we replace each singular point by a positive or negative crossing to get a link with two crossings. For diagrams $(c)$ and $(d)$, the corresponding links are not null homologous, so we omit them in our discussion. For $(e)$ and $(f)$, the corresponding links are affine (they lie in a disk inside $\rp{2}$), so there are no $1\rightarrow1$ bifurcations. Then, the differential behaves like the reduced Khovanov homology if the marked point $M$ is presented, or the usual Khovanov homology if the marked point is not presented. So we have $d^2=0$ in these two cases as well. For the rest two diagrams $(a)$ and $(b)$, we will check its commutativity by hand. We need to compute different cases depending on whether the crossings are positive or negative, where the point $P$ is and whether the marked point $M$ is present for each links. One example will be presented here and the rest are left as exercise. 

Consider the link projection in Figure $\ref{fig:example-for-d20}$ and its resolutions in Figure $\ref{fig:example-for-d20-5}$.

\begin{figure}[t]
	\[
	{
		\fontsize{8pt}{10pt}\selectfont
		\def\svgscale{0.7}
		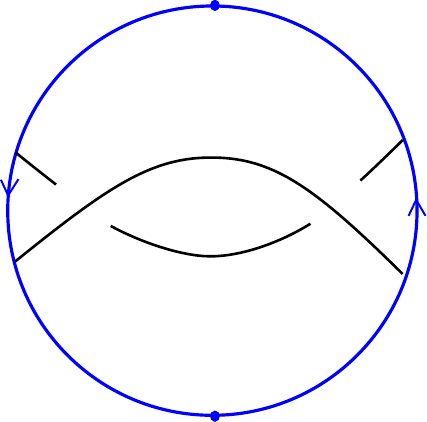
	}
	\]
	\caption{A link projection with 2 crossings}
	\label{fig:example-for-d20}
\end{figure}

\begin{figure}[t]
	\[
	{
		\fontsize{8pt}{10pt}\selectfont
		\def\svgscale{0.6}
		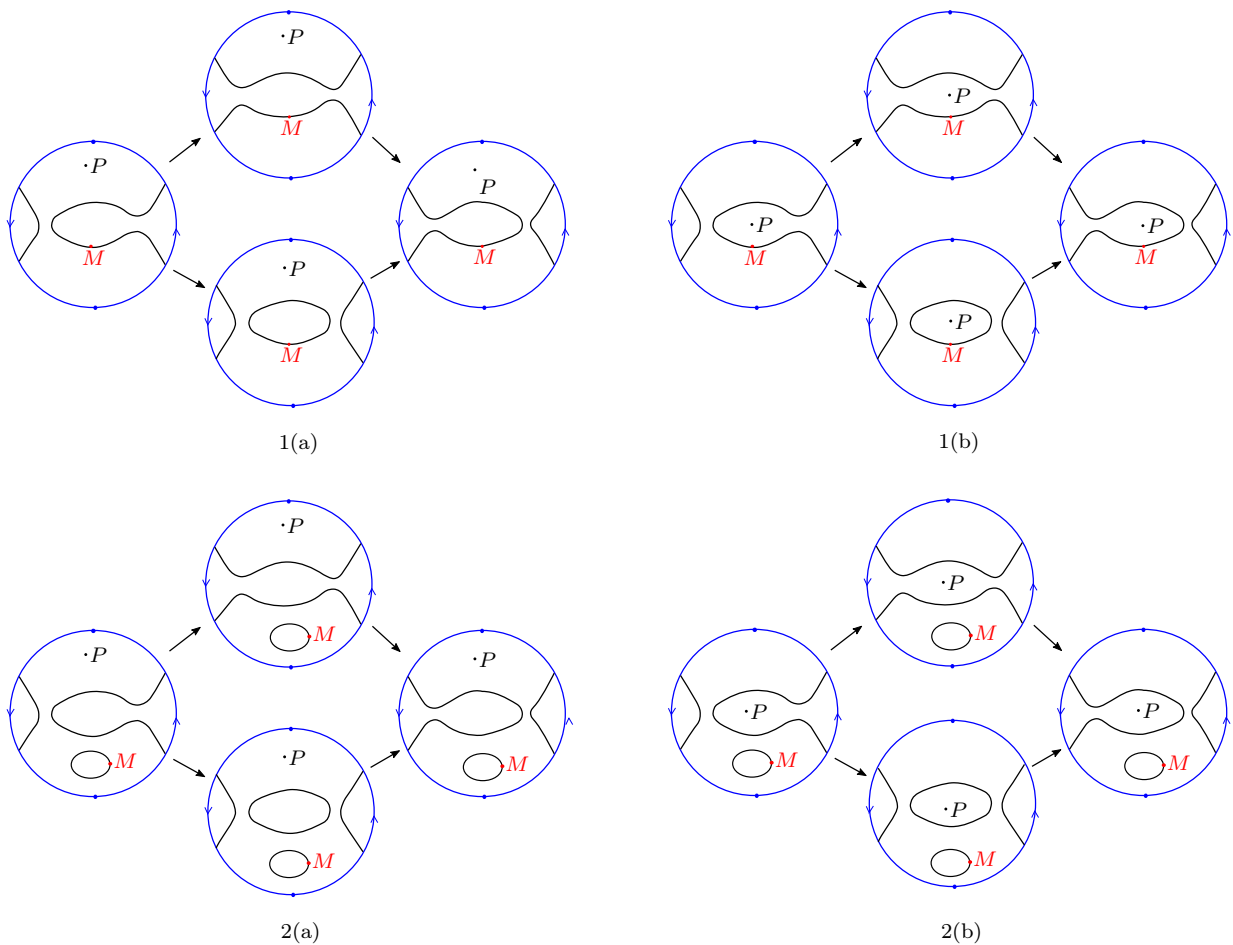
	}
	\]
	\caption{The cubes of resolutions corresponding to different $P$ and $M$}
	\label{fig:example-for-d20-5}
\end{figure}

\begin{enumerate}
	\item Suppose the marked point $M$ is present.
	\begin{enumerate}
		\item The cube of resolutions is drawn in $1(a) $ of Figure \ref{fig:example-for-d20-5}, with a choice of the extra point $P$.

	\begin{center}	
		\begin{tikzcd}[row sep=tiny]
			& V_1\arrow{dr}{g} \\
			V_0\{-1\} \arrow{ur}{f}\arrow{dr}{\Delta} & & V_0 \{1\}\\
			& V_0\otimes V  \arrow{ur}{m}		
		\end{tikzcd}
			
\end{center}
	
		Above are the corresponding vector spaces and maps between them. Both of the composition $m\circ \Delta, g \circ f$ equal to $0$, so it commutes.
		\item We have the same cube of resolutions as before, but with another choice of $P$ in $1(b)$ of Figure \ref{fig:example-for-d20-5}.

\begin{center}

			\begin{tikzcd}[row sep=tiny]
			& V_0 \arrow{dr}{f} \\
			V_1\{-1\} \arrow{ur}{g}\arrow{dr}{\Delta} & & V_1 \{1\}\\
			& V_1 \otimes V  \arrow{ur}{m}
    		\end{tikzcd}

\end{center}
    Again we have $m \circ \Delta = f\circ g=0$.
	\end{enumerate}
	\item Suppose the marked point $M$ is not present.
	\begin{enumerate}
		\item 
	Even though the circle with marked point $M$ is not involved, we draw it here, as the $1\rightarrow1$ bifurcation induces non-trivial map on the tensor factor corresponding to the marked circle. The cube of resolutions is shown in $2(a)$ of Figure \ref{fig:example-for-d20-5}.

		\begin{center}

			\begin{tikzcd}[row sep=tiny]
			&V_1 \otimes V  \arrow{dr}{g\otimes id_V} \\
			V_0\otimes V\{-1\} \arrow{ur}{f\otimes id_V}\arrow{dr}{id_{V_0} \otimes \Delta} & & V_0 \otimes V \{1\}\\
			& V_0\otimes V \otimes V  \arrow{ur}{id_{V_0}\otimes m}		
		\end{tikzcd}
				
\end{center}
		Again, both the compositions $m \circ \Delta $ and $g \circ f$ equal to 0. This time, we use the fact that we are working over $\mathbb{F}_2$, as $m \circ \Delta(\coun) = m (\coun\otimes \un + \un\otimes \coun ) = 2\coun=0$.

		\item We change the position of the point $P$ to get the cube of resolutions in 2(b) of Figure $\ref{fig:example-for-d20-5}$.

		\begin{center}

			\begin{tikzcd}[row sep=tiny]
			& V_0 \otimes V \arrow{dr}{f\otimes id_V} \\
			V_1\otimes V \{-1\}\arrow{ur}{g\otimes id_V}\arrow{dr}{id_{V_1} \otimes \Delta} & & V_1 \otimes V \{1\}\\
			& V_1\otimes V \otimes V \arrow{ur}{id_{V_1}\otimes m}		
		\end{tikzcd}
			
\end{center}
		Again this commutes because $f \circ g=0$ and $m \circ \Delta =0$.
	\end{enumerate}
\end{enumerate}
\end{proof}  

\subsection{Well-definedness of the homology}
\label{sec1.3}
In the last section, we got a chain complex $\widetilde{\mathit{CKh}}^{P,\alpha}_{\bullet}(L)$ for a link projection $L$ with a marked point $M$ on the projection $L$, a dyad $\alpha$ and another point $P$ outside the projection $L$. In this section, we will give a canonical choice of the point $P$ for a given link projection $L$, show its homology $\widetilde{\mathit{Kh}}^{P,\alpha}_{\bullet}(L)$ does not depend on the choice of the marked point $M$ and is invariant under the Reidemeister moves in $\rp{2}$ which passes $P$ even times.

 The chain complex $\widetilde{\mathit{CKh}}^{P,\alpha}_{\bullet}(L)$ depends on the position of $P$. To express the dependence in a succinct way, we take a detour into the discussion of Seifert surfaces and linking numbers for null homologous links in $\rp{3}$ first. 

 \begin{definition} For a null homologous link $K$ in $\rp{3}$, a \textbf{Seifert surface} $F$ of $K$ is a connected compact oriented surface contained in $\rp{3}$ such that $K$ is its oriented boundary.
 \end{definition}

\begin{figure}[t]
	\[
	{
		\fontsize{8pt}{10pt}\selectfont
		\def\svgscale{0.6}
		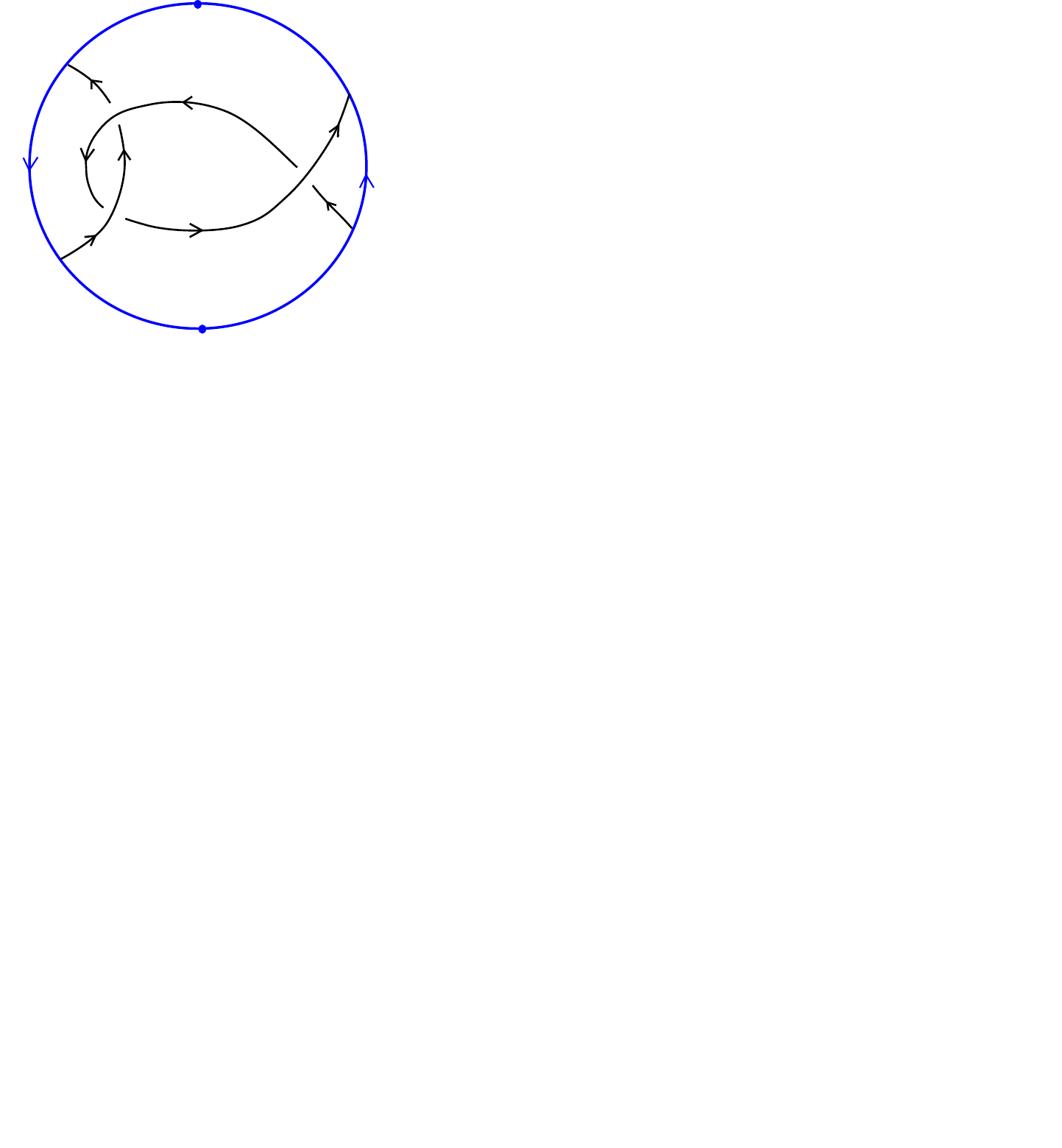
	}
	\]
	\caption{An illustration of a null homologous knot in $\rp{3}$, with a Seifert surface and the decomposition $\rp{2}\backslash L = R_0\bigsqcup R_1$}
	\label{fig:example-for-kh}
\end{figure}

 \begin{lemma}
 	Any null homologous link in $\rp{3}$ has a  Seifert surface.
 \end{lemma}

\begin{proof}
	It is essentially the same proof as the one for oriented links in $S^3$, see for example Theorem 2.2 in \cite{lickorish2012introduction}, with a little caution. Let $L$ be the link projection in $\rp{2}$ for a null homologous link $K$ in $\rp{3}$. Form the resolution $\widehat{L}$ which preserves the orientation at each crossing. See Figure $\ref{fig:example-for-kh}$ for an example. As stated in Lemma $\ref{lem:1.1.1}$, $\widehat{L}$ is a disjoint union of circles, each of which is null homologous as well, dividing $\rp{2}$ into a disk part and a M\"obius band part. Each of these disks gets an orientation from the orientation of $\widehat{L}$. Now join these disks together with half-twisted strips at crossings, which relay the orientations between disks as we formed the resolution preserving the orientation. If the obtained surface is not connected, connect components together by removing small discs and inserting thin tubes.	
\end{proof}

 \begin{definition}
 	\label{def1.3.3}
 	Given a null homologous link $K$ and an oriented link $K'$ in  $\rp{3}$, we define the \textbf{linking number} $\text{lk}(K,K')$ of $K$ and $K'$ as the intersection number $\langle F,K'\rangle$ of a Seifert surface $F$ of $K$ with $K'$.
 \end{definition}

 Note that the linking number $\textit{lk}(K,K')$ doesn't depend on the choice of the Seifert surface $F$ of $K$, as $H_2(\rp{3},\mathbb{Z})=0$. It does depend on the orientation of $\rp{3}$, $K$ and $K'$, but if we only care about the parity of $\textit{lk}(K,K')$, which will be our main concern, then it doesn't matter how we choose the orientation. The orientation of $K$ will be fixed by the given orientation of the null homologous link, and we won't specify the orientation of $K'$ nor of $\rp{3}$ in the following discussion. 
 
Let us go back to the dependence of $\widetilde{\mathit{CKh}}^{P,\alpha}_{\bullet}(L)$ on $P$. Define $C_P$ as the union of the fiber over $P$ in the twisted $I$-bundle $\rp{2}\tilde{\times}I$ with the deleted point $*$ in $\rp{3}\backslash\{*\} = \rp{2}\tilde{\times}I$. So $C_P$ a circle in $\rp{3}$. We can divide the complement of $L$ in $\rp{2}$ into two regions $\rp{2} \backslash L = R_0 \bigsqcup R_1$, such that 
\begin{equation}
	\label{eq1.3.0.1}
	R_i = \{P \in \rp{2} \,\backslash L \,\,\,|\,\,\, \textit{lk}(L,C_P) = i \text{ mod }2\}.
\end{equation}
We will show that  $\widetilde{\mathit{CKh}}^{P,\alpha}_{\bullet}(L)$ depends on $P$ only through the parity of $\textit{lk}(L,C_P)$.
 
 \begin{proposition}
 	
 	For two points $P,Q$ in the complement of $L$ in $\rp{2}$,
 	\label{prop:1.3.4}
 	\begin{equation*}
 	\widetilde{\mathit{CKh}}^{Q,\alpha}_{\bullet}(L) =
 	\begin{cases}
 		\widetilde{\mathit{CKh}}^{P,\alpha}_{\bullet}(L) & \text{if } \textit{lk}(L,C_Q) = \textit{lk}(L,C_P) \text{ mod } 2\\
 		\widetilde{\mathit{CKh}}^{P,\alpha^*}_{\bullet}(L) & \text{if } \textit{lk}(L,C_Q) = \textit{lk}(L,C_P)+1 \text{ mod } 2,
 	\end{cases}       
 \end{equation*}
 	 where $\alpha^*= (V_1,V_0,g,f)$ is the dual dyad of $\alpha=(V_0,V_1,f,g)$.
 \end{proposition}
\begin{proof}
 Choose an embedded path $\gamma$ in $\rp{2}$ connecting $P$ and $Q$ and avoiding the double points of $L$. Since $L$ is null homologous in $\rp{2}$, the number of intersection points $\langle \gamma, L\rangle$ between $\gamma$ and $L$  mod $2$ doesn't depend on the choice of $\gamma$. From the previous construction of Seifert surfaces of $L$, it is clear that \[\textit{lk}(L,C_Q) = \textit{lk}(L,C_P)+\brac{\gamma,L} \text{ mod }2.\] We separate into two cases depending on the parity of $\langle \gamma, L\rangle$. 
	 
	 \begin{enumerate}
	 	\item  If $\langle \gamma, L\rangle$ is even, then for any given state $s\in\{0,1\}^n $, the intersection number  $\langle \gamma, L_s\rangle$ of $\gamma$ and the smoothing $L_s$ is even as well, as $\gamma$ avoids the double points of $L$. For each circle $c$ in $L_s$, if $\gamma$ intersects $c$ even times, then either both $P$ and $Q$ are encircled by $c$, or neither of $P$ and $Q$ are encircled by $c$. If $\gamma$ intersects $c$ odd times, then exactly one of $P$ and $Q$ is encircled by $c$. Since $\langle \gamma, L_s\rangle$ is even, there are even number of circles in $L_s$ such that $\gamma$ intersects it odd times, so the numbers of circles encircling $P$ and $Q$ in the smoothing $L_s$ are the same mod $2$, $i.e.$, $e_s(P) = e_s (Q)$ for any state $s$.  Therefore, the two chain complexes $\widetilde{\mathit{CKh}}^{P,\alpha}_{\bullet}(L)$ and $\widetilde{\mathit{CKh}}^{Q,\alpha}_{\bullet}(L)$ are exactly the same.
	 	\item If $\langle \gamma, L\rangle$ is odd, then by the previous argument, we get $e_s(Q) = e_s (P)+1$ mod $2$ for any state $s$. By checking the definition of the chain complex, we find it is equivalent to switching the role of $V_0$ with $V_1$ and $f$ with $g$. Therefore, we have $\widetilde{\mathit{CKh}}^{Q,\alpha^*}_{\bullet}(L)=\widetilde{\mathit{CKh}}^{P,\alpha}_{\bullet}(L)$, where $\alpha^*=(V_1,V_0,g,f)$ is the dual dyad of $\alpha= (V_0,V_1,f,g)$.
	 \end{enumerate}\end{proof}


See Figure $\ref{fig:example-for-kh}$ for an example of the division $\rp{2} \backslash L = R_0 \bigsqcup R_1$. From now on, we will always assume $P$ lies in the region $R_0$, and drop $P$ from the notation $\widetilde{\mathit{CKh}}^{P,\alpha}_{\bullet}(L)$ and $\widetilde{\mathit{Kh}}^{P,\alpha}_{\bullet}(L)$.

\begin{remark}
	If we change the orientation of some components of the null homologous link $K$, it will change the orientation-preserving smoothing $\hat{L}$, and hence the the division  $\rp{2} \backslash L = R_0 \bigsqcup R_1$. For this reason, we need to fix an orientation on the null homologous link $K$.
\end{remark}

Now we move on to prove invariance under different choices of the marked point $M$.
\begin{proposition}
	 \label{prop1.3.5}
 The homology $\widetilde{\mathit{Kh}}^{\alpha}_{\bullet}(L)$ doesn't depend on the choice of marked point $M$ on $L$.

\end{proposition} 

\textit{Proof.}  Our proof is inspired by the discussion relating the Heegaard Floer homology of branched double cover of $S^3$ branching over $K$ and the Khovanov homology of it in \cite{MR2141852}. See \cite{MR3071132} as well. We define an automorphism $\phi_s: \widetilde{\mathit{CKh}}^{\alpha}_{s}(L) \longmapsto \widetilde{\mathit{CKh}}^{\alpha}_{s}(L)$ induced by the change of marked point $M$ for each state $s$, and check it commutes with the differential $d$. Therefore, we obtain a chain automorphism $\widetilde{\mathit{CKh}}^{\alpha}_{\bullet}(L) \longmapsto \widetilde{\mathit{CKh}}^{\alpha}_{\bullet}(L)$, which induces an isomorphism on $\widetilde{\mathit{Kh}}^{\alpha}_{\bullet}(L)$. 

Let us begin to define $\phi_s$. For a given state $s$, there is a marked circle in  the smoothing $L_s$. Label the marked circle by 0, and the rest circles in the smoothing $L_s$ from 1 to $k_s-1$. Suppose we change the position of marked point $M$ to the circle labeled 1. Give $V^{\otimes (k_s-1)}$ an algebra structure as the quotient of the polynomial algebra over $\mathbb{F}_2$ generated by $S_i$, quotienting out by the relations $S_i^2=0$ for $i=1,...,k_s-1$, where $S_i=\un\otimes \un\otimes...\otimes \coun \otimes \un\otimes...\otimes \un$ with $\coun$ at the $i$th component.  Define $\eta_s:V^{\otimes (k_s-1)}\longmapsto V^{\otimes (k_s-1)} $ to be the algebra automorphism, such that \begin{align*}
		\eta_s(S_1) &= S_1\\
	\eta_s(S_i) &= S_1 +S_i, \text{  for   } 2\leq i\leq k_s-1,
\end{align*} and extend multiplicatively. Note that since we moved the marked point from circle 0 to circle 1, these two labels are switched in the image of the map $\phi_s$.

In this formalism, merging of two circles labeled $i$ and $j$ corresponds to further quotienting out the relation $S_i=S_j$, and splitting of a circle labeled $i$ corresponds to product with $S_i + S_{k_s}$, where $k_s$ is the label of the new circle. Merging of a circle labeled $i$ with the marked circle corresponds to dividing out by the relation $S_i=0$, and splitting of the marked circle corresponds to product with $S_{k_s}$ with $k_s$ the label of the new circle. See Section 5 and 6 in \cite{MR2141852} for more discussion. 

Define $\phi_s: V_i\otimes V^{\otimes (k_s-1)} \longmapsto V_i\otimes V^{\otimes (k_s-1)} $  as $\phi_s =  id_{V_i}\otimes \eta_s$, where $i=e_s(P)$. 

Now we need to check the following diagram commutes.


\begin{center}
\begin{tikzcd}
	\widetilde{\mathit{CKh}}^{\alpha}_{s}(L)  \arrow[r,"d_1"] \arrow{d}{\phi_s}
	& \widetilde{\mathit{CKh}}^{\alpha}_{s'}(L) \arrow[d,"\phi_{s'}"] \\
	\widetilde{\mathit{CKh}}^{\alpha}_{s}(L) \arrow{r}{d_2}
	& \widetilde{\mathit{CKh}}^{\alpha}_{s'}(L)
\end{tikzcd}

\end{center}
We discuss the cases when $d$ is given by the $2\rightarrow1$ bifurcation, the $1\rightarrow2$ bifurcation or the $1\rightarrow1$ bifurcation separately. 

The $1\rightarrow1$ bifurcation is the easiest case, as $d$ acts only on the $V_i$ component of the tensor product, so it commutes with the change of variables in the $V^{\otimes (k_s-1)}$ component.

For the $2\rightarrow1$ bifurcation, it separates into cases depending on whether the marked circles are involved or not.  Note that all the maps in the square are algebra homomorphisms on the tensor component $V^{\otimes{(k_s-1)}}$ and identity on the $V_i$ component, so it is enough to check it on the generators $S_i$. We check different cases as follows.
\begin{enumerate}
	\item Suppose the bifurcation merges two circles different from the circle $0$ and $1$, say circle $2$ and circle $3$. Then both $\eta_s$ and $\eta_{s'}$ send $S_1$ to $S_1$ and $S_i$ to $S_i+S_1$ for $i\neq 1$, while both $d_1$ and $d_2$ quotient out the relation $S_2=S_3$. Then it is obvious the above square commutes.
	\item Suppose the bifurcation merges the circle $0$ with a circle other than circle $1$, say circle $2$. Then $\eta_s$ and $\eta_{s'}$ behave the same as in the first case, while $d_1$ quotients out $S_2=0$, and $d_2$ quotients out $S_1=S_2$. Hence,
	\begin{align*}
		 &\eta_{s'} \circ d_1(S_1) = \eta_{s'}(S_1) = S_1 = d_2(S_1) = d_2 \circ \eta_s(S_1) \\
		 &\eta_{s'} \circ d_1(S_2) = \eta_{s'} (0) = 0 = d_2 (S_1+S_2)= d_2 \circ \eta_s(S_2) \\ 
		  &\eta_{s'} \circ d_1(S_i) = \eta_{s'} (S_i) = S_1+S_i = d_2 (S_1+S_i)= d_2 \circ \eta_s(S_i), \text{ for } i > 2.
	\end{align*} 
If the bifurcation merges circle $1$ with a circle other than circle $0$, then we get a similar diagram, except the arrows $\phi_s$ and $\phi_{s'}$ are reversed. We can check its commutativity by a similar calculation.
		  \item Suppose the bifurcation merges circle $0$ and circle $1$. $\eta_s$ is the same as above, while $\eta_s'$ is the identity, as we move $M$ on the same circle. $d_1$, $d_2$ both quotient out the relation $S_1=0$. Hence,
		  \begin{align*}
		  	&\eta_{s'} \circ d_1(S_1) = \eta_{s'}(0) = 0 = d_2(S_1) = d_2 \circ \eta_s(S_1) \\
		  	&\eta_{s'} \circ d_1(S_i) = \eta_{s'} (S_i) = S_i = d_2 (S_1+S_i)= d_2 \circ \eta_s(S_i), 
		  	\text{ for } i \geq 2.
		  \end{align*}  

\end{enumerate}
The analysis for the $1\rightarrow2 $ bifurcation is similar. Again, we divide into cases depending on whether the marked circle is involved or not.
\begin{enumerate}
	\item Suppose the bifurcation splits the a circle other than circle $0$ and circle $1$, say circle $2$, into two new circles labeled $2$ and $k_s$. Then both $d_1$ and $d_2$ are multiplication with $S_2 + S_k$. The change of variables map $\phi_{s'}$ sends $S_2+S_k$ to $S_2+S_1 +S_k + S_1$, which equals to $S_2 + S_k$ as we are working over $\mathbb{F}_2$. So the square commutes. 
	\item Suppose the bifurcation splits the  circle $0$. Then $d_1$ is multiplying $S_k$, and $d_2$ is multiplying $S_1 + S_k$. As $\phi_{s'}(S_k) = S_1 +S_k$,  the square commutes.
	\item Suppose the bifurcation splits the circle 1. Then $d_1$ is multiplying $S_1+S_k$ and $d_2$ is multiplying $S_k$. Then the square commutes as well, because  $ \phi_{s'}(S_k+S_1)=S_1+S_k+S_1 = S_k$, as we are working over $\mathbb{F}_2$. 
\end{enumerate}




\begin{flushright}
	$\square$
\end{flushright}

\begin{remark}
	The choice of the change of variable map $\eta_s$ for changing the marked point will become natural when we discuss the relation of this homology with the Heegaard Floer homology of the branching double cover of $\rp{3}$ branched over $K$.
\end{remark}

Now we discuss the invariance of $\widetilde{\mathit{Kh}}^{\alpha}_{\bullet}(L)$ under Reidemeister moves. First, there are 5 Reidemeister moves in $\rp{2}$, the three usual ones and two additional ones that act across the boundary of the 2-disk. They are drawn in the Figure $\ref{fig:r-moves}$. Two links are ambient isotopic in $\rp{3}$ if a link projection of one link can be transformed to a link projection of the other by a finite sequence of the moves R-I to R-V. See \cite{MR1296890} for further discussion. 

  The decomposition $\rp{2}\backslash L = R_0 \bigsqcup R_1$ described before Proposition $\ref{prop:1.3.4}$ is changed under Reidemeister moves, such that the region in $\rp{2}$ swept by the moves is changed from $R_i$ to $R_{i+1}$. We will pick a point $P$ which lies in $R_0$ for both of the link projections before and after the Reidemeister move. It could be achieved by choosing $P$ which lies in $R_0$ for the link projection before the Reidemeister move, such that the Reidemeister move doesn't cross it. So we can assume the Reidemeister moves doesn't cross the point $P$.

\begin{figure}
\[
{
	\fontsize{9pt}{11pt}\selectfont
	\def\svgscale{0.6}
	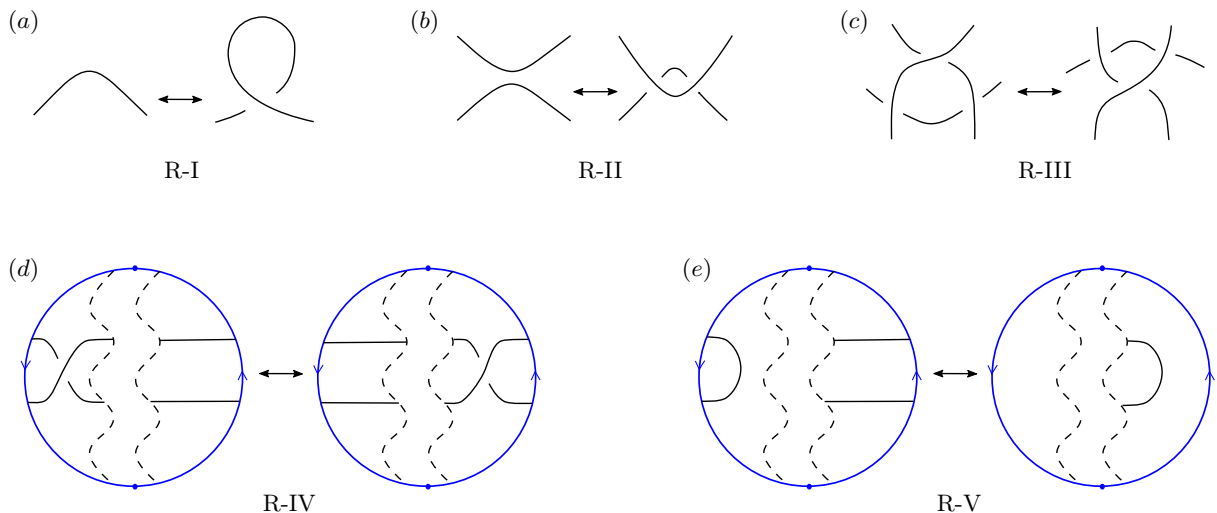
}
\]
			\caption{Reidemeister moves in $\rp{3}$}
	\label{fig:r-moves}
\end{figure}
\begin{proposition}
	The homology $\widetilde{\mathit{Kh}}^{\alpha}_{\bullet}(L)$ is invariant under Reidemeister moves, so it is an invariant of null homologous links in $\rp{3}$.
\end{proposition}

\begin{proof}

 Note that for R-IV and R-V, they don't change the chain complex $\widetilde{\mathit{CKh}}^{\alpha}_{\bullet}(L)$ even, so the homology is of course invariant. We are left with Reidemeister moves R-I, R-II and R-III. For them, the proof of invariance for the usual Khovanov homology, as described in Section 3.5 of \cite{MR1917056}, works with a slight change. We borrow the notation $[\![L ]\!]$ for $\widetilde{\mathit{CKh}}^{\alpha}_{\bullet}(L)$ from \cite{MR1917056} as well, where we only draw the part of $L$ which is changed under the Reidemeister moves.
 
For R-I, we can assume the marked point $M$ does not lie on the moving part of the knot under R-I, as we can move $M$ by a change of variable described in Proposition $\ref{prop1.3.5}$ if necessary. Then the proof of the invariance of the Khovanov homology under R-I works here as well with no change. Note that we have done the shift in the homological degree implicitly in the definition of $\widetilde{\mathit{CKh}}^{\alpha}_{i}(L)$, by letting $\#1(s) = i+n_-$. The shift in quantum degree is the same as in \cite{MR1917056} as well, which is $\#1(s) + n_+-2n_- = i +n_+-n_-$. The reason we are doing these shifts is precisely to make the homology $\widetilde{\mathit{Kh}}^{\alpha}_{\bullet}(L)$ invariant under R-I.

\begin{figure}[t]
	\[
	{
		\fontsize{8pt}{10pt}\selectfont
		\def\svgscale{0.6}
		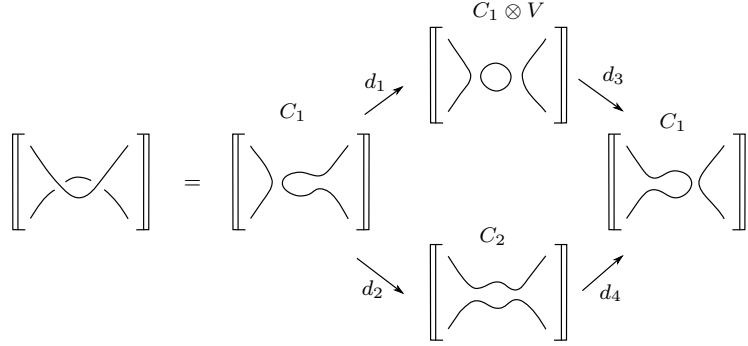
	}
	\]
	\caption{Invariance under RII}
	\label{fig:invariance-unde-r2}
\end{figure}
 
For R-II, again we can assume the marked point $M$ does not lie on the moving part of the knot under R-II. Consider the chain complex in Figure $\ref{fig:invariance-unde-r2}$, where $d_1: C_1 \longmapsto C_1 \otimes V$ is the chain map given by comultiplication $\Delta$, and $d_3 :C_1 \otimes V \longmapsto C_1$ is the multiplication $m$. Note that $d_3$ is an isomorphism on the subchain $C'$ of $C_1 \otimes V$, which takes value $\un$ in the factor $V$. So we can quotient out the subcomplex $ \{C' \longmapsto C_1\}$ without changing the homology. Now $d_1$ induces an isomorphism between $C_1$ and $(C_1\otimes V)/C'$. Then we further quotient out the subcomplex generated by $C_1$. This final quotient is isomorphic to $C_2$, as in the proof for the usual Khovanov homology. Observe that in the proof, we only used the properties of maps $d_1$ and $d_3$, which correspond to splitting and merging a circle respectively, and are entirely similar to the corresponding maps in the usual Khovanov chain complex. The map $d_2$ and $d_4$ might be different from the corresponding maps in the usual Khovanov chain complex, as some $1\rightarrow1$ bifurcations could happen, but we don't use the properties of $d_2$ and $d_4$ in the proof. 

For R-III, we can assume the marked point does no lie on the moving part of R-III as well. Then the usual proof of invariance under R-III works here as well, for the similar reason as in the proof of invariance under R-II. Check Section 3.5 of \cite{MR1917056} for more details.
\end{proof}


Therefore, we obtain the following theorem. 
\begin{theorem}
		For each dyad $\alpha$, $\widetilde{\mathit{Kh}}^{\alpha}(L)$ is an invariant of null homologous links in $\rp{3}$.
\end{theorem}

\subsection{The Euler characteristic}
\label{sec1.4}
The Euler characteristic of the usual Khovanov homology for links in $S^3$ gives the unnormalized Jones polynomial of the link. For (not necessarily null homologous) framed links in $\rp{3}$, \cite{MR2113902} and \cite{MR3189291} construct some Khovanov-type homology whose Euler characteristic gives the Kauffman bracket $\brac{L}$ of the framed link $L$, which is an element in the skein module $S(\rp{3})$ of $\rp{3}$. $S(\rp{3})$ is free $\mathbb{Z} \left[A^{\pm 1}\right]$-module over two generators. See discussion of the skein module of lens space in \cite{MR1238877}. In this paper, we use another convention such that $q = -A^2$. Our homology theory depends on the input $\alpha=(V_0,V_1,f,g)$, and its Euler characteristic will be a linear combination of the graded dimension of $V_0$ and $V_1$. Let us begin with some definitions.

\begin{definition}
	For a graded vector space $W = \bigoplus_m W_m$ with homogeneous components $\{W_m\}$, the \textbf{graded dimension} of $W$ is the Laurent polynomial $qdim(W) = \sum_mq^m dim(W_m)$. The \textbf{Euler characteristic} $\chi(C)$ of a chain complex $C =\bigoplus_{i,m}C_{i,m} $ of graded vector space is the alternating sum of the graded dimensions of its homology groups, where $i$ is the homology grading, and $m$ is the quantum grading,
	\[\chi(C) = \sum_{i,m}(-1)^iq^m dim (H_{i,m}).\]
\end{definition}
 
Note that $\chi(C)$ is the same as the alternating sum of the graded dimensions of its chain groups if the differential $d$ preserves the quantum grading, and all the chain groups are finite dimensional.

Now we define two variations of the Kauffman bracket of null homologous link projection in $\rp{2}$. Fix a null homologous link $K$ in $\rp{3}$ and its projection $L$ to $\rp{2}$. Choose a point $P$ in the complement of $L$ in $\rp{2}$.

\begin{definition}
	 Define the \textbf{even Kauffman bracket} $\langle L \rangle_{0}^P$ of a null homologous link projection $L$ in $\rp{2}$ avoiding $P$ by the following two rules:
	 \begin{enumerate}
	 	\item The skein relation: $\langle \backoverslash\rangle_0^P = \langle\hsmoothing\rangle_0^P
	 	- q\langle\smoothing\rangle_0^P$,
	 	\item If $L$ is a disjoint union of $k$ (necessarily null homologous) circles, then \begin{equation*}
	 		\langle L \rangle_{0}^P =
	 		\begin{cases}
	 			(q+q^{-1})^{k-1} & \text{if }P \text{ is encircled by even number of circles in }L,\\
	 			0 & \text{otherwise.}
	 		\end{cases}       
	 	\end{equation*}
	 \end{enumerate} 
 Similarly, define the \textbf{odd Kauffman bracket} $\langle L \rangle_{1}^P$ of a null homologous link projection $L$ in $\rp{2}$ avoiding $P$ by the rules:
  \begin{enumerate}
 	\item The skein relation: $\langle \backoverslash\rangle_1^P = \langle\hsmoothing\rangle_1^P
 	- q\langle\smoothing\rangle_1^P$,
 	\item If $L$ is a disjoint union of $k$ (necessarily null homologous) circles, then \begin{equation*}
 		\langle L \rangle_{1}^P =
 		\begin{cases}
 			(q+q^{-1})^{k-1} & \text{if }P \text{ is encircled by odd number of circles in }L,\\
 			0 & \text{otherwise.}
 		\end{cases}       
 	\end{equation*}
 \end{enumerate} 

\end{definition}

As in the case of $\widetilde{\mathit{CKh}}^{P,\alpha}_{\bullet}(L)$, the Kauffman bracket $\brac{L}_i^P$ depends on $P$ only through the parity of the linking number $\textit{lk}(L,C_P)$ defined in Definition \ref{def1.3.3}, so for a given null homologous link projection $L$, we pick $P$ such that $\textit{lk}(L,C_P)$ is even, and drop $P$ from the notation $\brac{L}_i^P$ for $i=0,1$.
\begin{definition}
	Define the \textbf{even/odd Jones polynomial} of $L$ by $J_i(L) = (-1)^{n_-}q^{n_+-2n_-}\langle L\rangle_i$ for $i=0,1$ respectively.
\end{definition}
 It is easy to see $J_i(L)$ is invariant under Reidemeister moves in $\rp{2}$, so it is an invariant for null homologous links in $\rp{3}$. 
 
From our definition of the chain complex $\widetilde{\mathit{CKh}}^{\alpha}_{\bullet}(L)$, we have the following description of its Euler characteristics.

\begin{proposition}
	For a given null homologous link projection $L$ and a dyad $\alpha=(V_0,V_1,f,g)$, we have $\chi (\widetilde{\mathit{CKh}}^{\alpha}_{\bullet}(L)) = qdim(V_0) J_0(L) + qdim(V_1) J_1(L)$.  
\end{proposition} 
\begin{proof}
 By the definition of the chain complex $ \widetilde{\mathit{CKh}}^{\alpha}_{\bullet}(L)$, both the left hand side and the right hand side of the equation satisfy the same skein relations relating link projections $\backoverslash, \hsmoothing$ and $\smoothing$, so it is enough to verify the equation when $L$ is a disjoint union of circles, which again follows from the definition of $\widetilde{\mathit{CKh}}^{\alpha}_{\bullet}(L)$.
\end{proof}
Note that the sum of the even and odd Kauffman bracket recovers the usual Kauffman bracket for null homologous link projections in $\rp{2}$:\[\brac{L} = \brac{L}_0 + \brac{L}_1,\] so they are some refinements of the usual Kauffman bracket for null homologous link projections in $\rp{2}$

Another observation is that the Euler characteristic $\chi (\widetilde{\mathit{CKh}}^{\alpha}_{\bullet}(L))$ does not depend on the map $f:V_0 \longmapsto V_1$ and $g:V_1 \longmapsto V_0$ in $\alpha$. So by changing $f,g$ while keeping $V_0$ and $V_1$ fixed, we get different homology theories categorifying the same Jones polynomial.

\subsection{An unreduced version of the chain complex}
\label{sec1.5}
In defining the chain complex $\widetilde{\mathit{CKh}}^{\alpha}_{\bullet}(L)$, we choose a marked point $M$ on $L$, assign the vector space $V_{i}$ to the marked circle in each smoothing $L_s$ and assign $V$ to each of the other circles. This resembles the usual definition of reduced Khovanov homology, where we assign the base field $\mathbb{F}_2$ to the marked circle and $V$ to each of the other circles. The difference with the usual one occurs when there is a $1\rightarrow1$ bifurcation, then we change the vector space associated to the marked circle from $V_i$ to $V_{i+1}$.

As in the usual Khovanov homology, we can define an unreduced version of the chain complex as well, denoted as $CKh^{\alpha}_{\bullet}(L)$. Pick a point $P$ in the preferred region $R_0$ of the complement of $L$. (See the discussion before Proposition $\ref{prop:1.3.4}$) Now for each state $s$, we associate the vector space $V_{e_s(P)}\otimes V^{\otimes k_s}$ to it, which assigns $V$ to each of the circles in the smoothing $L_s$, and treat $V_{e_s(P)}$ as a background component. For the differential, we use the usual maps $m: V\otimes V \longmapsto V$ and $\Delta:V \longmapsto V \otimes V$ for the $2\rightarrow1$ bifurcation and the $1\rightarrow2$ bifurcation respectively, as in the usual Khovanov homology. For the $1\rightarrow1$ bifurcation, we use maps $f:V_0\longmapsto V_1$ and $g: V_1 \longmapsto V_0$ acting on the background component $V_i$. The proofs of $d^2=0$ and the invariance of the homology $Kh^{\alpha}_{\bullet}(L)$ under Reidemeister moves are almost the same as we have presented in the previous section (actually slightly easier because we don't need to divide into cases according to where the marked point is). The Euler characteristic of  $CKh^{\alpha}_{\bullet}(L)$ is then related to the unnormalized Jones polynomial, \[ \chi( CKh^{\alpha}_{\bullet}(L)) = (q+q^{-1})\chi(\widetilde{ CKh}^{\alpha}_{\bullet}(L)).\]
\begin{remark}
	 It might have been more natural to start with the unreduced version instead of the reduced version. The reason we chose to present the reduced version first is because the reduced one is what we obtained in the computation of the Heegaard Floer homology of the branched double cover of $\rp{3}$ branching over a knot.   Another reason is that the proofs for the unreduced version are contained in the proofs for the reduced version.
\end{remark}

\subsection{Some example calculation}
\label{sec1.6}
In this subsection, we present computations of $\widetilde{\mathit{Kh}}^{\alpha}_{\bullet}(K)$ for some specific choices of the dyad $\alpha= (V_0,V_1,f,g)$ and the null homologous knot $K$ in $\rp{3}$ drawn in Figure $\ref{fig:example-for-kh}$. The cube of resolution is drawn in Figure $\ref{fig:cube-of-resolution}$. 

\begin{figure}[t]
	\[
	{
		\fontsize{8pt}{10pt}\selectfont
		\def\svgwidth{5.5in}
		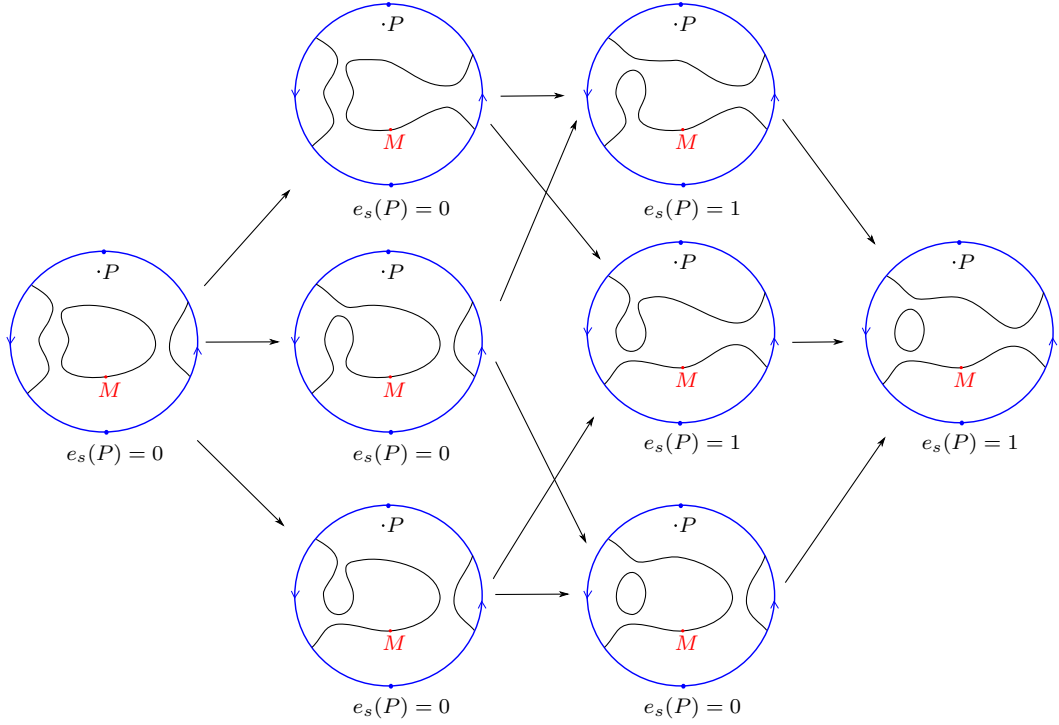
	}
	\]
	\caption{The cube of resolutions}
	\label{fig:cube-of-resolution}
\end{figure}
To save some space when writing down the homology, we will express the Khovanov-type homology $\widetilde{\mathit{Kh}}^{\alpha}_{\bullet}(K)$ by its graded Poincar\'e polynomial $p(\widetilde{\mathit{Kh}}^{\alpha}_{\bullet}(K)) $, which is \[p(\widetilde{\mathit{Kh}}^{\alpha}_{\bullet}(K)) = \sum_{i,m} t^iq^m\text{dim}(\widetilde{\mathit{Kh}}^{\alpha}_{i,m}(K))  .\] The Euler characteristic is obtained by taking $t=-1$.


\begin{enumerate}
	\item $\alpha_{\textit{APS}} =(\mathbb{F}_2,\mathbb{F}_2,0,0)$. For this choice of dyad, the unreduced chain complex $CKh^{\alpha_{\textit{APS}}}_{\bullet}$ recovers the chain complex defined in \cite{MR2113902} and \cite{MR3189291} for null homologous knots in $\rp{3}$ with coefficients $\mathbb{F}_2$. Their homology theories work for all knots in $\rp{3}$, not only the null homologous ones. \cite{MR2113902} proposes a $\mathbb{F}_2$ version, and \cite{MR3189291} fixes some choice of signs so that it works over $\mathbb{Z}$. Note that since we have the symmetry $\alpha_{\textit{APS}} = \alpha_{\textit{APS}}^*$, we can pick the point $P$ anywhere in the complement of $L$ in $\rp{2}$, and it will give the same chain complex, so $P$ plays no role in their homology. 
    	\begin{center}
    	
    	\begin{tikzcd}[row sep=tiny]
    		& \mathbb{F}_2\{-2\} \arrow{r}\arrow{dr} & \mathbb{F}_2\{-1\} \arrow{dr}& \\
    		V\{-3\} \arrow{ur} \arrow{r} \arrow{dr} &\mathbb{F}_2\{-2\} \arrow{ur}\arrow{dr} & \mathbb{F}_2\{-1\} \arrow{r} & V\\
    		&  \mathbb{F}_2\{-2\} \arrow{r}\arrow{ur}	& V\{-1\} \arrow{ur} &
    	\end{tikzcd}
    	
    \end{center} 
The chain complex is shown above, where the maps $V\longmapsto\mathbb{F}_2$ are multiplication of $V$ on the trivial $V$-module $\mathbb{F}_2$, the maps are $\mathbb{F}_2 \longmapsto V$ are comultiplication of the trivial comodule $\mathbb{F}_2$, and the maps $\mathbb{F}_2 \longmapsto \mathbb{F}_2$, $V\longmapsto V$ are 0. The Poincar\'e polynomial is  \[ p(\widetilde{\mathit{Kh}}^{\alpha_{APS}}_{\bullet}) =  t^{-2}q^{-4} +t^{-1}q^{-2}+q^{-1}+1 +tq. \] 
	\item $\alpha_0 = (\mathbb{F}_2,0,0,0) $ and its dual $\alpha_1 = \alpha_0^* = (0,\mathbb{F}_2,0,0)$. The corresponding chain complexes $\widetilde{\mathit{CKh}}^{\alpha_i}_{\bullet}(L)$ are subcomplexes of $\widetilde{\mathit{CKh}}^{\alpha_{\textit{APS}}}_{\bullet}(L)$, consisting of those states $s$ such that $e_s(P) =i$ for $i=0,1$. The Euler characteristics of them equal to the even/odd Jones polynomials  $\chi(CKh^{\alpha_i}_{\bullet}(L)) = J_i(L)$ for $i=0,1$ respectively. In some sense, all  other chain complexes for other choices of $\alpha$ are linear combinations of these two chain complexes.
	
	\begin{center}
		
		\begin{tikzcd}[row sep=tiny]
			& \mathbb{F}_2\{-2\} \arrow{r}\arrow{dr} & 0 \arrow{dr}& \\
			V\{-3\} \arrow{ur} \arrow{r} \arrow{dr} &\mathbb{F}_2\{-2\} \arrow{ur}\arrow{dr} & 0 \arrow{r} & 0\\
			&  \mathbb{F}_2\{-2\} \arrow{r}\arrow{ur}	& V\{-1\} \arrow{ur} &
		\end{tikzcd}
		
	\end{center}
The chain complex for $\alpha_0$ is shown above, where the arrows represent the same maps as in the case of $\alpha_{\textit{APS}}$. The Poincar\'e polynomial is  \[p(\widetilde{\mathit{Kh}}^{\alpha_{0}}_{\bullet}) = t^{-2}q^{-4}+t^{-1}q^{-2}+1. \]	

\begin{center}
	
	\begin{tikzcd}[row sep=tiny]
		& 0 \arrow{r}\arrow{dr} & \mathbb{F}_2\{-1\} \arrow{dr}& \\
		0 \arrow{ur} \arrow{r} \arrow{dr} &0 \arrow{ur}\arrow{dr} & \mathbb{F}_2\{-1\} \arrow{r} & V\\
		&  0 \arrow{r}\arrow{ur}	& 0 \arrow{ur} &
	\end{tikzcd}
	
\end{center}
The chain complex for $\alpha_1$ is as above. Note that $\widetilde{\mathit{CKh}}^{\alpha_{\textit{APS}}}_{\bullet}(L) = \widetilde{\mathit{CKh}}^{\alpha_0}_{\bullet}(L)\oplus \widetilde{\mathit{CKh}}^{\alpha_1}_{\bullet}(L)$, as the maps $f=g=0$. The Poincar\'e polynomial is  \[ p(\widetilde{\mathit{Kh}}^{\alpha_{1}}_{\bullet})=q^{-1}+tq.\]

	\item
	\label{ahf} 
	$\alpha_{\textit{HF}} = (W,\overline{V},f,g)$. $W= \langle a,b,c,d\rangle$ is the span of four elements $a,b,c,d$, with $qdeg (a)=1$, $qdeg(b) = qdeg(c)=0$ and $qdeg(d)=-1$. $\overline{V} = \langle \overline{v}_+,\overline{v}_-\rangle$ is the span of two elements $\overline{v}_+,\overline{v}_-$, with $qdeg(\overline{v}_+)=1$ and $qdeg(\overline{v}_-)=-1$. $f$ and $g$ are defined as follows.
	\begin{align*}
		f(a)=f(d)=0,\,\,\,\,\,\,\, &f(b)=f(c)=\overline{v}_- \\
		g(\overline{v}_-) = 0,\,\,\,\,\,\,&g(\overline{v}_+)=b+c
	\end{align*}
    This one will appear later in the discussion of the Heegaard Floer homology of the branched double cover of $\rp{3}$.
    \begin{center}
    	
    	\begin{tikzcd}[row sep=tiny]
    		& W \{-2\}\arrow{r}\arrow{dr} & \overline{V} \{-1\} \arrow{dr}& \\
    		W\otimes V\{-3\} \arrow{ur} \arrow{r} \arrow{dr} &W\{-2\} \arrow{ur}\arrow{dr} & \overline{V}\{-1\} \arrow{r} & \overline{V}\otimes V\\
    		&  W\{-2\} \arrow{r}\arrow{ur}	& W \otimes V\{-1\} \arrow{ur} &
    	\end{tikzcd}  	
    \end{center}
     The chain complex is shown above, with the obvious maps on each arrow. The Poincar\'e polynomial is \[ p(\widetilde{\mathit{Kh}}^{\alpha_{\textit{HF}}}_{\bullet})=t^{-2}(q^{-5}+ 2q^{-4}+q^{-3}) + t^{-1}(q^{-3}+ q^{-2}+q^{-1}) + q^{-1}+2+q+tq^2.\] 
     
     In contrast, if we consider $\alpha'_{\textit{HF}} = (W,\overline{V},0,0)$, changing both $f$ and $g$ to $0$, then Euler characteristic stays the same, while the Poincar\'e polynomial becomes \[p(\widetilde{\mathit{Kh}}^{\alpha'_{\textit{HF}}}_{\bullet})=t^{-2}(q^{-5}+2q^{-4}+q^{-3}) + t^{-1}(q^{-3}+2q^{-2}+q^{-1}) + q^{-2}+q^{-1}+3+q+t(1+q^2). \]
    
\end{enumerate}

\subsection{Other links in $\rp{3}$}
\label{section2.7}

 We discuss briefly what happens to other links in $\rp{3}$, those that are non-vanishing in $H_1(\rp{3},\mathbb{Z})$, if we perform a similar construction in this subsection. Suppose $K$ is an oriented link in $\rp{3}$ such that $\left[K\right]\neq 0$ in $H_1(\rp{3},\mathbb{Z})$. Consider its link projection $L$ in $\rp{2}$ with $n$ crossings. Now, for each $s\in \{0,1\}^n$, the smoothing $L_s$ of $L$ is a disjoint union of several null homologous circles and one special circle which generates $H_1(\rp{2},\mathbb{Z})$. Each null homologous circle divides $\rp{2}$ into a disk and M\"obius band as before, while cutting $\rp{2}$ along the special circle gives a disk. We can still define $e_s(P) \in \{0,1\}$ by picking a point $P$ in the complement of $L$ in $\rp{2}$ and count the number of null homologous circles encircling $P$ mod $2$. We can assign $V$ to each trivial circle, and $V_{e_s(P)}$ to the special circle. 
 
 In terms of bifurcations, there is no $1\rightarrow1$ bifurcation because of the presence of the special circle. On the other hand, there are two kinds of $2\rightarrow1$ bifurcations, which correspond to merging of two trivial circles and merging of a trivial circle with the special circle respectively. If the $2\rightarrow1$ bifurcation from $L_s$ to $L_{s'}$ merges two trivial circles, then $e_s(P) = e_{s'}(P)$ as before, and the corresponding differential map is the same as the multiplication $m: V\otimes V \rightarrow V$ in the usual Khovanov homology.  If the $2\rightarrow1$ bifurcation merges a trivial circle with the special circle, then the relation between $e_s(P)$ and $e_{s'}(P)$ depends on the specific position of $P$ and the trivial circle $P$: 
 \begin{enumerate}
 	\item If $P$ is not encircled by the trivial circle, then $e_{s'}(P) = e_s(P)$, and the corresponding differential map is the multiplication $m: V_{e_s(P)} \otimes V \rightarrow V_{e_s(P)}$, where $V_{e_s(P)}$ has the structure of a trivial $V$-module as before. 
 	\item If $P$ is encircled by the trivial circle, then $e_{s'}(P) = e_s(P) +1 $ mod $2$. We denote the corresponding differential map by $f_m : V_0 \otimes V \rightarrow V_1$ if $e_s(P)=0$, and by $g_m: V_1 \otimes V \rightarrow V_0$ if $ e_s(P)=1$.
 \end{enumerate}
The situation of $1\rightarrow2$ bifurcations is similar. If the $1\rightarrow2$ bifurcation splits a trivial circle, then $e_s(P) = e_{s'}(P)$, and the differential map is the same as the comultiplication $\Delta: V\rightarrow V \otimes V$ in the usual Khovanov homology. If the $1\rightarrow2$ bifurcation splits the special circle into a trivial circle and the new special circle, then again we divide into two cases.

 \begin{enumerate}
 	\item If $P$ is not encircled by the trivial circle, then $e_s(P) = e_{s'}(P)$, and the corresponding differential map is the comultiplication $\Delta: V_{e_s(P)}\rightarrow V_{e_s(P)} \otimes V$, where $V_{e_s(P)}$ has the structure of a trivial $V$-comodule.
 	\item If $P$ is encircled by the trivial circle, then $ e_{s'}(P) = e_s(P) +1$ mod $2$. We denote the corresponding differential map by $f_{\Delta}: V_0 \rightarrow V_1\otimes V$ if $e_s(P)=0$, and by $g_{\Delta}:V_1 \rightarrow V_0 \otimes V$ if $e_s(P)=1$.  
 \end{enumerate}
 
Then, to show $d^2=0$, we use similar strategy by considering singular graphs in $\rp{2}$ with two singular points as in Figure \ref{fig:singular-curves-with-2-singularity}. This time configurations $(a)$ and $(b)$ are impossible. Configurations $(e)$ and $(f)$ are affine as before, so we are left with configurations $(c)$ and $(d)$. They put the requirements \[g_m \circ f_{\Delta} =0,\,\, f_m \circ g_{\Delta}=0,\,\, f_{\Delta} \circ g_m = 0,\,\, g_{\Delta} \circ f_m=0\] on the differential maps. One way to achieve these requirements is starting with a dyad $\alpha = (V_0,V_1,f,g)$, and let \[f_m = m \circ (f\otimes id_V), \,\, g_m = m \circ (g\otimes id_V), \,\, f_{\Delta} = \Delta \circ f,\,\, g_{\Delta} = \Delta \circ g.\]
It can be shown that the homology $\widetilde{Kh}^{\alpha}(L) $ of the chain complex is invariant under Reidemeister moves not crossing the point $P$ as before. Hence we have the following theorem. Recall that $C_P$ is the union of the point $*$ in $\rp{3} \backslash \{*\} = \rp{2}\tilde{\times}I$ and the fiber over $P$ in the twisted $I$-bundle over $P$. 

\begin{theorem}
	$\widetilde{Kh}^{\alpha}(L)$ is a link invariant for $L$ considered as a link inside $\rp{3} \backslash C_P$.
\end{theorem}

The issue is that the homology depends on the choice of $P$ in a subtle way. As $L$ is non-trivial in $H_1(\rp{3},\mathbb{Z})$, it doesn't bound a Seifert surface, and we no longer have the subdivision $\rp{2} \backslash L = R_0 \bigsqcup R_1$ as before. We can't make a canonical choice of $P$ for a link projection, so the homology is only defined for a link with a choice of an extra point $P$ outside its link projection in $\rp{2}$.

\begin{remark}
	 The space $\rp{3} \backslash C_P$ is homeomorphic to the solid torus, which is the twisted $I$-bundle over a M\"obius band. We can also view the solid torus as the trivial $I$-bundle over an annulus and consider the link projection to the annulus. This leads to another invariant called annular Khovanov homology. See \cite{MR2113902}, \cite{MR3035332}, and \cite{MR2728482}.
	
\end{remark}

 \section{Heegaard Floer homology of branched double covers of $\rp{3}$}

In \cite{MR2141852}, Ozsv\'ath and Szab\'o introduced a spectral sequence associated to a link $L \subset S^3$ converging to the Heegaard Floer homology $\widehat{HF}(\Sigma(S^3,L),\mathbb{F}_2)$ of the branched double cover $\Sigma(S^3,L)$, whose $E^2$ page consists of the reduced Khovanov homology of the mirror of $L$.  In this section, we are going to extend this construction to null homologous links in $\rp{3}$. We will obtain a spectral sequence converging to the Heegaard Floer homology $\widehat{HF}(\Sigma_0(\rp{3},L),\mathbb{F}_2)$ of the even branched double cover $\Sigma_0(\rp{3},L)$ of $\rp{3}$, whose $E^2$ page consists of the Khovanov-type homology $\widetilde{\mathit{Kh}}^{\alpha_{\textit{HF}}}_{\bullet}(m(L))$ of the mirror of $L$, with the dyad $\alpha_{\textit{HF}} = (W,\overline{V},f,g)$ introduced in $(\ref{ahf})$ in section $1.6$.
\subsection{Branched double covers of $\rp{3}$}
$\label{sec2.1}$

For a link $K$ in a $3$-manifold $M$, the branched double covers $ \Sigma_h(M,K) $ are classified by the set of maps\[\left\{ h :\pi_1(M\backslash K) \longmapsto \mathbb{F}_2 \,\, \vert \,\, h([m_i] )=1\right\},\]
where $m_i$ is the meridian of the $i$th component of $K$. It is the same as the set of maps \[\left\{ h :H_1(M\backslash K,\mathbb{Z}) \longmapsto \mathbb{F}_2 \,\, \vert \,\, h([m_i] )=1\right\},\] as $\mathbb{F}_2$ is abelian. For the purpose of this paper, we will discuss the branched double covers $\Sigma(\rp{3}, K)$ when $K$ is a null homologous link in $\rp{3}$. Let's compute $H_1(\rp{3}\backslash K,\mathbb{Z})$ for these $K$ first.

\begin{lemma}
	
	Let $K$ be a null homologous link in $\rp{3}$ with $n$ component. Then, 
	we have 
	\begin{equation*}
	H_1(\rp{3}\backslash K,\mathbb{Z}) =
		\begin{cases}
		\mathbb{Z}^n \oplus\mathbb{Z}/2, &\text{ if each component of }K \text{ is null homologous} ,\\
		\mathbb{Z}^n, & \text{otherwise.  } 
		\end{cases}       
	\end{equation*} 
Define $M =\langle [m_1],...,[m_n] \rangle$ as the submodule of $H_1(\rp{3}\backslash K,\mathbb{Z})$ generated by the meridians of each component of $K$. Then in both cases, $M$ is a submodule of index $2$ in $H_1(\rp{3}\backslash K,\mathbb{Z})$. In particular, there are two branched double covers $\Sigma_h(\rp{3},K)$, determined by $h([l])$ for some $[l]\in H_1(\rp{3}\backslash L,\mathbb{Z})$ which is not in $M$.

\end{lemma} 
\begin{proof}
	Consider the Heegaard splitting $\rp{3} = U_1 \cup_f U_2$, where each $U_i$ is a solid torus, and $f:\partial U_1 \longmapsto \partial U_2$ is the map sending the meridian $\mu_1$ of $U_1$ to $-\mu_2 + 2l_2$, where $\mu_2$ and $l_2$ are the meridian and longitude of $U_2$. By isotopy, we can assume that $K$ lies entirely in $U_2$. Now, we can construct $\rp{3} \backslash K$ by gluing $U_1$ to $U_2 \backslash K$ along $f$. For $U_2 \backslash K$, we compute its homology in a way similar to computing the Wirtinger presentation of $\pi_1(S^3 \backslash K)$. View $U_2=A \times I$ as the trivial $I-$bundle over an annulus $A$, and consider the link projection of $K$  to $A$. We can build $U_2\backslash K$ starting from $A \times I$, gluing a tube $S^1 \times I$ for each arc in the projection, then a disk for each crossing in the projection and finally a $B^3$. See for example Chapter 11 in \cite{lickorish2012introduction} or problem 22 in Section 1.2 in \cite{Hatcher:478079} for a detailed description. The only difference is that we start with a solid torus $A\times I$ instead of a ball $B^3$. From this cell decomposition, we have $H_1(U_2 \backslash K,\mathbb{Z}) = \mathbb{Z}^{n+1} =\langle [m_1], [m_2],...,[m_n], [l_2] \rangle $, generated by the meridians $m_i$ of each component of $L$ and the longitude $l_2$ of $U_2$. Now gluing $U_1$ along $f$ adds the relation $-[\mu_2]+2[l_2]$ to $H_1(U_2\backslash K,\mathbb{Z})$, so we need to express $[\mu_2]$ in terms of $[m_i]$. For each component $K_i$ of $K$, we have $[K_i] = c_i[l_2]$ in $H_1(U_2,\mathbb{Z})$ for some $c_i$. View $U_2 = D^2 \times S^1$ as the trivial disk bundle over $S^1$, then the algebraic intersection number of $K_i$ with a generic fiber $D^2$ is $c_i$. So the punctured disk $D^2 \backslash (D^2 \cap K)$ gives a relation \[ [\mu_2] = \sum _{i=1}^n c_i [m_i]\] in $H_1(U_2 \backslash K,\mathbb{Z})$, and \[H_1 (\rp{3} \backslash K,\mathbb{Z}) = \langle [m_1], [m_2],...,[m_n], [l_2] \rangle/  \langle -\sum_{i=1}^n c_i[m_i]+2[l_2] \rangle .\] Now we discuss the cases whether each component of $K$ is null homologous or not.
	\begin{enumerate}
		\item Suppose each component $K_i$ of $K$ is null homologous, then all the $c_i$s are even integers, so \[H_1 (\rp{3} \backslash K,\mathbb{Z})  = \mathbb{Z}^n \oplus \mathbb{Z}_2,\] and $M = \langle [m_1], [m_2],...,[m_n] \rangle$ is a submodule of index $2$.
		\item Suppose there is some component of $L$ which generates $H_1(\rp{3},\mathbb{Z})$, then there is an even number of $i$ such that $c_i$ is odd, so \[ H_1 (\rp{3} \backslash K,\mathbb{Z})  = \mathbb{Z}^n\] by some change of variables, and $M$ is again a submodule of index 2. 
	\end{enumerate} 
From the description of $ H_1 (\rp{3} \backslash K,\mathbb{Z})$, it is easy to see there are two maps $h: H_1(\rp{3}\backslash K, \mathbb{Z}) \longmapsto \mathbb{F}_2$ in each case that send all the meridians $[m_i]$ to 1, so there are two branched double covers $\Sigma_h(\rp{3},K)$ in each case. If we take some circle $l \in \rp{3}\backslash K$ such that $[l]\in H_1(\rp{3}\backslash K)$ is not in the submodule $M$ (i.e. $l$ won't be null homologous in $\rp{3}$ if we fill in $K$), then the map $h:H_1(\rp{3}\backslash K, \mathbb{Z})$, hence the branched double cover $\Sigma_h(\rp{3},K)$, is determined by $h([l])$, as $[l]$ together with the $[m_i]s$ generate $ H_1 (\rp{3} \backslash K,\mathbb{Z})$.  \end{proof}
\begin{remark}
	If $K$ is an oriented link in $\rp{3}$ which generates $H_1(\rp{3},\mathbb{Z})$, then we have $[\mu_2] = \sum_i c_i[m_i]$ such that $\sum_i ci$ is odd.
	 If $h:H_1(\rp{3}\backslash K,\mathbb{Z}) \longmapsto \mathbb{F}_2$ were a map such that $h([m_i])=1$ for all $i$, then we would have \[2h([l_2]) =h(2[l_2]) = h(\sum_{i=1}^{n}c_i[m_i]) = \sum_{i=1}^{n}c_i,\] a contradiction. So there are no branched double cover of $\rp{3}$ over $K$ if $[K]$ generates $H_1(\rp{3},\mathbb{Z})$.
\end{remark}

Denote the projection from the branched double cover associated to $ h:H_1 (\rp{3} \backslash K,\mathbb{Z}) \longmapsto \mathbb{F}_2$ to $\rp{3}$ by $p_h:\Sigma_h(\rp{3},K) \longmapsto \rp{3}$. The condition $h([l])=0$ is equivalent to requiring the preimage $p_h^{-1}(l)$ is a disjoint union of two circles instead of one large circle. 

\begin{definition}
	Let $K$ be a null homologous link in $\rp{3}$, and $l$ be a circle in $\rp{3} \backslash K$ such that $[l]$ is not in the submodule $M\subset H_1 (\rp{3}\backslash K,\mathbb{Z})$. We define $\Sigma(\rp{3}, K, l)$ as the branched double cover $\Sigma_h(\rp{3},K)$ such that $h([l])=0$.
\end{definition}

 For a null homologous link $K$ in $\rp{3}$, recall we divided the complement of its projection $L$ in $\rp{2}$ into two regions $\rp{2}\backslash L= R_0 \sqcup R_1$, picked a point $P$ in $\rp{2}\backslash L$, and defined a circle $C_P$ associated to it. See the discussion below Definition $\ref{def1.3.3}$ and the Equation $\ref{eq1.3.0.1}$. $C_P$ is such a circle that $[C_P]\in H_1(\rp{3}\backslash K,\mathbb{Z})$ is not in the submodule $M$.  

\begin{definition}
	\label{def2.1.3}
	For a  null homologous link $K$ in $\rp{3}$, we will call $\Sigma( \rp{3},K,C_P)$ the \textbf{even branched double cover} of $\rp{3}$ over $K$ if $P\in R_0$, denoted by $\Sigma_0(\rp{3},K)$, and the \textbf{odd branched double cover} where if $P\in R_1$, denoted by $\Sigma_1{(\rp{3},K)}$.
\end{definition}

 Note that it is well defined, as for a different choice $P'\in R_0$,  $[C_P]-[C_P']$ is a sum of even number of meridians $[m_i]$ in $H_1 (\rp{3} \backslash K,\mathbb{Z})$ by the definition of $R_0$, so $h([C_P]) = h([C_P'])$.  
 
 Now we describe all the other branched double covers for the smoothings $L_s$ of the knot projection $L$ of $K$.
 
 \begin{lemma}
 	\label{lem2.1.3}
 	Let $s\in \{0,1\}^n$ be a state for the link projection $L$ with $n$ crossings, and $L_s$ be the corresponding smoothing of $L$. Then the branched double cover   \begin{equation*}
 		\Sigma(\rp{3},L_s,C_P) =
 		\begin{cases}
 			(\rp{3}\#\rp{3})\#(S^1\times S^2)^{\#(k_s-1)} &\text{ if } e_s(P)=0,\\
 			(S^1\times S^2)^{\#k_s} & \text{ if } e_s(P)=1,
 		\end{cases}       
 	\end{equation*} 
 where $k_s$ is the number of circles in $L_s$, and $e_s(P)$ is the number of circles in $L_s$ encircling $P$ mod 2.
 \end{lemma}  
\begin{proof}
	 If $e_s(P)=0$, then pick a point $P'\in \rp{2} \backslash L_s$ which is encircled by none of the circles in $L_s$, so $e_s(P')=0$ as well. Consider a path $\gamma $ in $\rp{2}$ connecting $P$ and $P'$ such that $\gamma$ intersects $L_s$ transversely. Let $c_i$ be the algebraic intersection number of $\gamma$ with the $i$th circle in $L_s$, then \[[C_{P'}] = [C_P] + \sum_{i=1 }^{k_s}c_i[m_i]  \] in $H_1(\rp{3}\backslash L_s, \mathbb{Z})$. Note that $e_s(P')-e_s (P) = \sum_i c_i$ mod $2$ by the definition of $e_s$, so $\sum_i c_i =0 $ mod 2. This implies $h([C_{P'}]) = h([C_P])=0$ as $h([m_i])=1$, so $\Sigma(\rp{3},L_s,C_P) = \Sigma(\rp{3},L_s,C_{P'})$, and we will look for $\Sigma(\rp{3},L_s,C_{P'})$. 
	 
	 Now we find the branched double cover $\Sigma(\rp{3}, L_s,C_{P'})$ by induction on $k_s$. When $k_s=0$, the link $L_s$ is empty, and $\Sigma(\rp{3},\emptyset,C_{P'})$ is the double cover $\rp{3}\sqcup \rp{3}$, as the preimage $p_h^{-1}(C_{P'})$ is a disjoint union of two circles. When $k_s=1$, because $P'$ is not encircled by this circle, we can pick some ball $B^3 \subset \rp{3} \backslash C_{P'}$ containing this circle. The effect of adding one circle to the branching locus is the same as replacing the preimage $p_h^{-1}(B^3) = B^3 \sqcup B^3$ by the branched double cover of $B^3$ over the circle, which is $D^1 \times S^2$. It is the same as doing a $0$-surgery to $\rp{3}\sqcup \rp{3}$ with one $B^3$ in each copy of $\rp{3}$. So we get $\Sigma(\rp{3}, L_s,C_{P'})=\rp{3}\# \rp{3}$ for $k_s=1$. Note that in the process of $0$-surgery, we didn't change the preimage of $C_{P'}$ at all, which stays as a disjoint union of two circles. For each of the rest circles in  $L_s$, we apply a similar $0$-surgery avoiding the preimage of $C_{P'}$. For connected 3-manifolds, doing $0$-surgery is the same as taking connected sum with a copy of $S^1\times S^2$, so we obtain $\Sigma(\rp{3},L_s,C_{P'})=(\rp{3}\#\rp{3})\#(S^1\times S^2)^{\#(k_s-1)}.$
	 
	 If $e_s(P)=1$, the proof is very similar. The only difference is that if we pick some $P'\in \rp{2}\backslash L$ encircled by none of the circles in $L_s$ this time, then $e_s(P') -e_s(P)=1$ mod $2$, and $h([C_P']) = 1$ instead of $0$. Hence we start with the other double cover $S^3$ of $\rp{3}$ instead of $\rp{3}\sqcup \rp{3}$ for the empty knot, and we get  $\Sigma(\rp{3},L_s,C_{P'})=(S^1\times S^2)^{\#k_s}.$ 
\end{proof}

Now we review the relation between the branched double covers $\Sigma({\rp{3}, L,C_P})$, $\Sigma(\rp{3},L_0,C_P)$ and $\Sigma(\rp{3},L_1,C_P)$, where $L_0,L_1$ are the link projections obtained from $L$ by $0$ or $1$-smoothings at one crossing, as introduced in section 2 in \cite{MR2141852}. Note the convention they used for $0$ and $1$ smoothings is the reverse of that in this paper.

\begin{definition}
	 Let $M$ be an oriented $3$-manifold with torus boundary and three simple, closed curves $\alpha,\beta,\gamma $ in $\partial M$ with algebraic intersection numbers \[\#(\alpha \cap \beta) = \#(\beta\cap \gamma) =\#(\gamma\cap\alpha)=-1. \] A \textbf{triad} of $3$-manifolds $(Y_{\alpha},Y_{\beta},Y_{\gamma})$ is an ordered triple of $3$-manifolds, such that there exists an $M$ and $(\alpha,\beta,\gamma)$ as above such that $Y_i$ is obtained from $M$ by attaching a solid torus along the boundary with the meridian mapped to $i$, for $i = \alpha,\beta,\gamma$ respectively. 
\end{definition}

The following lemma is shown as 	 Proposition 2.1 in \cite{MR2141852}, except we switch $0$ and $1$ smoothings because of the different conventions.
\begin{lemma}
	\label{lem2.1.5}
	 The branched double covers $(\Sigma({\rp{3}, L},C_P),\Sigma(\rp{3},L_1,C_P),\Sigma(\rp{3},L_0,C_P))$ forms a triad of $3$-manifolds.
\end{lemma}
\begin{proof}
	Consider some ball $B^3$ containing the changed crossing of $L$ in $\rp{3}$. The branched double cover of $S^2$ branching over $4$ points is a torus, and the branched double covers of $B^3$ branching over $\backoverslash, \hsmoothing,\smoothing$ are all solid torus, with meridians obtained by double cover of one of the arcs pushed to the boundary of $B^3$. Therefore, they form a triad. Check \cite{MR2141852} for details. 
\end{proof}

We will give a more detailed account of the surgery from $\Sigma(\rp{3}, L_1,C_P)$ to $\Sigma(\rp{3},L_0,C_P)$ in Section $\ref{sec2.3}$.
\subsection{Link surgery spectral sequence of  Heegaard Floer Homology}
$\label{sec2.2}$

 In this section we briefly review the construction of link surgeries spectral sequence in Section $4$ of \cite{MR2141852}, which is a generalization of the surgery long exact sequence associated to a triad of $3$-manifolds $(Y_{\alpha}, Y_{\beta}, Y_{\gamma})$: \[...\longrightarrow\widehat{HF}(Y_{\alpha})\longrightarrow \widehat{HF}(Y_{\beta})\longrightarrow\widehat{HF}(Y_{\gamma}) \longrightarrow ...\]
 
 For the definition and basic facts about Heegaard Floer homology $\widehat{HF}$, see \cite{MR2113020}. We will use the Heegaard Floer homology with coefficients $\mathbb{F}_2$ throughout the paper, so we omit $\mathbb{F}_2$ from the notation.
 
 Let $L = K_1 \cup K_2\cup ...\cup K_n$ be an $n$-component, framed link ($i.e.$ with a choice of longitude $l_i$ for each component $K_i$) in a $3$-manifold $Y$. A \textbf{multi-framing} is an $n$-tuple $I = (s_1,...s_n)$, where each $s_i \in \{\alpha,\beta,\gamma\}$. (In \cite{MR2141852}, the corresponding labels for s are $ \{0,1,\infty\}$. Unfortunately, $0$ and $1$-surgery correspond to the $1$ and $0$-smoothing respectively, so we use $ \{\alpha,\beta,\gamma\}$ instead) For each multi-framing $I$, there is a three-manifold $Y(I)$, obtained from $Y$ by performing $s_i$-framed surgery on the component $K_i$ for $i=1,...,n$. Here, $\alpha $ means the $l_i$-framed surgery, $\beta$ means the $(l_i+m_i)$-framed surgery where $m_i$ is the meridian of $K_i$, and $\gamma$ means no surgery, which is the same as the $m_i$-framed surgery.
 
  We give the set $\{\alpha,\beta,\gamma\}^n$ the lexicographical order, such that $\alpha <\beta <\gamma$. For $I = (s_i)_{i=1...n},I'= (s'_i)_{i=1...n} \in   \{\alpha,\beta,\gamma\}^n$, we call $I'$ an an \textbf{immediate successor} of $I$ if they are only different at one slot $j$, such that $(s_j,s'_j) = (\alpha,\beta) $ or $(\beta, \gamma)$. 
  
  For a sequence of multiframings $I^1<...<I^k$, such that each $I^{i+1}$ is an immediate successor of $I^{i}$, there is an induced map on the Heegaard Floer chain complex \[D_{I^1<...<I^k}:\widehat{CF}(Y(I^1)) \longmapsto \widehat{CF}(Y(I^k))\] defined by counting holomorphic polygons. 
  
  Let $X= \bigoplus_{I\in \{\alpha,\beta,\gamma\}}\widehat{CF}(Y(I))$, endowed with a map $D:X \longmapsto X$ defined by \[D(\xi) = \sum_J \sum _{\{I=I^1 <...<I^k=J\}} D_{I^1<...<I^k}(\xi)\] summing over all sequences $I=I^1<...<I^k=J$  such that $I^{i+1}$ is an immediate successor of $I^i$, for $\xi \in \widehat{CF}(Y(I))$. We can filter $X$ by the lexicographical order, and write $D = \sum_{i=0}^n d_i$, where each $d_i$ is summing over sequences of length $i$ and $d_0$ is the sum of differentials in the chain complex $\widehat{CF}(Y(I))$ for each $I$.
  
  The main inputs from the theory of Heegaard Floer homology are the following two propositions.
  
  \begin{proposition}
  	(Proposition 4.6 in \cite{MR2141852}) The map $D$ satisfies $D^2=0$, so it is a differential map.
  \end{proposition}

\begin{proposition}
	(Theorem 4.7 in \cite{MR2141852})  Let $K$ be a framed knot in a $3$-manifold $Y$, and let \[\widehat{f}: \widehat{CF}(Y_0(K)) \longmapsto \widehat{CF}(Y_1(K))\] denote the chain map induced by the cobordism of adding a two-handle. Then the chain complex $\widehat{CF}(Y)$ is quasi-isomorphic to the mapping cone of $\widehat{f}$.
\end{proposition}

  Then by applying induction on $n$, together with an algebraic lemma on mapping cones (Lemma 4.4 in \cite{MR2141852}), one can prove the following theorem. 
  
  \begin{theorem}
  	\label{th2.2.3}
  	(Theorem 4.1 in \cite{MR2141852}) There is a spectral sequence whose $E^1$ term is \[ \bigoplus _{I \in\{\alpha,\beta\}^n} \widehat{HF}(Y(I)),\] which converges to $\widehat{HF}(Y)$, such that the $d_1$ map is given by adding maps $\widehat{f}_{*}: \widehat{HF}(Y(I)) \longmapsto\widehat{HF}(Y(I')) $ together for all pairs $(I,I')$ where $I'$ is an immediate successor of $I$.
  \end{theorem}
  
 Now we discuss the relation of this with the reduced Khovanov homology of links in $S^3$. Check Section 5 and 6 in \cite{MR2141852} for details. Consider the link projection $L\in \mathbb{R}^2$ of a link $K$ in $S^3$ with $n$ crossings. Let $Y= \Sigma(S^3,K)$ be the branched double cover of $S^3$ over K. For the $i$th crossing of $L$, pick the vertical arc in $S^3$ connecting the two double points, and let $l_i$ be the preimage of this arc in $Y$. Each $l$ is framed such that the $1$-surgery of $Y$ on $l$ gives the branched double cover $\Sigma(S^3,K_0)$, where $K_0$ is obtained from $K$ by changing the crossing with the $0$-smoothing. For each state $I\in \{\alpha,\beta\}^n$, we consider the corresponding state $s=\phi(I) \in \{1,0\}^n$, where $\phi$ sends $\alpha$ to $1$ and $\beta$ to $0$. Then the manifold $Y(I)$ obtained from $Y$ be surgery according to the multiframing $I$ is the same as the branched double cover $\Sigma(S^3,L_s)$ for the state $s = \phi(I)$, by the local analysis for small balls $B^3$ in $S^3$ containing the crossings as in Lemma $\ref{lem2.1.5}$. 
 
 Each $L_s$ is an unlink with $k_s$ components, so the branched double cover $\Sigma(S^3,L_s) =Y(I) $ is a connected sum of $k_s-1$ copies of $S^1\times S^2$, and $\widehat{HF}(Y(I)) = V^{\otimes (k_s-1)} $, where $V = \langle v_+,v_- \rangle$ is the same as before. Hence we have an identification between $\widehat{HF}(Y(I))$ and $\widetilde{CKh}_{\phi(I)}(m(L))$, where $m(L)$ is the mirror of $L$. The mirror $m$ appears as the map $\phi$ sends $\alpha $ to $1$ and $\beta $ to $0$, so we turn all the crossings upside down to make it consistent with the usual convention for the Khovanov homology, which gives the Khovanov homology for $m(L)$ instead of $L$. This identification is natural in the sense that it turns the map $d_1$ in $\bigoplus _{I \in\{\alpha,\beta\}^n}\widehat{HF}(Y(I))$ to the differential $d$ in $\widetilde{CKh}(m(L))$. 
  Together with Theorem $\ref{th2.2.3}$,  we obtain the following result:
 
 \begin{theorem}
 	(Theorem 1.1 in \cite{MR2141852}) Let $L \subset S^3$  be a link. There is a spectral sequence whose $E^2$ terms consists of the reduced Khovanov homology of the mirror of $L$ with coefficients in $\mathbb{F}_2$, which converges to $\widehat{HF}(\Sigma(S^3,L),\mathbb{F}_2)$.
 	\label{theorem2.2.4}
 \end{theorem} 
  Our task in the next section is to generalize this result for null homologous links in $\rp{3}$, where the reduced Khovanov homology is replaced by the Khovanov-type homology $\widetilde{Kh}^{\alpha_{\textit{HF}}}(L)$ from Definition $\ref{def1.1.3}$, such that $\alpha_{\textit{HF}}$ is as in $(\ref{ahf})$ in Section 1.6.
\subsection{The differential $d_1$ in the spectral sequence for a branched double cover of $\rp{3}$}
$\label{sec2.3}$
By the discussion in Lemma $\ref{lem2.1.3}$, we have an identification of $\widehat{HF}(Y(I))$ with $\widetilde{CKh}^{a_{HF}}_{\phi(I)}(L)$ as vector spaces. We are left to compute $d_1$ in the spectral sequence to show it coincides with the corresponding edge maps in $\widetilde{CKh}^{\alpha_{\textit{HF}}}$. For $1\rightarrow2$ and $2\rightarrow1$ bifurcations, the same proof in \cite{MR2141852} for links in $S^3$ works, which we will review briefly. The case of $1\rightarrow1$ bifurcations needs a detailed description of the $1$-surgery involved.  

As discussed in Section 3.1, there are two branched double covers $\Sigma(\rp{3},K)$ for a null homologous link. Let's fix which branched cover we will use first. For the link projection $L$ of a null homologous link $K$ in $\rp{3}$, we will consider the even branched double cover $\Sigma(\rp{3},L,C_P)$ by picking $P\in R_0$, see Definition $\ref{def2.1.3}$. For each smoothing $L_s$, we will consider the branched double cover $\Sigma(\rp{3},L_s,C_P)$ obtained from $\Sigma(\rp{3},L,C_P)$ by doing $1$-surgery, as described in Lemma $\ref{lem2.1.5}$. Note that these branched double covers don't depend on the specific choice of $P$ as long as $P\in R_0$. By Lemma $\ref{lem2.1.3}$, we have
\begin{equation*}
	\Sigma(\rp{3},L_s,C_P) =
	\begin{cases}
		(\rp{3}\#\rp{3})\#(S^1\times S^2)^{\#(k_s-1)} &\text{ if } e_s(P)=0,\\
		(S^1\times S^2)^{\#k_s} & \text{ if } e_s(P)=1,
	\end{cases}       
\end{equation*}  
where $k_s$ is the number of circles in the smoothing $L_s$. Hence, the corresponding Heegaard Floer homology is  \begin{equation*}
	\widehat{HF}(\Sigma(\rp{3},L_s,C_P) )=
	\begin{cases}
		W\otimes V^{\otimes (k_s-1)} &\text{ if } e_s(P)=0,\\
		\overline{V} \otimes V^{\otimes (k_s-1)} & \text{ if } e_s(P)=1,
	\end{cases}       
\end{equation*} 
where $W = \widehat{HF}(\rp{3}\#\rp{3}) = \langle a,b,c,d \rangle$, $V = \widehat{HF}(S^1\times S^2) = \langle \un,\coun\rangle$ and  $\overline{V} = \widehat{HF}(S^1\times S^2) = \langle \overline{v}_+,\overline{v}_-\rangle$. See for example Chapter $4$ and $7$ in \cite{MR1957829} for the computation of $\widehat{HF}(\rp{3}\#\rp{3})$ and its absolute grading.  Here, $\overline{V}$ and $V$ are the same vector space. The reason we want to distinguish them will become clear in the proof of next proposition. Note that the quantum gradings of the generators in Heegaard Floer homology are the twice the absolute gradings of them. Also note that $\widehat{HF}(\Sigma(\rp{3},L_s,C_P) )$ is exactly the vector space we associate to the state $s$ in our chain complex $\widetilde{CKh}^{a_{HF}}_{s}(L)$.

\begin{proposition}
	\label{prop 2.3.1}
	(Proposition 6.2 in \cite{MR2141852}) For the link projection $L$ of a null homologous link in $\rp{3}$ and a point $P\in R_0$, there is an isomorphism \[\Psi_s:\widetilde{CKh}^{a_{\textit{HF}}}_{s}(L) \longmapsto \widehat{HF}(\Sigma(\rp{3},L_s,C_P) ) \] for each state $s$, such that the following diagram commutes if $s\rightarrow s'$ is a $1\rightarrow2$ bifurcation or $2\rightarrow1$ bifurcation:
	\begin{center}
	\begin{tikzcd}
		\widetilde{\mathit{CKh}}^{\alpha_{\textit{HF}}}_{s}(L)  \arrow[r,"d_{CKh}"] \arrow{d}{\Psi_s}
		& \widetilde{\mathit{CKh}}^{\alpha_{\textit{HF}}}_{s'}(L) \arrow[d,"\Psi_{s'}"] \\
		\widehat{HF}(\Sigma(\rp{3},L_s,C_P) ) \arrow{r}{d_{HF}}
		& \widehat{HF}(\Sigma(\rp{3},L_{s'},C_P) )
	\end{tikzcd} 

	\end{center}
Here $d_{CKh}$ is the differential defined for the chain complex $\widetilde{\mathit{CKh}}^{\alpha_{\textit{HF}}}_{\bullet}(L) $ as in Section $\ref{sec1.2}$, and $d_{HF}$ is the map induced on $\widehat{HF}$ by $1$-surgery described in Lemma $\ref{lem2.1.5}$.	
	 
\end{proposition}
\begin{proof}
	As in Section 5 in \cite{MR2141852}, we have described an algebra structure on $V^{\otimes (k_s-1)}$ as the quotient of the polynomial algebra over $\mathbb{F}_2$ generated by $S_i$, divided out by the relations $S_i^2=0$ for $i=1,...,k_s-1$, where $S_i=\un\otimes \un\otimes...\otimes \coun \otimes \un\otimes...\otimes \un$ with $\coun$ at the $i$th component in the proof of Proposition \ref{prop1.3.5}. So $\widetilde{\mathit{CKh}}^{\alpha_{\textit{HF}}}_{s}(L)$ is a free module of this algebra generated by elements in $V_i$, where $V_i = W$ if $e_s(P)=0$ and $V_i = \overline{V}$ if $e_s(P)=1$. On the other hand, $ \widehat{HF}(\Sigma(\rp{3},L_s,C_P) )$ is a also free module over the algebra \[ \wedge^*H_1((S^1\times S^2) ^{\#(k_s-1)},\mathbb{Z}_2) \cong V^{\otimes (k_s-1)},\] generated by elements in $V_i$, where again $V_i = W = \widehat{HF}(\rp{3}\#\rp{3})$ if $e_s(P)=0$, and $V_i=\overline{V} = \widehat{HF}(S^1\times S^2)$ if $e_s(P)=1$. Furthermore, we can identify generators of $H_1((S^1\times S^2) ^{\#(k_s-1)},\mathbb{Z}_2)$ with circles in $L_s$ as follows. Recall we have picked a point $M$ on the link projection $L$ when defining the differential. Label the circles in the smoothing $L_s$ from $0$ to $k_s-1$, such that the circle with $M$ is labeled $0$. Let $\gamma_i$ be the branched double cover of an arc from the circle $0$ to the circle $i$ in $\rp{2}$ avoiding other circles. Then $\{[\gamma_i]\}_{i=1}^{k_s-1}$ is a basis of $H_1((S^1\times S^2) ^{\#(k_s-1)},\mathbb{Z}_2)$. We define \[\xi_s:\wedge^*H_1((S^1\times S^2) ^{\#(k_s-1)},\mathbb{Z}_2) \longmapsto V^{\otimes (k_s-1)}\]  as the algebra automorphism, such that $\xi_s([\gamma_i]) = S_i$, and $\Psi_s$ the corresponding module isomorphism, which is identity on $V_i$. 
	
	Now we show commutativity of the diagram for $1\rightarrow2$ bifurcations and $2\rightarrow1$ bifurcations. Note that $d_{CKh}$ acts trivially on the $V_i$ component in this case. On the  $V^{\otimes (k_s-1)}$ component, we have describes the behavior $d_{CKh}$ in this formalism in the proof of Proposition $\ref{prop1.3.5}$ as follows, which is from Section 5 of \cite{MR2141852}: \begin{equation*}
		d_{CKh} =
		\begin{cases}
			\text{quotienting out } S_i=S_j &\text{ if } s\rightarrow s' \text{ merges circle }i \text{ with circle }j,\\
			\text{wedging with } S_i+S_{k_s}  & \text{ if } s\rightarrow s' \text{ splits circle }i,
		\end{cases}       
	\end{equation*} 
where $k_s$ is the label of the newly created circle in the splitting case, and we denote $0=S_0$ for the convenience of the notation. 
On the Heegaard Floer homology side, we have exactly the same map via identifying $\widehat{HF}(\Sigma(\rp{3},L_s,C_P))$ as  the free module over $V^{\otimes (k_s-1)}$ under the isomorphism $\xi_s$. See Proposition 6.1 in \cite{MR2141852} for a more detailed discussion of $d_{HF}$. Hence the square commutes.
\end{proof}

\begin{remark}
	With the map $\xi_s$, it is quite natural to write down the change of variable map in Proposition $\ref{prop1.3.5}$ for changing the position of the marked point.
\end{remark}
Now we study the $1\rightarrow1$ bifurcation. We will draw specific Kirby diagrams for the $1$-surgery, and compute the induced map on $\widehat{HF}$ using a proposition in \cite{MR1957829}.

If $s\rightarrow s'$  if a $1\rightarrow1$ bifurcation, then $e_{s'}(P) = e_s(p) +1$ mod $2$, and $k_s = k_{s'}$, so we switch from one branched double cover of $\rp{3}$ over an unlink of $k_s$ components to the other, while no change happens to the common $(S^1\times S^2)^{\#(k_s-1)}$ component, nor the identification of generators of $H_1((S^1\times S^2) ^{\#(k_s-1)},\mathbb{Z}_2)$ with the circles in the smoothing. So it is enough to consider the case when $k_s=1$.

\begin{lemma}
	The Kirby diagrams of the $1$-surgery corresponding to the $1\rightarrow1$ bifurcation are shown in Figure $\ref{fig:kirby}$, where $(a)$ is the case when $e_s(P)=0$ and $(b)$ is the case when $e_s(P)=1$, and the red curve is the one we are doing surgery on. 
\end{lemma}

\begin{figure}[t]
	\[
	{
		\fontsize{8pt}{10pt}\selectfont
		\def\svgwidth{4.5in}
		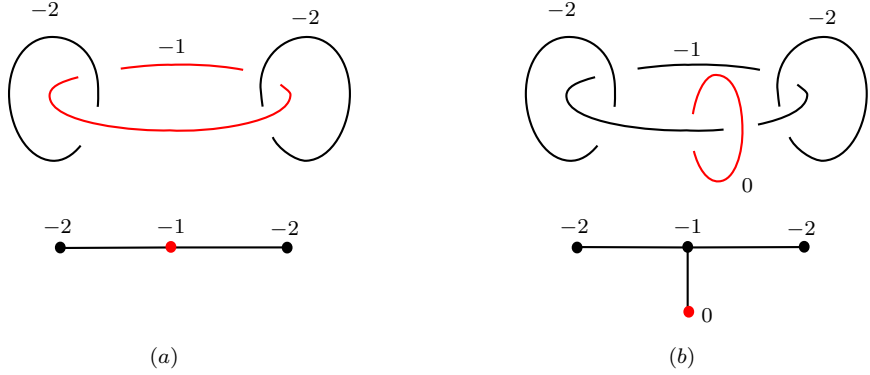
	}
	\]
	\caption{Kirby diagrams of  $1$-surgeries associated to $1\rightarrow1$ bifurcations}
	\label{fig:kirby}
\end{figure}

\begin{figure}[t]
	\[
	{
		\fontsize{8pt}{10pt}\selectfont
		\def\svgwidth{4.5in}
		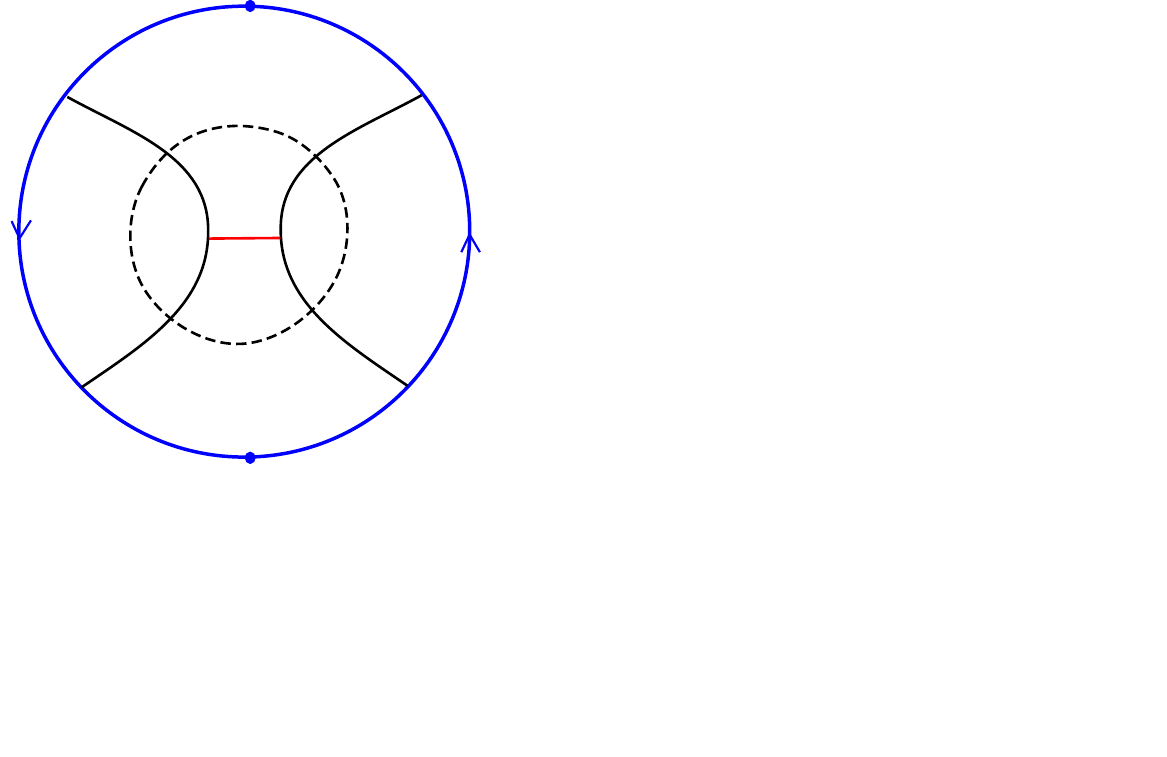
	}
	\]
	\caption{Explanation of Kirby diagram $(a)$}
	\label{fig:cobordism-a}
\end{figure}

\begin{figure}[t]
	\[
	{
		\fontsize{8pt}{10pt}\selectfont
		\def\svgwidth{4.5in}
		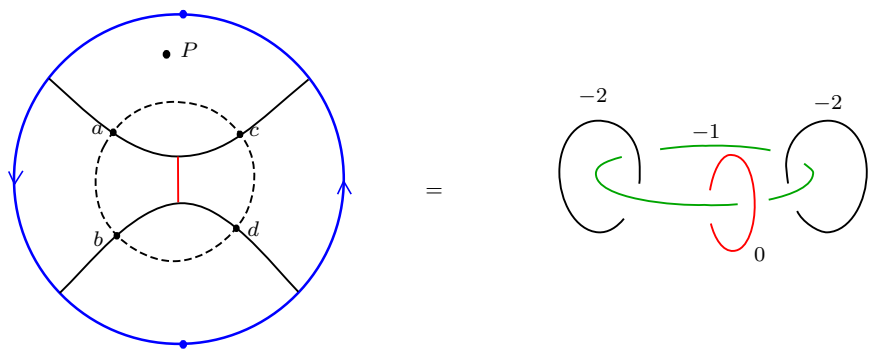
	}
	\]
	\caption{Explanation of Kirby diagram (b)}
	\label{fig:cobordism-b}
\end{figure}

\begin{proof}
	Suppose $e_s(P)=0$, then $e_{s'}(P)=1$, and the $1$-surgery changes the branched double cover from $\rp{3} \# \rp{3}$ to $S^1 \times S^2$. The link we are doing surgery along is the branched double cover of the red arc in Figure $\ref{fig:cobordism-a}$, which could be seen as the red circle in the Kirby diagram of $\rp{3} \# \rp{3}$. To determine the framing, we can look at the  $3$-manifold when it is with framing $-n$. Such obtained $3$-manifold is a lens space $L(p,q)$ with $p = 4n-4$, $q=2n-1$, where $p/q$ equals to the continued fraction $[2,n,2]$. We want to get $S^1\times S^2$, so $p=0$ and $n=1$. 
		
	Now if $e_s(P)=1$, then $e_{s'}(P)=0$ and the $1$-surgery changes the branched double cover from $S^1\times S^2$ to $\rp{3}\#\rp{3}$. Note that the red arc in Figure $\ref{fig:cobordism-b}$ is isotopic to a pushoff of the arc $ab$ to the boundary of $B^3$, so the branched double cover of the red arc is isotopic to a meridian of the green curve, which is the knot we are doing surgery along to get $S^1\times S^2$ from $\rp{3}\#\rp{3}$. To determine the framing, let's look at the following cobordism in Figure \ref{fig:framing-of-cobordism-b}. The compositions of the top and bottom two arrows both represent the cobordism from $\rp{3}\#\rp{3}$ to itself by doing surgeries along the branched double cover of the red and green arcs. The difference is that we add the two $2$-handles in different order. It is clear that in the bottom left cobordism from $\rp{3}\#\rp{3}$ to $\rp{3}\#\rp{3}\#(S^1\times S^2)$, the framing of the red circle is $0$. Therefore, the framing of the red circle is $0$ in the top right cobordism as well, which is the cobordism we are describing in Figure \ref{fig:cobordism-b}.
\end{proof}

\begin{figure}[t]
	\[
	{
		\fontsize{8pt}{10pt}\selectfont
		\def\svgwidth{4.5in}
		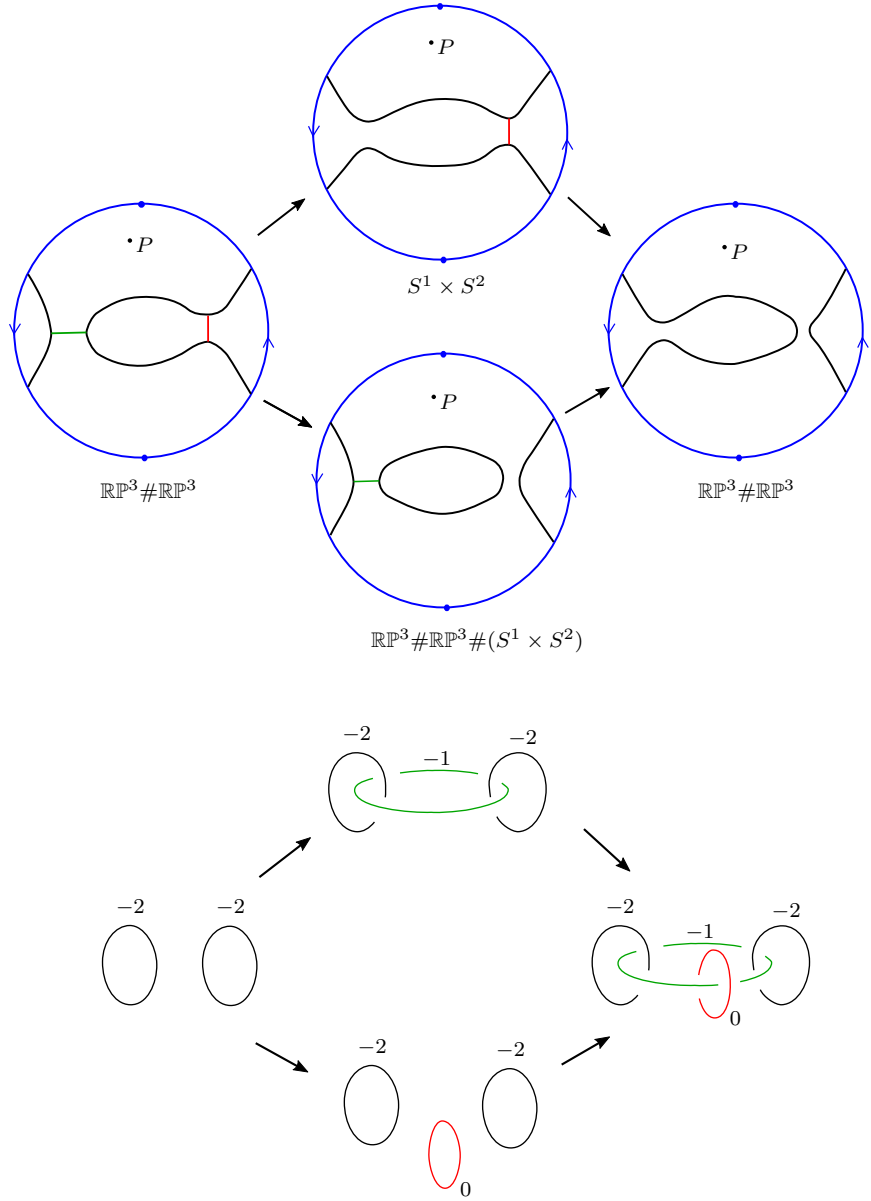
	}
	\]
	\caption{Two different decomposition of the cobordism from $\rp{3}\#\rp{3}$ to itself}
	\label{fig:framing-of-cobordism-b}
\end{figure}

Now to compute the induced map on $\widehat{HF}$ by the $1$-surgery, we quote the following proposition from \cite{MR1957829}.

\begin{proposition}
	\label{prop2.3.3}
	(Proposition 9.3 in \cite{MR1957829}) Let $Y$ be a closed oriented $3$-manifold, and $K \subset Y$ a framed knot, such that the cobordism $Z(K)$ of the $1$-surgery along $K$ has $b_2^{-}(Z(K)) =0 = b_2^{+}(Z(K))$. Let $\mathfrak{s}$ be a $\text{Spin}^c$ structure on $Z(K)$ whose restriction  $\mathfrak{t}$ and $\mathfrak{p}$ to the boundary components $Y$ and $Y(K)$ respectively are torsion. 
	
	\begin{enumerate}
		\item If $K$ represents a non-torsion class in $H_1(Y)$, and if $\widehat{HF}(Y,\mathfrak{t})$ is standard, then the induced map\[F_{Z(K),\mathfrak{s}}:\widehat{HF}(Y,\mathfrak{t})\longmapsto\widehat{HF}(Y(K),\mathfrak{p})\] vanishes on the kernel of the action by $[K]$, inducing an isomorphism \[\widehat{HF}(Y,\mathfrak{t})/\text{Ker}[K]\cong \widehat{HF}(Y(K),\mathfrak{p}).\]
		\item
		If $K$ represents a torsion class in $H_1(Y)$ and $\widehat{HF}(Y(K),\mathfrak{p})$ is standard, then the map \[\widehat{HF}(Y,\mathfrak{t}) \longmapsto \widehat{HF}(Y(K),\mathfrak{p})\] induces an isomorphism \[ \widehat{HF}(Y,\mathfrak{t}) \cong \text{Ker}[L],\] where $[L]\in H_1(Y(K))$ is represented by the core of the glued-in solid torus. 
	\end{enumerate}
\end{proposition}
Here, $\widehat{HF}(Y,\mathfrak{t})$ is $\textbf{standard}$ means that $\widehat{HF}(Y,\mathfrak{t}) \cong (\Lambda^bH^1(Y,\mathbb{Z}))\otimes \mathbb{F}_2$, where $b = b_1(Y)$. The original proposition is stated for $HF^{\infty}$, but it holds for $\widehat{HF}$ as well. When $\widehat{HF}(Y,\mathfrak{t})$ is standard, we have $HF^{\infty}(Y,\mathfrak{t}) = \widehat{HF}(Y,\mathfrak{t}) \otimes \mathbb{Z}[U,U^{-1}]$, and the cobordism map is $U$-equivariant, so the statement for $HF^{\infty}$ implies the one for $\widehat{HF}$.

 Note that $\widehat{HF}(Y,\mathfrak{t})$ is standard for $Y = \rp{3}\#\rp{3}$ and $S^1\times S^2$ and all torsion $\text{Spin}^c$ structures of them. Now we compute the induced map on $\widehat{HF}$ of the surgery corresponding to $1\rightarrow1$ bifurcations. Recall we have used the notation  \[\widehat{HF}(\rp{3}\#\rp{3}) = W =\langle a,b,c,d \rangle,\] where $a,b,c,d$ are generators corresponding to different torsion $\text{Spin}^c$ structures on $\rp{3}\#\rp{3}$, with absolute gradings $1/2,0,0,-1/2$ respectively, and \[\widehat{HF}(S^1\times S^2) = \overline{V} =\langle \overline{v}_+,\overline{v}_- \rangle,\] where $\overline{v}_+$ and $\overline{v}_-$ are generators corresponding to the torsion $\text{Spin}^c$ structure $\mathfrak{s}_0$ on $S^1\times S^2$, with absolute gradings $1/2,-1/2$ respectively.

\begin{proposition}
	\label{prop2.3.4}
	For the cobordism $Z_a$ associated to  $(a)$ in Figure $\ref{fig:kirby}$, the induced map on $\widehat{HF}$ is
	\begin{align*}
		f=F_{Z_a} :\widehat{HF}(\rp{3}\#\rp{3})& \longmapsto \widehat{HF}(S^1\times S^2)\\
		f(b)=f(c) = \overline{v}_-&, \, \, f(a)=f(d) =0.
	\end{align*} 
   For the cobordism $Z_b$ associated to $(b)$ in Figure $\ref{fig:kirby}$, the induced map on $\widehat{HF}$ is 
   	\begin{align*}
   	g=F_{Z_b}:\widehat{HF}(S^1\times S^2)& \longmapsto \widehat{HF}(\rp{3}\times \rp{3})\\
   	g(\overline{v}_+) = b+c&, \, \,  g(\overline{v}_-) = 0.
   \end{align*} 
Here we are summing over all $\text{Spin}^c$ structures on the cobordism.
\end{proposition}

\begin{proof}
	Let's start with $Z_a$. Let's compute the degree shift of the cobordism map $F_{Z_a,\mathfrak{s}}$ on $\widehat{HF}$ first. We will use the results of Section $11$ in \cite{manolescu2010heegaard}, where homology classes and $\text{Spin}^c$ structures of the cobordism are described in terms of linking matrices.
	
	 Let \begin{equation*}
		\Lambda = 
		\begin{pmatrix}
			-2& 0 &1 \\
			0 & -2 & 1 \\
			1 & 1 & -1
		\end{pmatrix}
	\end{equation*}
denote the linking matrix for $(a)$ in Figure $\ref{fig:kirby}$, and denote the $i$th column of $ \Lambda$ by $\Lambda_i$. Then we have \[H_2(Z_a,\mathbb{Z}) = \{ u \in \mathbb{Z}^3 \vert \Lambda_1\cdot u = \Lambda_2\cdot u =0 \} = \langle (1,1,2) \rangle \cong \mathbb{Z},\] and the intersection form $H_2(Z_a,\mathbb{Z}) \times H_2(Z_a,\mathbb{Z}) \longmapsto \mathbb{Z}$ is represented by $\Lambda$ restricted to the span of $(1,1,2)$, which is $0$. Therefore, we have $b_2(Z_a) =1, $ and $b_2^+(Z_a) = b_2^-(Z_a)=0$. Then, for any $\text{Spin}^c$ structure $s$ on $Z_a$, we have \[deg(F_{Z_a},\mathfrak{s}) = \frac{c_1(\mathfrak{s})^2-2\chi(Z_a)-3\sigma(Z_a)}{4} = \dfrac{0-2-0}{4}=-\frac{1}{2}.\] So  non-trivial maps only happen from $\widehat{HF}(\rp{3}\#\rp{3},\mathfrak{t})$ to $\widehat{HF}(S^1\times S^2,\mathfrak{s}_0)$, where $\mathfrak{s}_0$ is the torsion $\text{Spin}^c$ structure on $S^1\times S^2$, and $\mathfrak{t}$ is a $\text{Spin}^c$ structure on $\rp{3}\#\rp{3}$ such that the corresponding generator has absolute grading $0$, $i.e.$ generators $b$ and $c$. 

Consider the $\text{Spin}^c$ structure $\mathfrak{t}$ corresponding to the generator $b$ first, and let $\mathfrak{s}$ be the $\text{Spin}^c$ structure on $Z_a$ which restricts to $\mathfrak{t}$ and $\mathfrak{s}_0$ on each boundary component. Then the pair $(Z_a,\mathfrak{s})$ satisfies the condition of Proposition $\ref{prop2.3.3}$, with $Y = \rp{3}\#\rp{3}$, and $K$ the red curve in $(a)$ of Figure $\ref{fig:kirby}$. The knot $K$ represents a torsion class in $H_1(\rp{3}\#\rp{3}) = \mathbb{Z}_2\oplus \mathbb{Z}_2$. Let $L$ be the core of the glued-in solid torus in $Y(K) = S^1\times S^2$, then $[L]$ is a generator of $H_1(S^1\times S^2)$, and the action of $[L]$ on $\widehat{HF}(S^1\times S^2,\mathfrak{s}_0) = \langle\overline{v}_+,\overline{v} \rangle $ is such that \[[L]\cdot \overline{v}_+ = \overline{v}_-, \,\,\, [L]\cdot \overline{v}_-=0.\] So $\text{Ker}([L]) = \langle \overline{v}_-\rangle$, and $F_{Z_a,\mathfrak{s}}: \widehat{HF}(\rp{3}\#\rp{3},\mathfrak{t}) \longmapsto \text{Ker}[L]$ is an isomorphism sending $b$ to $\overline{v}_-$ by Proposition $\ref{prop2.3.3}$. The case when $\mathfrak{t}$ corresponds to the generator $c$ is similar, which is an isomorphism sending $c$ to $\overline{v}_-$. Therefore by summing up these two maps, we get the map \begin{align*}
	f=F_{Z_a} :\widehat{HF}(\rp{3}\#\rp{3})& \longmapsto \widehat{HF}(S^1\times S^2)\\
	f(b)=f(c) = \overline{v}_-&,\,\, f(a)=f(d) =0.
\end{align*}  as stated in the proposition.

Now we consider the cobordism $Z_b$ as in $(b)$ of Figure $\ref{fig:kirby}$. This time the linking matrix is 
\begin{equation*}
	\Lambda = 
	\begin{pmatrix}
		-2& 0 &1 &0\\
		0 & -2 & 1 &0 \\
		1 & 1 & -1 &1 \\
		0 & 0 & 1 &0
	\end{pmatrix}.
\end{equation*}
 Let $\Lambda_i$ be the $i$th column of $\Lambda$. Then we have 
 \[H_2(Z_b,\mathbb{Z}) = \{ u \in \mathbb{Z}^4 \vert \Lambda_1\cdot u = \Lambda_2\cdot u  = \Lambda_3\cdot u=0 \} = \langle (1,1,2,0) \rangle \cong \mathbb{Z},\] and again the intersection form is given by $\Lambda$ restricted to the span of $(1,1,2,0)$, which again vanishes. So we get $b_2(Z_b) =1$ and $b_2^+(Z_b) = b_2^-(Z_b)=0$. For any $\text{Spin}^c$ structure $\mathfrak{s}$ on $Z_b$, we again have $deg(F_{Z_b,\mathfrak{s}}) = -1/2$. Therefore, non-trivial maps only happen from $\widehat{HF}(S^1\times S^2,\mathfrak{s}_0)$ to $ \widehat{HF}(\rp{3}\#\rp{3},\mathfrak{t})$, where $\mathfrak{t}$ is again a $\text{Spin}^c$ structure on $\rp{3}\#\rp{3}$ corresponds to the generator $b$ or $c$ of absolute grading $0$.
 
Consider the $\text{Spin}^c$ structure $\mathfrak{t}$ corresponding to the generator $b$ first, and let $\mathfrak{s}$ be the $\text{Spin}^c$ structure on $Z_b$ which restricts to $\mathfrak{s}_0$ and $\mathfrak{t}$ on each boundary component. The pair $(Z_b,\mathfrak{s})$ satisfies the condition of Proposition $\ref{prop2.3.3}$ with $Y = S^1\times S^2$ and $K$ is the red curve in $(b)$ of Figure $\ref{fig:kirby}$. This time $[K]$ is twice the generator of $H_1(S^1\times S^2) = \mathbb{Z}$, which is non-torsion. According to 
Proposition $\ref{prop2.3.3}$, we have an isomorphism \[F_{Z_b,\mathfrak{s}}: \widehat{HF}(S^1\times S^2,\mathfrak{s}_0)/\text{Ker}[K] \longmapsto \widehat{HF}(\rp{3}\#\rp{3}, \mathfrak{t}),\]  where $\widehat{HF}(S^1\times S^2,\mathfrak{s}_0)/\text{Ker}[K] = \langle \overline{v}_+\rangle$. So $F_{Z_b,\mathfrak{s}}$ sends $\overline{v}_+$ to $b$. The case when $\mathfrak{t}$ corresponds to the generator $c$ is similar. By summing up these two maps, we get \begin{align*}
	g=F_{Z_b}:\widehat{HF}(S^1\otimes S^2)& \longmapsto \widehat{HF}(\rp{3}\times \rp{3})\\
	g(\overline{v}_+) = b+c&,\,\, g(\overline{v}_-) = 0.
\end{align*} 
\end{proof}

Combining Proposition $\ref{prop 2.3.1}$ and Proposition $\ref{prop2.3.4}$ and the description of the link spectral sequence of $\widehat{HF}$ relating to the cube of resolutions at the end of Section $\ref{sec2.2}$, we get the following:

\begin{theorem}
	Let $L$ be a null homologous link in $\rp{3}$. There is a spectral sequence whose $E^2$ terms consists of the Khovanov-type homology $\widetilde{\mathit{Kh}}^{\alpha_{\textit{HF}}}(m(L))$ of the mirror of $L$ with the dyad $\alpha_{HF} = (W,\overline{V},f,g)$ introduced in part ($\ref{ahf}$) of Section 1.6, which converges to the Heegaard Floer homology $\widehat{HF}(\Sigma_0(\rp{3},L))$ of the even branched double cover $\Sigma_0(\rp{3},L)$ of $\rp{3}$. 
\end{theorem}
\begin{proof}
	 The proof is similar to that of Theorem \ref{theorem2.2.4}. See the discussion at the end of Section \ref{sec2.2}. For each double point in a link projection of $L$ to $\rp{2}$, we associate an knot $K_i$ in $\Sigma(\rp{3}, L,C_P)$ to it, which is the branched double cover of a vertical arc connecting the double points. Consider the link spectral sequence of $\widehat{HF}$ of the link in $\Sigma(\rp{3}, L,C_P)$ consisting of all the $K_i$. For each smoothing $L_s$ of the link projection of $L$, we have the  branched double cover $\Sigma(\rp{3},L_s,C_P)$ equal to $(\rp{3}\#\rp{3})\# (S^1\times S^2)^{\#(k_s-1)}$ if $e_s(P)=0$, and equal to $(S^1\times S^2)^{\#k_s}$ if $e_s(P)=1$. Hence, the $E^1$ terms in the link surgery spectral sequence, which are the Heegaard Floer homology of the corresponding branched double covers, are the same as $\widetilde{CKh}_{\alpha_{HF}}(m(L))$ as vector spaces. The differential $d_1$ in the spectral sequence has been computed in Proposition \ref{prop 2.3.1} and Proposition \ref{prop2.3.4}, which is the same as the differential $d$ in the chain complex $\widetilde{CKh}_{\alpha_{HF}}(m(L))$. Therefore, $E^2$ terms of the link spectral sequence are the Khovanov-type homology $\widetilde{Kh}^{\alpha_{HF}}(m(L))$. 
\end{proof}
\begin{remark}
	There is another spectral sequence converging to the Heegaard Floer homology of the odd branched double cover $\widehat{HF}(\Sigma_1(\rp{3},L))$, whose $E^2$ terms consists of $\widetilde{\mathit{Kh}}^{\alpha^*_{\textit{HF}}}(m(L))$, where $\alpha^*_{\textit{HF}} = (\overline{V},W,g,f)$ is the dual dyad of $\alpha$. 
\end{remark}
\bibliographystyle{amsalpha}
\bibliography{biblio}

\end{document}

%% file: smoothing.pdf_tex
\begingroup%
  \makeatletter%
  \providecommand\color[2][]{%
    \errmessage{(Inkscape) Color is used for the text in Inkscape, but the package 'color.sty' is not loaded}%
    \renewcommand\color[2][]{}%
  }%
  \providecommand\transparent[1]{%
    \errmessage{(Inkscape) Transparency is used (non-zero) for the text in Inkscape, but the package 'transparent.sty' is not loaded}%
    \renewcommand\transparent[1]{}%
  }%
  \providecommand\rotatebox[2]{#2}%
  \newcommand*\fsize{\dimexpr\f@size pt\relax}%
  \newcommand*\lineheight[1]{\fontsize{\fsize}{#1\fsize}\selectfont}%
  \ifx\svgwidth\undefined%
    \setlength{\unitlength}{589.01869859bp}%
    \ifx\svgscale\undefined%
      \relax%
    \else%
      \setlength{\unitlength}{\unitlength * \real{\svgscale}}%
    \fi%
  \else%
    \setlength{\unitlength}{\svgwidth}%
  \fi%
  \global\let\svgwidth\undefined%
  \global\let\svgscale\undefined%
  \makeatother%
  \begin{picture}(1,0.266363)%
    \lineheight{1}%
    \setlength\tabcolsep{0pt}%
    \put(0,0){\includegraphics[width=\unitlength,page=1]{smoothing.pdf}}%
    \put(0.4479226,0.00519316){\color[rgb]{0,0,0}\makebox(0,0)[lt]{\lineheight{1.25}\smash{\begin{tabular}[t]{l}0 Smoothing\end{tabular}}}}%
    \put(0.84098342,0.00626631){\color[rgb]{0,0,0}\makebox(0,0)[lt]{\lineheight{1.25}\smash{\begin{tabular}[t]{l}1 Smoothing\end{tabular}}}}%
    \put(0,0){\includegraphics[width=\unitlength,page=2]{smoothing.pdf}}%
  \end{picture}%
\endgroup%

%% file: bifurcation.pdf_tex
\begingroup%
  \makeatletter%
  \providecommand\color[2][]{%
    \errmessage{(Inkscape) Color is used for the text in Inkscape, but the package 'color.sty' is not loaded}%
    \renewcommand\color[2][]{}%
  }%
  \providecommand\transparent[1]{%
    \errmessage{(Inkscape) Transparency is used (non-zero) for the text in Inkscape, but the package 'transparent.sty' is not loaded}%
    \renewcommand\transparent[1]{}%
  }%
  \providecommand\rotatebox[2]{#2}%
  \newcommand*\fsize{\dimexpr\f@size pt\relax}%
  \newcommand*\lineheight[1]{\fontsize{\fsize}{#1\fsize}\selectfont}%
  \ifx\svgwidth\undefined%
    \setlength{\unitlength}{1006.83229237bp}%
    \ifx\svgscale\undefined%
      \relax%
    \else%
      \setlength{\unitlength}{\unitlength * \real{\svgscale}}%
    \fi%
  \else%
    \setlength{\unitlength}{\svgwidth}%
  \fi%
  \global\let\svgwidth\undefined%
  \global\let\svgscale\undefined%
  \makeatother%
  \begin{picture}(1,0.13284726)%
    \lineheight{1}%
    \setlength\tabcolsep{0pt}%
    \put(0.47842032,0.07170189){\color[rgb]{0,0,0}\makebox(0,0)[lt]{\lineheight{1.25}\smash{\begin{tabular}[t]{l}1$\rightarrow$2\end{tabular}}}}%
    \put(0,0){\includegraphics[width=\unitlength,page=1]{bifurcation.pdf}}%
    \put(0.13940217,0.07027412){\color[rgb]{0,0,0}\makebox(0,0)[lt]{\lineheight{1.25}\smash{\begin{tabular}[t]{l}2$\rightarrow$1      \end{tabular}}}}%
    \put(0,0){\includegraphics[width=\unitlength,page=2]{bifurcation.pdf}}%
    \put(0.82257949,0.07095377){\color[rgb]{0,0,0}\makebox(0,0)[lt]{\lineheight{1.25}\smash{\begin{tabular}[t]{l}1$\rightarrow$1\end{tabular}}}}%
  \end{picture}%
\endgroup%

%% file: esp.pdf_tex
\begingroup%
  \makeatletter%
  \providecommand\color[2][]{%
    \errmessage{(Inkscape) Color is used for the text in Inkscape, but the package 'color.sty' is not loaded}%
    \renewcommand\color[2][]{}%
  }%
  \providecommand\transparent[1]{%
    \errmessage{(Inkscape) Transparency is used (non-zero) for the text in Inkscape, but the package 'transparent.sty' is not loaded}%
    \renewcommand\transparent[1]{}%
  }%
  \providecommand\rotatebox[2]{#2}%
  \newcommand*\fsize{\dimexpr\f@size pt\relax}%
  \newcommand*\lineheight[1]{\fontsize{\fsize}{#1\fsize}\selectfont}%
  \ifx\svgwidth\undefined%
    \setlength{\unitlength}{913.96170966bp}%
    \ifx\svgscale\undefined%
      \relax%
    \else%
      \setlength{\unitlength}{\unitlength * \real{\svgscale}}%
    \fi%
  \else%
    \setlength{\unitlength}{\svgwidth}%
  \fi%
  \global\let\svgwidth\undefined%
  \global\let\svgscale\undefined%
  \makeatother%
  \begin{picture}(1,0.40597238)%
    \lineheight{1}%
    \setlength\tabcolsep{0pt}%
    \put(0,0){\includegraphics[width=\unitlength,page=1]{esp.pdf}}%
    \put(0.06241028,0.37910057){\color[rgb]{0,0,0}\makebox(0,0)[lt]{\lineheight{1.25}\smash{\begin{tabular}[t]{l}$P$\end{tabular}}}}%
    \put(0.24754819,0.38018358){\color[rgb]{0,0,0}\makebox(0,0)[lt]{\lineheight{1.25}\smash{\begin{tabular}[t]{l}$P$\end{tabular}}}}%
    \put(0.55512841,0.34081883){\color[rgb]{0,0,0}\makebox(0,0)[lt]{\lineheight{1.25}\smash{\begin{tabular}[t]{l}$P$\end{tabular}}}}%
    \put(0.37102826,0.33912502){\color[rgb]{0,0,0}\makebox(0,0)[lt]{\lineheight{1.25}\smash{\begin{tabular}[t]{l}$P$\end{tabular}}}}%
    \put(0,0){\includegraphics[width=\unitlength,page=2]{esp.pdf}}%
    \put(0.74727359,0.33637399){\color[rgb]{0,0,0}\makebox(0,0)[lt]{\lineheight{1.25}\smash{\begin{tabular}[t]{l}$P$\end{tabular}}}}%
    \put(0,0){\includegraphics[width=\unitlength,page=3]{esp.pdf}}%
    \put(0.93231481,0.33637538){\color[rgb]{0,0,0}\makebox(0,0)[lt]{\lineheight{1.25}\smash{\begin{tabular}[t]{l}$P$\end{tabular}}}}%
    \put(0.40275785,0.22819888){\color[rgb]{0,0,0}\makebox(0,0)[lt]{\lineheight{1.25}\smash{\begin{tabular}[t]{l}2$\rightarrow$1 and 1$\rightarrow$2 bifurcations\end{tabular}}}}%
    \put(0,0){\includegraphics[width=\unitlength,page=4]{esp.pdf}}%
    \put(0.40453395,0.13594934){\color[rgb]{0,0,0}\makebox(0,0)[lt]{\lineheight{1.25}\smash{\begin{tabular}[t]{l}$P$\end{tabular}}}}%
    \put(0.58742987,0.13594885){\color[rgb]{0,0,0}\makebox(0,0)[lt]{\lineheight{1.25}\smash{\begin{tabular}[t]{l}$P$\end{tabular}}}}%
    \put(0.44573216,0.00023196){\color[rgb]{0,0,0}\makebox(0,0)[lt]{\lineheight{1.25}\smash{\begin{tabular}[t]{l}1$\rightarrow$1 bifurcations\end{tabular}}}}%
    \put(0,0){\includegraphics[width=\unitlength,page=5]{esp.pdf}}%
  \end{picture}%
\endgroup%

%% file: singular_curves.pdf_tex
\begingroup%
  \makeatletter%
  \providecommand\color[2][]{%
    \errmessage{(Inkscape) Color is used for the text in Inkscape, but the package 'color.sty' is not loaded}%
    \renewcommand\color[2][]{}%
  }%
  \providecommand\transparent[1]{%
    \errmessage{(Inkscape) Transparency is used (non-zero) for the text in Inkscape, but the package 'transparent.sty' is not loaded}%
    \renewcommand\transparent[1]{}%
  }%
  \providecommand\rotatebox[2]{#2}%
  \newcommand*\fsize{\dimexpr\f@size pt\relax}%
  \newcommand*\lineheight[1]{\fontsize{\fsize}{#1\fsize}\selectfont}%
  \ifx\svgwidth\undefined%
    \setlength{\unitlength}{392.75286763bp}%
    \ifx\svgscale\undefined%
      \relax%
    \else%
      \setlength{\unitlength}{\unitlength * \real{\svgscale}}%
    \fi%
  \else%
    \setlength{\unitlength}{\svgwidth}%
  \fi%
  \global\let\svgwidth\undefined%
  \global\let\svgscale\undefined%
  \makeatother%
  \begin{picture}(1,0.65800367)%
    \lineheight{1}%
    \setlength\tabcolsep{0pt}%
    \put(0,0){\includegraphics[width=\unitlength,page=1]{singular_curves.pdf}}%
    \put(0.00823498,0.63294659){\color[rgb]{0,0,0}\makebox(0,0)[lt]{\lineheight{1.25}\smash{\begin{tabular}[t]{l}$(a)$\end{tabular}}}}%
    \put(0,0){\includegraphics[width=\unitlength,page=2]{singular_curves.pdf}}%
    \put(0.37908437,0.63294725){\color[rgb]{0,0,0}\makebox(0,0)[lt]{\lineheight{1.25}\smash{\begin{tabular}[t]{l}$(b)$\end{tabular}}}}%
    \put(0,0){\includegraphics[width=\unitlength,page=3]{singular_curves.pdf}}%
    \put(0.37908487,0.26139666){\color[rgb]{0,0,0}\makebox(0,0)[lt]{\lineheight{1.25}\smash{\begin{tabular}[t]{l}$(e)$\end{tabular}}}}%
    \put(0,0){\includegraphics[width=\unitlength,page=4]{singular_curves.pdf}}%
    \put(0.75070723,0.63294725){\color[rgb]{0,0,0}\makebox(0,0)[lt]{\lineheight{1.25}\smash{\begin{tabular}[t]{l}$(c)$\end{tabular}}}}%
    \put(0,0){\includegraphics[width=\unitlength,page=5]{singular_curves.pdf}}%
    \put(0.00823633,0.26139738){\color[rgb]{0,0,0}\makebox(0,0)[lt]{\lineheight{1.25}\smash{\begin{tabular}[t]{l}$(d)$\end{tabular}}}}%
    \put(0,0){\includegraphics[width=\unitlength,page=6]{singular_curves.pdf}}%
    \put(0.75070525,0.26139666){\color[rgb]{0,0,0}\makebox(0,0)[lt]{\lineheight{1.25}\smash{\begin{tabular}[t]{l}$(f)$\end{tabular}}}}%
    \put(0,0){\includegraphics[width=\unitlength,page=7]{singular_curves.pdf}}%
  \end{picture}%
\endgroup%

%% file: example_for_d2.pdf_tex
\begingroup%
  \makeatletter%
  \providecommand\color[2][]{%
    \errmessage{(Inkscape) Color is used for the text in Inkscape, but the package 'color.sty' is not loaded}%
    \renewcommand\color[2][]{}%
  }%
  \providecommand\transparent[1]{%
    \errmessage{(Inkscape) Transparency is used (non-zero) for the text in Inkscape, but the package 'transparent.sty' is not loaded}%
    \renewcommand\transparent[1]{}%
  }%
  \providecommand\rotatebox[2]{#2}%
  \newcommand*\fsize{\dimexpr\f@size pt\relax}%
  \newcommand*\lineheight[1]{\fontsize{\fsize}{#1\fsize}\selectfont}%
  \ifx\svgwidth\undefined%
    \setlength{\unitlength}{122.79534983bp}%
    \ifx\svgscale\undefined%
      \relax%
    \else%
      \setlength{\unitlength}{\unitlength * \real{\svgscale}}%
    \fi%
  \else%
    \setlength{\unitlength}{\svgwidth}%
  \fi%
  \global\let\svgwidth\undefined%
  \global\let\svgscale\undefined%
  \makeatother%
  \begin{picture}(1,0.9887729)%
    \lineheight{1}%
    \setlength\tabcolsep{0pt}%
    \put(0,0){\includegraphics[width=\unitlength,page=1]{example_for_d2.pdf}}%
  \end{picture}%
\endgroup%

%% file: d2_calculation.pdf_tex
\begingroup%
  \makeatletter%
  \providecommand\color[2][]{%
    \errmessage{(Inkscape) Color is used for the text in Inkscape, but the package 'color.sty' is not loaded}%
    \renewcommand\color[2][]{}%
  }%
  \providecommand\transparent[1]{%
    \errmessage{(Inkscape) Transparency is used (non-zero) for the text in Inkscape, but the package 'transparent.sty' is not loaded}%
    \renewcommand\transparent[1]{}%
  }%
  \providecommand\rotatebox[2]{#2}%
  \newcommand*\fsize{\dimexpr\f@size pt\relax}%
  \newcommand*\lineheight[1]{\fontsize{\fsize}{#1\fsize}\selectfont}%
  \ifx\svgwidth\undefined%
    \setlength{\unitlength}{768.78483659bp}%
    \ifx\svgscale\undefined%
      \relax%
    \else%
      \setlength{\unitlength}{\unitlength * \real{\svgscale}}%
    \fi%
  \else%
    \setlength{\unitlength}{\svgwidth}%
  \fi%
  \global\let\svgwidth\undefined%
  \global\let\svgscale\undefined%
  \makeatother%
  \begin{picture}(1,0.7623989)%
    \lineheight{1}%
    \setlength\tabcolsep{0pt}%
    \put(0,0){\includegraphics[width=\unitlength,page=1]{d2_calculation.pdf}}%
    \put(0.77050956,0.10288851){\color[rgb]{0,0,0}\makebox(0,0)[lt]{\lineheight{1.25}\smash{\begin{tabular}[t]{l}$P$\end{tabular}}}}%
    \put(0,0){\includegraphics[width=\unitlength,page=2]{d2_calculation.pdf}}%
    \put(0.08723032,0.1426275){\color[rgb]{1,0,0}\makebox(0,0)[lt]{\lineheight{1.25}\smash{\begin{tabular}[t]{l}$M$\end{tabular}}}}%
    \put(0,0){\includegraphics[width=\unitlength,page=3]{d2_calculation.pdf}}%
    \put(0.06786424,0.22917364){\color[rgb]{0,0,0}\makebox(0,0)[lt]{\lineheight{1.25}\smash{\begin{tabular}[t]{l}$P$\end{tabular}}}}%
    \put(0,0){\includegraphics[width=\unitlength,page=4]{d2_calculation.pdf}}%
    \put(0.22899169,0.33511494){\color[rgb]{0,0,0}\makebox(0,0)[lt]{\lineheight{1.25}\smash{\begin{tabular}[t]{l}$P$\end{tabular}}}}%
    \put(0,0){\includegraphics[width=\unitlength,page=5]{d2_calculation.pdf}}%
    \put(0.23030599,0.14553145){\color[rgb]{0,0,0}\makebox(0,0)[lt]{\lineheight{1.25}\smash{\begin{tabular}[t]{l}$P$\end{tabular}}}}%
    \put(0,0){\includegraphics[width=\unitlength,page=6]{d2_calculation.pdf}}%
    \put(0.38577367,0.22565306){\color[rgb]{0,0,0}\makebox(0,0)[lt]{\lineheight{1.25}\smash{\begin{tabular}[t]{l}$P$\end{tabular}}}}%
    \put(0,0){\includegraphics[width=\unitlength,page=7]{d2_calculation.pdf}}%
    \put(0.06151429,0.55409138){\color[rgb]{1,0,0}\makebox(0,0)[lt]{\lineheight{1.25}\smash{\begin{tabular}[t]{l}$M$\end{tabular}}}}%
    \put(0,0){\includegraphics[width=\unitlength,page=8]{d2_calculation.pdf}}%
    \put(0.22299552,0.47436064){\color[rgb]{1,0,0}\makebox(0,0)[lt]{\lineheight{1.25}\smash{\begin{tabular}[t]{l}$M$\end{tabular}}}}%
    \put(0,0){\includegraphics[width=\unitlength,page=9]{d2_calculation.pdf}}%
    \put(0.38085098,0.55442791){\color[rgb]{1,0,0}\makebox(0,0)[lt]{\lineheight{1.25}\smash{\begin{tabular}[t]{l}$M$\end{tabular}}}}%
    \put(0,0){\includegraphics[width=\unitlength,page=10]{d2_calculation.pdf}}%
    \put(0.22326083,0.66028395){\color[rgb]{1,0,0}\makebox(0,0)[lt]{\lineheight{1.25}\smash{\begin{tabular}[t]{l}$M$\end{tabular}}}}%
    \put(0,0){\includegraphics[width=\unitlength,page=11]{d2_calculation.pdf}}%
    \put(0.067865,0.62904152){\color[rgb]{0,0,0}\makebox(0,0)[lt]{\lineheight{1.25}\smash{\begin{tabular}[t]{l}$P$\end{tabular}}}}%
    \put(0,0){\includegraphics[width=\unitlength,page=12]{d2_calculation.pdf}}%
    \put(0.22899245,0.73498276){\color[rgb]{0,0,0}\makebox(0,0)[lt]{\lineheight{1.25}\smash{\begin{tabular}[t]{l}$P$\end{tabular}}}}%
    \put(0,0){\includegraphics[width=\unitlength,page=13]{d2_calculation.pdf}}%
    \put(0.23030675,0.5453993){\color[rgb]{0,0,0}\makebox(0,0)[lt]{\lineheight{1.25}\smash{\begin{tabular}[t]{l}$P$\end{tabular}}}}%
    \put(0,0){\includegraphics[width=\unitlength,page=14]{d2_calculation.pdf}}%
    \put(0.38577446,0.61209){\color[rgb]{0,0,0}\makebox(0,0)[lt]{\lineheight{1.25}\smash{\begin{tabular}[t]{l}$P$\end{tabular}}}}%
    \put(0.22225996,0.40380386){\color[rgb]{0,0,0}\makebox(0,0)[lt]{\lineheight{1.25}\smash{\begin{tabular}[t]{l}1(a)\end{tabular}}}}%
    \put(0.76132914,0.40435289){\color[rgb]{0,0,0}\makebox(0,0)[lt]{\lineheight{1.25}\smash{\begin{tabular}[t]{l}1(b)\end{tabular}}}}%
    \put(0.22396978,0.00382225){\color[rgb]{0,0,0}\makebox(0,0)[lt]{\lineheight{1.25}\smash{\begin{tabular}[t]{l}2(a)\end{tabular}}}}%
    \put(0.76338017,0.00382129){\color[rgb]{0,0,0}\makebox(0,0)[lt]{\lineheight{1.25}\smash{\begin{tabular}[t]{l}2(b)\end{tabular}}}}%
    \put(0,0){\includegraphics[width=\unitlength,page=15]{d2_calculation.pdf}}%
    \put(0.60187112,0.55409113){\color[rgb]{1,0,0}\makebox(0,0)[lt]{\lineheight{1.25}\smash{\begin{tabular}[t]{l}$M$\end{tabular}}}}%
    \put(0,0){\includegraphics[width=\unitlength,page=16]{d2_calculation.pdf}}%
    \put(0.76335232,0.47436039){\color[rgb]{1,0,0}\makebox(0,0)[lt]{\lineheight{1.25}\smash{\begin{tabular}[t]{l}$M$\end{tabular}}}}%
    \put(0,0){\includegraphics[width=\unitlength,page=17]{d2_calculation.pdf}}%
    \put(0.92120783,0.55442766){\color[rgb]{1,0,0}\makebox(0,0)[lt]{\lineheight{1.25}\smash{\begin{tabular}[t]{l}$M$\end{tabular}}}}%
    \put(0,0){\includegraphics[width=\unitlength,page=18]{d2_calculation.pdf}}%
    \put(0.76361765,0.6602837){\color[rgb]{1,0,0}\makebox(0,0)[lt]{\lineheight{1.25}\smash{\begin{tabular}[t]{l}$M$\end{tabular}}}}%
    \put(0,0){\includegraphics[width=\unitlength,page=19]{d2_calculation.pdf}}%
    \put(0.61206757,0.58080206){\color[rgb]{0,0,0}\makebox(0,0)[lt]{\lineheight{1.25}\smash{\begin{tabular}[t]{l}$P$\end{tabular}}}}%
    \put(0,0){\includegraphics[width=\unitlength,page=20]{d2_calculation.pdf}}%
    \put(0.77354246,0.68660693){\color[rgb]{0,0,0}\makebox(0,0)[lt]{\lineheight{1.25}\smash{\begin{tabular}[t]{l}$P$\end{tabular}}}}%
    \put(0,0){\includegraphics[width=\unitlength,page=21]{d2_calculation.pdf}}%
    \put(0.77453357,0.50186104){\color[rgb]{0,0,0}\makebox(0,0)[lt]{\lineheight{1.25}\smash{\begin{tabular}[t]{l}$P$\end{tabular}}}}%
    \put(0,0){\includegraphics[width=\unitlength,page=22]{d2_calculation.pdf}}%
    \put(0.93151613,0.57983626){\color[rgb]{0,0,0}\makebox(0,0)[lt]{\lineheight{1.25}\smash{\begin{tabular}[t]{l}$P$\end{tabular}}}}%
    \put(0,0){\includegraphics[width=\unitlength,page=23]{d2_calculation.pdf}}%
    \put(0.24936922,0.06138307){\color[rgb]{1,0,0}\makebox(0,0)[lt]{\lineheight{1.25}\smash{\begin{tabular}[t]{l}$M$\end{tabular}}}}%
    \put(0,0){\includegraphics[width=\unitlength,page=24]{d2_calculation.pdf}}%
    \put(0.24989069,0.24665387){\color[rgb]{1,0,0}\makebox(0,0)[lt]{\lineheight{1.25}\smash{\begin{tabular}[t]{l}$M$\end{tabular}}}}%
    \put(0,0){\includegraphics[width=\unitlength,page=25]{d2_calculation.pdf}}%
    \put(0.40758355,0.14061304){\color[rgb]{1,0,0}\makebox(0,0)[lt]{\lineheight{1.25}\smash{\begin{tabular}[t]{l}$M$\end{tabular}}}}%
    \put(0,0){\includegraphics[width=\unitlength,page=26]{d2_calculation.pdf}}%
    \put(0.62758611,0.14350507){\color[rgb]{1,0,0}\makebox(0,0)[lt]{\lineheight{1.25}\smash{\begin{tabular}[t]{l}$M$\end{tabular}}}}%
    \put(0,0){\includegraphics[width=\unitlength,page=27]{d2_calculation.pdf}}%
    \put(0.60752521,0.18280347){\color[rgb]{0,0,0}\makebox(0,0)[lt]{\lineheight{1.25}\smash{\begin{tabular}[t]{l}$P$\end{tabular}}}}%
    \put(0,0){\includegraphics[width=\unitlength,page=28]{d2_calculation.pdf}}%
    \put(0.76847895,0.28822361){\color[rgb]{0,0,0}\makebox(0,0)[lt]{\lineheight{1.25}\smash{\begin{tabular}[t]{l}$P$\end{tabular}}}}%
    \put(0,0){\includegraphics[width=\unitlength,page=29]{d2_calculation.pdf}}%
    \put(0.92786642,0.18345183){\color[rgb]{0,0,0}\makebox(0,0)[lt]{\lineheight{1.25}\smash{\begin{tabular}[t]{l}$P$\end{tabular}}}}%
    \put(0,0){\includegraphics[width=\unitlength,page=30]{d2_calculation.pdf}}%
    \put(0.789711,0.06226065){\color[rgb]{1,0,0}\makebox(0,0)[lt]{\lineheight{1.25}\smash{\begin{tabular}[t]{l}$M$\end{tabular}}}}%
    \put(0,0){\includegraphics[width=\unitlength,page=31]{d2_calculation.pdf}}%
    \put(0.79024644,0.24753144){\color[rgb]{1,0,0}\makebox(0,0)[lt]{\lineheight{1.25}\smash{\begin{tabular}[t]{l}$M$\end{tabular}}}}%
    \put(0,0){\includegraphics[width=\unitlength,page=32]{d2_calculation.pdf}}%
    \put(0.9479393,0.14149061){\color[rgb]{1,0,0}\makebox(0,0)[lt]{\lineheight{1.25}\smash{\begin{tabular}[t]{l}$M$\end{tabular}}}}%
    \put(0,0){\includegraphics[width=\unitlength,page=33]{d2_calculation.pdf}}%
  \end{picture}%
\endgroup%

%% file: example_for_kh.pdf_tex
\begingroup%
  \makeatletter%
  \providecommand\color[2][]{%
    \errmessage{(Inkscape) Color is used for the text in Inkscape, but the package 'color.sty' is not loaded}%
    \renewcommand\color[2][]{}%
  }%
  \providecommand\transparent[1]{%
    \errmessage{(Inkscape) Transparency is used (non-zero) for the text in Inkscape, but the package 'transparent.sty' is not loaded}%
    \renewcommand\transparent[1]{}%
  }%
  \providecommand\rotatebox[2]{#2}%
  \newcommand*\fsize{\dimexpr\f@size pt\relax}%
  \newcommand*\lineheight[1]{\fontsize{\fsize}{#1\fsize}\selectfont}%
  \ifx\svgwidth\undefined%
    \setlength{\unitlength}{415.50584211bp}%
    \ifx\svgscale\undefined%
      \relax%
    \else%
      \setlength{\unitlength}{\unitlength * \real{\svgscale}}%
    \fi%
  \else%
    \setlength{\unitlength}{\svgwidth}%
  \fi%
  \global\let\svgwidth\undefined%
  \global\let\svgscale\undefined%
  \makeatother%
  \begin{picture}(1,1.08101651)%
    \lineheight{1}%
    \setlength\tabcolsep{0pt}%
    \put(0,0){\includegraphics[width=\unitlength,page=1]{example_for_kh.pdf}}%
    \put(0.00806907,0.70126846){\color[rgb]{0,0,0}\makebox(0,0)[lt]{\lineheight{1.25}\smash{\begin{tabular}[t]{l}The knot projection $L$ of an \\ null-homologous knot \end{tabular}}}}%
    \put(0,0){\includegraphics[width=\unitlength,page=2]{example_for_kh.pdf}}%
    \put(0.55896925,0.69942619){\makebox(0,0)[lt]{\lineheight{1.25}\smash{\begin{tabular}[t]{l}The orientation-preserving\\resolution of $L$\end{tabular}}}}%
    \put(0,0){\includegraphics[width=\unitlength,page=3]{example_for_kh.pdf}}%
    \put(0.00010493,0.20229207){\makebox(0,0)[lt]{\lineheight{1.25}\smash{\begin{tabular}[t]{l}A Seifert surface of $L$, where \\the orientation is such that in \\the red region, the normal \\vector is pointing upwards, \\and in the green region, it is \\pointing downwards. \end{tabular}}}}%
    \put(0,0){\includegraphics[width=\unitlength,page=4]{example_for_kh.pdf}}%
    \put(0.55079101,0.19991157){\makebox(0,0)[lt]{\lineheight{1.25}\smash{\begin{tabular}[t]{l}The division of the complement \\of L into $R_0$ and $R_1$, where the blank \\region is $R_0$, and the shaded\\region is $R_1$, and a choice of \\the point $P$.\\   \end{tabular}}}}%
    \put(0,0){\includegraphics[width=\unitlength,page=5]{example_for_kh.pdf}}%
    \put(0.73900469,0.49916074){\color[rgb]{0,0,0}\makebox(0,0)[lt]{\lineheight{1.25}\smash{\begin{tabular}[t]{l}P\end{tabular}}}}%
  \end{picture}%
\endgroup%

%% file: r_moves_1.pdf_tex
\begingroup%
  \makeatletter%
  \providecommand\color[2][]{%
    \errmessage{(Inkscape) Color is used for the text in Inkscape, but the package 'color.sty' is not loaded}%
    \renewcommand\color[2][]{}%
  }%
  \providecommand\transparent[1]{%
    \errmessage{(Inkscape) Transparency is used (non-zero) for the text in Inkscape, but the package 'transparent.sty' is not loaded}%
    \renewcommand\transparent[1]{}%
  }%
  \providecommand\rotatebox[2]{#2}%
  \newcommand*\fsize{\dimexpr\f@size pt\relax}%
  \newcommand*\lineheight[1]{\fontsize{\fsize}{#1\fsize}\selectfont}%
  \ifx\svgwidth\undefined%
    \setlength{\unitlength}{757.58108759bp}%
    \ifx\svgscale\undefined%
      \relax%
    \else%
      \setlength{\unitlength}{\unitlength * \real{\svgscale}}%
    \fi%
  \else%
    \setlength{\unitlength}{\svgwidth}%
  \fi%
  \global\let\svgwidth\undefined%
  \global\let\svgscale\undefined%
  \makeatother%
  \begin{picture}(1,0.41409225)%
    \lineheight{1}%
    \setlength\tabcolsep{0pt}%
    \put(0,0){\includegraphics[width=\unitlength,page=1]{r_moves_1.pdf}}%

    \put(-0.00080142,0.40052514){\makebox(0,0)[lt]{\lineheight{1.25}\smash{\begin{tabular}[t]{l}$(a)$\end{tabular}}}}%
    \put(0.12905331,0.27651346){\makebox(0,0)[lt]{\lineheight{1.25}\smash{\begin{tabular}[t]{l}R-I\end{tabular}}}}%
    \put(0,0){\includegraphics[width=\unitlength,page=2]{r_moves_1.pdf}}%
    \put(0.33266357,0.40052533){\makebox(0,0)[lt]{\lineheight{1.25}\smash{\begin{tabular}[t]{l}$(b)$\end{tabular}}}}%
    \put(0.4720776,0.27651563){\makebox(0,0)[lt]{\lineheight{1.25}\smash{\begin{tabular}[t]{l}R-II\end{tabular}}}}%
    \put(0,0){\includegraphics[width=\unitlength,page=3]{r_moves_1.pdf}}%
    \put(0.68871413,0.40052533){\makebox(0,0)[lt]{\lineheight{1.25}\smash{\begin{tabular}[t]{l}$(c)$\end{tabular}}}}%
    \put(0.83681234,0.27651563){\makebox(0,0)[lt]{\lineheight{1.25}\smash{\begin{tabular}[t]{l}R-III\end{tabular}}}}%
    \put(0,0){\includegraphics[width=\unitlength,page=4]{r_moves_1.pdf}}%
    \put(-0.00079948,0.1953307){\makebox(0,0)[lt]{\lineheight{1.25}\smash{\begin{tabular}[t]{l}$(d)$\end{tabular}}}}%
    \put(0.21061128,0.00029576){\makebox(0,0)[lt]{\lineheight{1.25}\smash{\begin{tabular}[t]{l}R-IV\end{tabular}}}}%
    \put(0,0){\includegraphics[width=\unitlength,page=5]{r_moves_1.pdf}}%
    \put(0.55817886,0.19533167){\makebox(0,0)[lt]{\lineheight{1.25}\smash{\begin{tabular}[t]{l}$(e)$\end{tabular}}}}%
    \put(0.76958959,0.00029685){\makebox(0,0)[lt]{\lineheight{1.25}\smash{\begin{tabular}[t]{l}R-V\end{tabular}}}}%
    \put(0,0){\includegraphics[width=\unitlength,page=6]{r_moves_1.pdf}}%
  \end{picture}%
\endgroup%

%% file: invariance_unde_r2.pdf_tex
\begingroup%
  \makeatletter%
  \providecommand\color[2][]{%
    \errmessage{(Inkscape) Color is used for the text in Inkscape, but the package 'color.sty' is not loaded}%
    \renewcommand\color[2][]{}%
  }%
  \providecommand\transparent[1]{%
    \errmessage{(Inkscape) Transparency is used (non-zero) for the text in Inkscape, but the package 'transparent.sty' is not loaded}%
    \renewcommand\transparent[1]{}%
  }%
  \providecommand\rotatebox[2]{#2}%
  \newcommand*\fsize{\dimexpr\f@size pt\relax}%
  \newcommand*\lineheight[1]{\fontsize{\fsize}{#1\fsize}\selectfont}%
  \ifx\svgwidth\undefined%
    \setlength{\unitlength}{460.84232737bp}%
    \ifx\svgscale\undefined%
      \relax%
    \else%
      \setlength{\unitlength}{\unitlength * \real{\svgscale}}%
    \fi%
  \else%
    \setlength{\unitlength}{\svgwidth}%
  \fi%
  \global\let\svgwidth\undefined%
  \global\let\svgscale\undefined%
  \makeatother%
  \begin{picture}(1,0.46063365)%
    \lineheight{1}%
    \setlength\tabcolsep{0pt}%
    \put(0,0){\includegraphics[width=\unitlength,page=1]{invariance_unde_r2.pdf}}%
    \put(0.23433779,0.20891083){\makebox(0,0)[lt]{\lineheight{1.25}\smash{\begin{tabular}[t]{l}$=$\end{tabular}}}}%
    \put(0,0){\includegraphics[width=\unitlength,page=2]{invariance_unde_r2.pdf}}%
    \put(0.47661119,0.06185026){\makebox(0,0)[lt]{\lineheight{1.25}\smash{\begin{tabular}[t]{l}$d_2$\end{tabular}}}}%
    \put(0.80065258,0.05769347){\makebox(0,0)[lt]{\lineheight{1.25}\smash{\begin{tabular}[t]{l}$d_4$\end{tabular}}}}%
    \put(0.80468271,0.35454152){\makebox(0,0)[lt]{\lineheight{1.25}\smash{\begin{tabular}[t]{l}$d_3$\end{tabular}}}}%
    \put(0.36441514,0.30351321){\makebox(0,0)[lt]{\lineheight{1.25}\smash{\begin{tabular}[t]{l}$C_1$\end{tabular}}}}%
    \put(0.48033604,0.34533033){\makebox(0,0)[lt]{\lineheight{1.25}\smash{\begin{tabular}[t]{l}$d_1$\end{tabular}}}}%
    \put(0.62744322,0.44188774){\makebox(0,0)[lt]{\lineheight{1.25}\smash{\begin{tabular}[t]{l}$C_1\otimes V$\end{tabular}}}}%
    \put(0.88268507,0.29036781){\makebox(0,0)[lt]{\lineheight{1.25}\smash{\begin{tabular}[t]{l}$C_1$\end{tabular}}}}%
    \put(0.63699692,0.13424485){\makebox(0,0)[lt]{\lineheight{1.25}\smash{\begin{tabular}[t]{l}$C_2$\end{tabular}}}}%
  \end{picture}%
\endgroup%

%% file: cube_of_resolution.pdf_tex
\begingroup%
  \makeatletter%
  \providecommand\color[2][]{%
    \errmessage{(Inkscape) Color is used for the text in Inkscape, but the package 'color.sty' is not loaded}%
    \renewcommand\color[2][]{}%
  }%
  \providecommand\transparent[1]{%
    \errmessage{(Inkscape) Transparency is used (non-zero) for the text in Inkscape, but the package 'transparent.sty' is not loaded}%
    \renewcommand\transparent[1]{}%
  }%
  \providecommand\rotatebox[2]{#2}%
  \newcommand*\fsize{\dimexpr\f@size pt\relax}%
  \newcommand*\lineheight[1]{\fontsize{\fsize}{#1\fsize}\selectfont}%
  \ifx\svgwidth\undefined%
    \setlength{\unitlength}{739.50273828bp}%
    \ifx\svgscale\undefined%
      \relax%
    \else%
      \setlength{\unitlength}{\unitlength * \real{\svgscale}}%
    \fi%
  \else%
    \setlength{\unitlength}{\svgwidth}%
  \fi%
  \global\let\svgwidth\undefined%
  \global\let\svgscale\undefined%
  \makeatother%
  \begin{picture}(1,0.67599212)%
    \lineheight{1}%
    \setlength\tabcolsep{0pt}%
    \put(0,0){\includegraphics[width=\unitlength,page=1]{cube_of_resolution.pdf}}%
    \put(0.90515457,0.42154655){\color[rgb]{0,0,0}\makebox(0,0)[lt]{\lineheight{1.25}\smash{\begin{tabular}[t]{l}$P$\end{tabular}}}}%
    \put(0,0){\includegraphics[width=\unitlength,page=2]{cube_of_resolution.pdf}}%
    \put(0.90103491,0.30915763){\color[rgb]{1,0,0}\makebox(0,0)[lt]{\lineheight{1.25}\smash{\begin{tabular}[t]{l}$M$\end{tabular}}}}%
    \put(0,0){\includegraphics[width=\unitlength,page=3]{cube_of_resolution.pdf}}%
    \put(0.64016753,0.648233){\color[rgb]{0,0,0}\makebox(0,0)[lt]{\lineheight{1.25}\smash{\begin{tabular}[t]{l}$P$\end{tabular}}}}%
    \put(0,0){\includegraphics[width=\unitlength,page=4]{cube_of_resolution.pdf}}%
    \put(0.63604793,0.53584407){\color[rgb]{1,0,0}\makebox(0,0)[lt]{\lineheight{1.25}\smash{\begin{tabular}[t]{l}$M$\end{tabular}}}}%
    \put(0,0){\includegraphics[width=\unitlength,page=5]{cube_of_resolution.pdf}}%
    \put(0.64016759,0.17087972){\color[rgb]{0,0,0}\makebox(0,0)[lt]{\lineheight{1.25}\smash{\begin{tabular}[t]{l}$P$\end{tabular}}}}%
    \put(0,0){\includegraphics[width=\unitlength,page=6]{cube_of_resolution.pdf}}%
    \put(0.63604799,0.05849079){\color[rgb]{1,0,0}\makebox(0,0)[lt]{\lineheight{1.25}\smash{\begin{tabular}[t]{l}$M$\end{tabular}}}}%
    \put(0,0){\includegraphics[width=\unitlength,page=7]{cube_of_resolution.pdf}}%
    \put(0.36185588,0.17087972){\color[rgb]{0,0,0}\makebox(0,0)[lt]{\lineheight{1.25}\smash{\begin{tabular}[t]{l}$P$\end{tabular}}}}%
    \put(0,0){\includegraphics[width=\unitlength,page=8]{cube_of_resolution.pdf}}%
    \put(0.35773628,0.05849079){\color[rgb]{1,0,0}\makebox(0,0)[lt]{\lineheight{1.25}\smash{\begin{tabular}[t]{l}$M$\end{tabular}}}}%
    \put(0,0){\includegraphics[width=\unitlength,page=9]{cube_of_resolution.pdf}}%
    \put(0.09112344,0.41269955){\color[rgb]{0,0,0}\makebox(0,0)[lt]{\lineheight{1.25}\smash{\begin{tabular}[t]{l}$P$\end{tabular}}}}%
    \put(0,0){\includegraphics[width=\unitlength,page=10]{cube_of_resolution.pdf}}%
    \put(0.08700385,0.30031063){\color[rgb]{1,0,0}\makebox(0,0)[lt]{\lineheight{1.25}\smash{\begin{tabular}[t]{l}$M$\end{tabular}}}}%
    \put(0.05700385,0.24031063){\makebox(0,0)[lt]{\lineheight{1.25}\smash{\begin{tabular}[t]{l}$e_s(P) = 0$\end{tabular}}}}%
    \put(0.32972459,0.24031063){\makebox(0,0)[lt]{\lineheight{1.25}\smash{\begin{tabular}[t]{l}$e_s(P) = 0$\end{tabular}}}}%
    \put(0.32972459,-0.00149079){\makebox(0,0)[lt]{\lineheight{1.25}\smash{\begin{tabular}[t]{l}$e_s(P) = 0$\end{tabular}}}}%
    \put(0.32972403,0.47315117){\makebox(0,0)[lt]{\lineheight{1.25}\smash{\begin{tabular}[t]{l}$e_s(P) = 0$\end{tabular}}}}%
    \put(0.60669073,-0.00149079){\makebox(0,0)[lt]{\lineheight{1.25}\smash{\begin{tabular}[t]{l}$e_s(P) = 0$\end{tabular}}}}%
    \put(0.60669088,0.25031063){\makebox(0,0)[lt]{\lineheight{1.25}\smash{\begin{tabular}[t]{l}$e_s(P) = 1$\end{tabular}}}}%
    \put(0.60669088,0.47315186){\makebox(0,0)[lt]{\lineheight{1.25}\smash{\begin{tabular}[t]{l}$e_s(P) = 1$\end{tabular}}}}%
    \put(0.87107702,0.25031063){\makebox(0,0)[lt]{\lineheight{1.25}\smash{\begin{tabular}[t]{l}$e_s(P) = 1$\end{tabular}}}}%
    \put(0,0){\includegraphics[width=\unitlength,page=11]{cube_of_resolution.pdf}}%
    \put(0.36210499,0.64823301){\color[rgb]{0,0,0}\makebox(0,0)[lt]{\lineheight{1.25}\smash{\begin{tabular}[t]{l}$P$\end{tabular}}}}%
    \put(0,0){\includegraphics[width=\unitlength,page=12]{cube_of_resolution.pdf}}%
    \put(0.35798542,0.53584408){\color[rgb]{1,0,0}\makebox(0,0)[lt]{\lineheight{1.25}\smash{\begin{tabular}[t]{l}$M$\end{tabular}}}}%
    \put(0,0){\includegraphics[width=\unitlength,page=13]{cube_of_resolution.pdf}}%
    \put(0.36185588,0.41269955){\color[rgb]{0,0,0}\makebox(0,0)[lt]{\lineheight{1.25}\smash{\begin{tabular}[t]{l}$P$\end{tabular}}}}%
    \put(0,0){\includegraphics[width=\unitlength,page=14]{cube_of_resolution.pdf}}%
    \put(0.35773628,0.30031063){\color[rgb]{1,0,0}\makebox(0,0)[lt]{\lineheight{1.25}\smash{\begin{tabular}[t]{l}$M$\end{tabular}}}}%
    \put(0,0){\includegraphics[width=\unitlength,page=15]{cube_of_resolution.pdf}}%
    \put(0.64016759,0.4215465){\color[rgb]{0,0,0}\makebox(0,0)[lt]{\lineheight{1.25}\smash{\begin{tabular}[t]{l}$P$\end{tabular}}}}%
    \put(0,0){\includegraphics[width=\unitlength,page=16]{cube_of_resolution.pdf}}%
    \put(0.63604799,0.30915757){\color[rgb]{1,0,0}\makebox(0,0)[lt]{\lineheight{1.25}\smash{\begin{tabular}[t]{l}$M$\end{tabular}}}}%
    \put(0,0){\includegraphics[width=\unitlength,page=17]{cube_of_resolution.pdf}}%
  \end{picture}%
\endgroup%

%% file: kirby_diagrams.pdf_tex
\begingroup%
  \makeatletter%
  \providecommand\color[2][]{%
    \errmessage{(Inkscape) Color is used for the text in Inkscape, but the package 'color.sty' is not loaded}%
    \renewcommand\color[2][]{}%
  }%
  \providecommand\transparent[1]{%
    \errmessage{(Inkscape) Transparency is used (non-zero) for the text in Inkscape, but the package 'transparent.sty' is not loaded}%
    \renewcommand\transparent[1]{}%
  }%
  \providecommand\rotatebox[2]{#2}%
  \newcommand*\fsize{\dimexpr\f@size pt\relax}%
  \newcommand*\lineheight[1]{\fontsize{\fsize}{#1\fsize}\selectfont}%
  \ifx\svgwidth\undefined%
    \setlength{\unitlength}{305.37468829bp}%
    \ifx\svgscale\undefined%
      \relax%
    \else%
      \setlength{\unitlength}{\unitlength * \real{\svgscale}}%
    \fi%
  \else%
    \setlength{\unitlength}{\svgwidth}%
  \fi%
  \global\let\svgwidth\undefined%
  \global\let\svgscale\undefined%
  \makeatother%
  \begin{picture}(1,0.44047329)%
    \lineheight{1}%
    \setlength\tabcolsep{0pt}%
    \put(0,0){\includegraphics[width=\unitlength,page=1]{kirby_diagrams.pdf}}%
    \put(0.02611626,0.41224109){\makebox(0,0)[lt]{\lineheight{1.25}\smash{\begin{tabular}[t]{l}$-2$\end{tabular}}}}%
    \put(0.32909154,0.4036929){\makebox(0,0)[lt]{\lineheight{1.25}\smash{\begin{tabular}[t]{l}$-2$\end{tabular}}}}%
    \put(0.17386268,0.36906016){\makebox(0,0)[lt]{\lineheight{1.25}\smash{\begin{tabular}[t]{l}$-1$\end{tabular}}}}%
    \put(0,0){\includegraphics[width=\unitlength,page=2]{kirby_diagrams.pdf}}%
    \put(0.0408977,0.16002677){\makebox(0,0)[lt]{\lineheight{1.25}\smash{\begin{tabular}[t]{l}$-2$\end{tabular}}}}%
    \put(0.30497231,0.15755302){\makebox(0,0)[lt]{\lineheight{1.25}\smash{\begin{tabular}[t]{l}$-2$\end{tabular}}}}%
    \put(0.172935,0.16002684){\makebox(0,0)[lt]{\lineheight{1.25}\smash{\begin{tabular}[t]{l}$-1$\end{tabular}}}}%
    \put(0.16459075,0.00893354){\makebox(0,0)[lt]{\lineheight{1.25}\smash{\begin{tabular}[t]{l}$(a)$\end{tabular}}}}%
    \put(0,0){\includegraphics[width=\unitlength,page=3]{kirby_diagrams.pdf}}%
    \put(0.62719077,0.41224495){\makebox(0,0)[lt]{\lineheight{1.25}\smash{\begin{tabular}[t]{l}$-2$\end{tabular}}}}%
    \put(0.93016595,0.40369676){\makebox(0,0)[lt]{\lineheight{1.25}\smash{\begin{tabular}[t]{l}$-2$\end{tabular}}}}%
    \put(0.77304587,0.36575534){\makebox(0,0)[lt]{\lineheight{1.25}\smash{\begin{tabular}[t]{l}$-1$\end{tabular}}}}%
    \put(0,0){\includegraphics[width=\unitlength,page=4]{kirby_diagrams.pdf}}%
    \put(0.64197225,0.16003052){\makebox(0,0)[lt]{\lineheight{1.25}\smash{\begin{tabular}[t]{l}$-2$\end{tabular}}}}%
    \put(0.90604679,0.15755678){\makebox(0,0)[lt]{\lineheight{1.25}\smash{\begin{tabular}[t]{l}$-2$\end{tabular}}}}%
    \put(0.77400945,0.16003059){\makebox(0,0)[lt]{\lineheight{1.25}\smash{\begin{tabular}[t]{l}$-1$\end{tabular}}}}%
    \put(0.76812494,0.00893814){\makebox(0,0)[lt]{\lineheight{1.25}\smash{\begin{tabular}[t]{l}$(b)$\end{tabular}}}}%
    \put(0,0){\includegraphics[width=\unitlength,page=5]{kirby_diagrams.pdf}}%
    \put(0.85335757,0.20764683){\makebox(0,0)[lt]{\lineheight{1.25}\smash{\begin{tabular}[t]{l}$0$\end{tabular}}}}%
    \put(0,0){\includegraphics[width=\unitlength,page=6]{kirby_diagrams.pdf}}%
    \put(0.8068025,0.05674836){\makebox(0,0)[lt]{\lineheight{1.25}\smash{\begin{tabular}[t]{l}0\end{tabular}}}}%
  \end{picture}%
\endgroup%

%% file: cobordism_a.pdf_tex
\begingroup%
  \makeatletter%
  \providecommand\color[2][]{%
    \errmessage{(Inkscape) Color is used for the text in Inkscape, but the package 'color.sty' is not loaded}%
    \renewcommand\color[2][]{}%
  }%
  \providecommand\transparent[1]{%
    \errmessage{(Inkscape) Transparency is used (non-zero) for the text in Inkscape, but the package 'transparent.sty' is not loaded}%
    \renewcommand\transparent[1]{}%
  }%
  \providecommand\rotatebox[2]{#2}%
  \newcommand*\fsize{\dimexpr\f@size pt\relax}%
  \newcommand*\lineheight[1]{\fontsize{\fsize}{#1\fsize}\selectfont}%
  \ifx\svgwidth\undefined%
    \setlength{\unitlength}{335.90000227bp}%
    \ifx\svgscale\undefined%
      \relax%
    \else%
      \setlength{\unitlength}{\unitlength * \real{\svgscale}}%
    \fi%
  \else%
    \setlength{\unitlength}{\svgwidth}%
  \fi%
  \global\let\svgwidth\undefined%
  \global\let\svgscale\undefined%
  \makeatother%
  \begin{picture}(1,0.64889783)%
    \lineheight{1}%
    \setlength\tabcolsep{0pt}%
    \put(0,0){\includegraphics[width=\unitlength,page=1]{cobordism_a.pdf}}%
    \put(0.11384538,0.51019992){\makebox(0,0)[lt]{\lineheight{1.25}\smash{\begin{tabular}[t]{l}$a$\end{tabular}}}}%
    \put(0,0){\includegraphics[width=\unitlength,page=2]{cobordism_a.pdf}}%
    \put(0.11568271,0.36718126){\makebox(0,0)[lt]{\lineheight{1.25}\smash{\begin{tabular}[t]{l}$b$\end{tabular}}}}%
    \put(0.28549206,0.50400636){\makebox(0,0)[lt]{\lineheight{1.25}\smash{\begin{tabular}[t]{l}$c$\end{tabular}}}}%
    \put(0.27912196,0.37167579){\makebox(0,0)[lt]{\lineheight{1.25}\smash{\begin{tabular}[t]{l}$d$\end{tabular}}}}%
    \put(0,0){\includegraphics[width=\unitlength,page=3]{cobordism_a.pdf}}%
    \put(0.21150471,0.5939573){\makebox(0,0)[lt]{\lineheight{1.25}\smash{\begin{tabular}[t]{l}$P$\end{tabular}}}}%
    \put(0,0){\includegraphics[width=\unitlength,page=4]{cobordism_a.pdf}}%
    \put(0.87819319,0.50172787){\makebox(0,0)[lt]{\lineheight{1.25}\smash{\begin{tabular}[t]{l}$a$\end{tabular}}}}%
    \put(0,0){\includegraphics[width=\unitlength,page=5]{cobordism_a.pdf}}%
    \put(0.87311882,0.38011429){\makebox(0,0)[lt]{\lineheight{1.25}\smash{\begin{tabular}[t]{l}$b$\end{tabular}}}}%
    \put(0.70108622,0.50426911){\makebox(0,0)[lt]{\lineheight{1.25}\smash{\begin{tabular}[t]{l}$c$\end{tabular}}}}%
    \put(0.69929403,0.3734439){\makebox(0,0)[lt]{\lineheight{1.25}\smash{\begin{tabular}[t]{l}$d$\end{tabular}}}}%
    \put(0,0){\includegraphics[width=\unitlength,page=6]{cobordism_a.pdf}}%
    \put(0.79914429,0.59395888){\makebox(0,0)[lt]{\lineheight{1.25}\smash{\begin{tabular}[t]{l}$P$\end{tabular}}}}%
    \put(0,0){\includegraphics[width=\unitlength,page=7]{cobordism_a.pdf}}%
    \put(0.44075656,0.45680185){\makebox(0,0)[lt]{\lineheight{1.25}\smash{\begin{tabular}[t]{l}RV move\end{tabular}}}}%
    \put(0,0){\includegraphics[width=\unitlength,page=8]{cobordism_a.pdf}}%
    \put(-0.00536973,0.08137141){\makebox(0,0)[lt]{\lineheight{1.25}\smash{\begin{tabular}[t]{l}$=$\end{tabular}}}}%
    \put(0,0){\includegraphics[width=\unitlength,page=9]{cobordism_a.pdf}}%
    \put(0.49309045,0.08393499){\makebox(0,0)[lt]{\lineheight{1.25}\smash{\begin{tabular}[t]{l}$=$\end{tabular}}}}%
    \put(0,0){\includegraphics[width=\unitlength,page=10]{cobordism_a.pdf}}%
    \put(0.59814501,0.19614468){\makebox(0,0)[lt]{\lineheight{1.25}\smash{\begin{tabular}[t]{l}$-2$\end{tabular}}}}%
    \put(0.91467093,0.18721422){\makebox(0,0)[lt]{\lineheight{1.25}\smash{\begin{tabular}[t]{l}$-2$\end{tabular}}}}%
    \put(0.75249953,0.15103253){\makebox(0,0)[lt]{\lineheight{1.25}\smash{\begin{tabular}[t]{l}$-1$\end{tabular}}}}%
    \put(0.10937963,0.12853368){\makebox(0,0)[lt]{\lineheight{1.25}\smash{\begin{tabular}[t]{l}$\mathbb{RP}^3$   \end{tabular}}}}%
    \put(0.35287366,0.13344085){\makebox(0,0)[lt]{\lineheight{1.25}\smash{\begin{tabular}[t]{l}$\mathbb{RP}^3$   \end{tabular}}}}%
  \end{picture}%
\endgroup%

%% file: cobordism_b.pdf_tex
\begingroup%
  \makeatletter%
  \providecommand\color[2][]{%
    \errmessage{(Inkscape) Color is used for the text in Inkscape, but the package 'color.sty' is not loaded}%
    \renewcommand\color[2][]{}%
  }%
  \providecommand\transparent[1]{%
    \errmessage{(Inkscape) Transparency is used (non-zero) for the text in Inkscape, but the package 'transparent.sty' is not loaded}%
    \renewcommand\transparent[1]{}%
  }%
  \providecommand\rotatebox[2]{#2}%
  \newcommand*\fsize{\dimexpr\f@size pt\relax}%
  \newcommand*\lineheight[1]{\fontsize{\fsize}{#1\fsize}\selectfont}%
  \ifx\svgwidth\undefined%
    \setlength{\unitlength}{338.57843001bp}%
    \ifx\svgscale\undefined%
      \relax%
    \else%
      \setlength{\unitlength}{\unitlength * \real{\svgscale}}%
    \fi%
  \else%
    \setlength{\unitlength}{\svgwidth}%
  \fi%
  \global\let\svgwidth\undefined%
  \global\let\svgscale\undefined%
  \makeatother%
  \begin{picture}(1,0.39445851)%
    \lineheight{1}%
    \setlength\tabcolsep{0pt}%
    \put(0,0){\includegraphics[width=\unitlength,page=1]{cobordism_b.pdf}}%
    \put(0.0969494,0.24958705){\makebox(0,0)[lt]{\lineheight{1.25}\smash{\begin{tabular}[t]{l}$a$\end{tabular}}}}%
    \put(0,0){\includegraphics[width=\unitlength,page=2]{cobordism_b.pdf}}%
    \put(0.09943582,0.11991667){\makebox(0,0)[lt]{\lineheight{1.25}\smash{\begin{tabular}[t]{l}$b$\end{tabular}}}}%
    \put(0.280664,0.2476381){\makebox(0,0)[lt]{\lineheight{1.25}\smash{\begin{tabular}[t]{l}$c$\end{tabular}}}}%
    \put(0.27839347,0.13169352){\makebox(0,0)[lt]{\lineheight{1.25}\smash{\begin{tabular}[t]{l}$d$\end{tabular}}}}%
    \put(0,0){\includegraphics[width=\unitlength,page=3]{cobordism_b.pdf}}%
    \put(0.20082877,0.33995418){\makebox(0,0)[lt]{\lineheight{1.25}\smash{\begin{tabular}[t]{l}$P$\end{tabular}}}}%
    \put(0,0){\includegraphics[width=\unitlength,page=4]{cobordism_b.pdf}}%
    \put(0.66374974,0.2871034){\makebox(0,0)[lt]{\lineheight{1.25}\smash{\begin{tabular}[t]{l}$-2$\end{tabular}}}}%
    \put(0.93701272,0.27939352){\makebox(0,0)[lt]{\lineheight{1.25}\smash{\begin{tabular}[t]{l}$-2$\end{tabular}}}}%
    \put(0.79530117,0.24517298){\makebox(0,0)[lt]{\lineheight{1.25}\smash{\begin{tabular}[t]{l}$-1$\end{tabular}}}}%
    \put(0,0){\includegraphics[width=\unitlength,page=5]{cobordism_b.pdf}}%
    \put(0.86773679,0.10256987){\makebox(0,0)[lt]{\lineheight{1.25}\smash{\begin{tabular}[t]{l}$0$\end{tabular}}}}%
    \put(0.48542087,0.17722514){\makebox(0,0)[lt]{\lineheight{1.25}\smash{\begin{tabular}[t]{l}$=$\end{tabular}}}}%
  \end{picture}%
\endgroup%

%% file: framing_of_cobordism_b.pdf_tex
\begingroup%
  \makeatletter%
  \providecommand\color[2][]{%
    \errmessage{(Inkscape) Color is used for the text in Inkscape, but the package 'color.sty' is not loaded}%
    \renewcommand\color[2][]{}%
  }%
  \providecommand\transparent[1]{%
    \errmessage{(Inkscape) Transparency is used (non-zero) for the text in Inkscape, but the package 'transparent.sty' is not loaded}%
    \renewcommand\transparent[1]{}%
  }%
  \providecommand\rotatebox[2]{#2}%
  \newcommand*\fsize{\dimexpr\f@size pt\relax}%
  \newcommand*\lineheight[1]{\fontsize{\fsize}{#1\fsize}\selectfont}%
  \ifx\svgwidth\undefined%
    \setlength{\unitlength}{353.36612885bp}%
    \ifx\svgscale\undefined%
      \relax%
    \else%
      \setlength{\unitlength}{\unitlength * \real{\svgscale}}%
    \fi%
  \else%
    \setlength{\unitlength}{\svgwidth}%
  \fi%
  \global\let\svgwidth\undefined%
  \global\let\svgscale\undefined%
  \makeatother%
  \begin{picture}(1,1.38842881)%
    \lineheight{1}%
    \setlength\tabcolsep{0pt}%
    \put(0,0){\includegraphics[width=\unitlength,page=1]{framing_of_cobordism_b.pdf}}%
    \put(0.14764574,1.09829591){\color[rgb]{0,0,0}\makebox(0,0)[lt]{\lineheight{1.25}\smash{\begin{tabular}[t]{l}$P$\end{tabular}}}}%
    \put(0,0){\includegraphics[width=\unitlength,page=2]{framing_of_cobordism_b.pdf}}%
    \put(0.49819527,1.32878214){\color[rgb]{0,0,0}\makebox(0,0)[lt]{\lineheight{1.25}\smash{\begin{tabular}[t]{l}$P$\end{tabular}}}}%
    \put(0,0){\includegraphics[width=\unitlength,page=3]{framing_of_cobordism_b.pdf}}%
    \put(0.50105468,0.91632361){\color[rgb]{0,0,0}\makebox(0,0)[lt]{\lineheight{1.25}\smash{\begin{tabular}[t]{l}$P$\end{tabular}}}}%
    \put(0,0){\includegraphics[width=\unitlength,page=4]{framing_of_cobordism_b.pdf}}%
    \put(0.83929076,1.09063642){\color[rgb]{0,0,0}\makebox(0,0)[lt]{\lineheight{1.25}\smash{\begin{tabular}[t]{l}$P$\end{tabular}}}}%
    \put(0,0){\includegraphics[width=\unitlength,page=5]{framing_of_cobordism_b.pdf}}%
    \put(0.69420133,0.32965948){\makebox(0,0)[lt]{\lineheight{1.25}\smash{\begin{tabular}[t]{l}$-2$\end{tabular}}}}%
    \put(0.88737551,0.32420922){\makebox(0,0)[lt]{\lineheight{1.25}\smash{\begin{tabular}[t]{l}$-2$\end{tabular}}}}%
    \put(0.7871971,0.30001812){\makebox(0,0)[lt]{\lineheight{1.25}\smash{\begin{tabular}[t]{l}$-1$\end{tabular}}}}%
    \put(0,0){\includegraphics[width=\unitlength,page=6]{framing_of_cobordism_b.pdf}}%
    \put(0.83840309,0.19920964){\makebox(0,0)[lt]{\lineheight{1.25}\smash{\begin{tabular}[t]{l}$0$\end{tabular}}}}%
    \put(0,0){\includegraphics[width=\unitlength,page=7]{framing_of_cobordism_b.pdf}}%
    \put(0.38783142,0.53111114){\makebox(0,0)[lt]{\lineheight{1.25}\smash{\begin{tabular}[t]{l}$-2$\end{tabular}}}}%
    \put(0.58100573,0.52566094){\makebox(0,0)[lt]{\lineheight{1.25}\smash{\begin{tabular}[t]{l}$-2$\end{tabular}}}}%
    \put(0.48082738,0.50146972){\makebox(0,0)[lt]{\lineheight{1.25}\smash{\begin{tabular}[t]{l}$-1$\end{tabular}}}}%
    \put(0,0){\includegraphics[width=\unitlength,page=8]{framing_of_cobordism_b.pdf}}%
    \put(0.12465382,0.32965948){\makebox(0,0)[lt]{\lineheight{1.25}\smash{\begin{tabular}[t]{l}$-2$\end{tabular}}}}%
    \put(0,0){\includegraphics[width=\unitlength,page=9]{framing_of_cobordism_b.pdf}}%
    \put(0.52471186,0.00031211){\makebox(0,0)[lt]{\lineheight{1.25}\smash{\begin{tabular}[t]{l}$0$\end{tabular}}}}%
    \put(0,0){\includegraphics[width=\unitlength,page=10]{framing_of_cobordism_b.pdf}}%
    \put(0.24107388,0.32878711){\makebox(0,0)[lt]{\lineheight{1.25}\smash{\begin{tabular}[t]{l}$-2$\end{tabular}}}}%
    \put(0,0){\includegraphics[width=\unitlength,page=11]{framing_of_cobordism_b.pdf}}%
    \put(0.40591299,0.16670942){\makebox(0,0)[lt]{\lineheight{1.25}\smash{\begin{tabular}[t]{l}$-2$\end{tabular}}}}%
    \put(0,0){\includegraphics[width=\unitlength,page=12]{framing_of_cobordism_b.pdf}}%
    \put(0.56673935,0.16301675){\makebox(0,0)[lt]{\lineheight{1.25}\smash{\begin{tabular}[t]{l}$-2$\end{tabular}}}}%
    \put(0,0){\includegraphics[width=\unitlength,page=13]{framing_of_cobordism_b.pdf}}%
    \put(-0.07516781,1.54839565){\makebox(0,0)[lt]{\lineheight{1.25}\smash{\begin{tabular}[t]{l}RP\end{tabular}}}}%
    \put(0.10638273,0.81228713){\makebox(0,0)[lt]{\lineheight{1.25}\smash{\begin{tabular}[t]{l}$\mathbb{RP}^3\#\mathbb{RP}^3$  \end{tabular}}}}%
    \put(0.80158502,0.81230632){\makebox(0,0)[lt]{\lineheight{1.25}\smash{\begin{tabular}[t]{l}$\mathbb{RP}^3\#\mathbb{RP}^3$ \end{tabular}}}}%
    \put(0.42122892,0.63931041){\makebox(0,0)[lt]{\lineheight{1.25}\smash{\begin{tabular}[t]{l}$\mathbb{RP}^3\#\mathbb{RP}^3\#(S^1\times S^2)$\end{tabular}}}}%
    \put(0.46331182,1.05078881){\makebox(0,0)[lt]{\lineheight{1.25}\smash{\begin{tabular}[t]{l} $S^1\times S^2$\end{tabular}}}}%
  \end{picture}%
\endgroup%